%% file: main.tex
\documentclass{article}
\usepackage[utf8]{inputenc}
\usepackage[T1]{fontenc}

\usepackage{amsmath}
\usepackage{amssymb}
\allowdisplaybreaks

\usepackage{graphicx}
\usepackage{xurl}
\usepackage[
natbib,
maxcitenames=6, 
maxbibnames=10,  
sorting=nyt,
sortcites,
defernumbers,
backend=biber,
backref,
doi=false,
url=false,
giveninits
]{biblatex}
\usepackage{xcolor}
\usepackage{mathtools}
\usepackage{bm}
\definecolor{cadmiumgreen}{rgb}{0.0, 0.42, 0.24}
\definecolor{hotpink}{rgb}{1.0, 0.41, 0.71}
\definecolor{cgreen}{RGB}{0, 180, 100}
\usepackage{hyperref}
\hypersetup{
    colorlinks,
    linkcolor=cgreen,
    citecolor=cgreen,
    urlcolor=blue,
    linktocpage
}

\usepackage{subcaption}
\usepackage{amsthm}
\usepackage{comment}
\usepackage[capitalise]{cleveref}
\crefformat{equation}{(#2#1#3)}
\usepackage[textsize=scriptsize]{todonotes}
\usepackage{csquotes}
\usepackage[shortlabels]{enumitem}
\usepackage{booktabs}

\usepackage[many]{tcolorbox}
\usepackage{cases}

\input{definitions}
\usepackage{authblk}
\usepackage[margin=1in]{geometry}

\newcommand{\ack}[1]{%
\section*{Acknowledgements}
#1
}

\numberwithin{equation}{section}

\newtheorem{thm}{Theorem}[section]
\newtheorem{prop}[thm]{Proposition}
\newtheorem{lemma}[thm]{Lemma}
\newtheorem{cor}[thm]{Corollary}
\newtheorem{rem}[thm]{Remark}
\crefname{rem}{Remark}{Remarks}

\crefname{thm}{Theorem}{Theorems}

\newcommand{\thistheoremname}{}
\newtheorem{genericthm}[thm]{\thistheoremname}

\newtheorem*{genericthm*}{\thistheoremname}
\newenvironment{namedthm*}[1]
  {\renewcommand{\thistheoremname}{#1}%
   \begin{genericthm*}}
  {\end{genericthm*}}

\theoremstyle{definition}
\newtheorem{example}[thm]{Example}
\newtheorem{definition}[thm]{Definition}

\newtheorem{assumption}{Assumption}

\usepackage{pgfplots}
\usepackage[]{tikz}
\pgfplotsset{compat=1.9}
\usetikzlibrary{arrows.meta}
\usetikzlibrary{patterns}
\usetikzlibrary{matrix}

\tcbset{
    sharp corners,
    colback = white,
    before skip = 0.2cm,    
    after skip = 0.5cm      
}        

\newtcolorbox{boxA}{
    boxrule = 1pt,
    colframe = black 
}

\makeatletter
\let\blx@rerun@biber\relax
\makeatother

\begin{document}

\title{Analysis of mean-field models arising from self-attention dynamics in transformer architectures with layer normalization }
\author[1,2,*]{Martin Burger}
\author[1]{Samira Kabri}
\author[3]{Yury Korolev}
\author[1]{Tim Roith}
\author[1]{Lukas Weigand}

\affil[1]{Helmholtz Imaging, Deutsches Elektronen-Synchrotron DESY, Notkestr. 85, 22607 Hamburg, Germany}
\affil[2]{Fachbereich Mathematik, Universit\"at Hamburg, Bundesstr. 55, 20146 Hamburg, Germany}
\affil[3]{Department of Mathematical Sciences, University of Bath, Bath BA2 7AY, UK}

\affil[*]{Corresponding author: martin.burger@desy.de}



\maketitle

\begin{abstract}
The aim of this paper is to provide a mathematical analysis of 
transformer architectures using a self-attention mechanism with layer normalization. In particular, observed patterns in such architectures resembling either clusters or  uniform distributions pose a number of challenging mathematical questions. We focus on a special case that admits a gradient flow formulation in the spaces of probability measures on the unit sphere under a special metric, which allows us to give at least partial answers in a rigorous way. The arising mathematical problems resemble those recently studied in aggregation equations, but with additional challenges emerging from restricting the dynamics to the sphere and the particular form of the interaction energy.

We provide a rigorous framework for studying the gradient flow, which also suggests a possible metric geometry to study the general case (i.e. one that is not described by a gradient flow). We further analyze the stationary points of the induced self-attention dynamics. The latter are related to stationary points of the interaction energy in the Wasserstein geometry, and we further discuss energy minimizers and maximizers in different parameter settings.\vspace{1em}

\noindent\textbf{Keywords:} Transformer architectures, self-attention dynamics, gradient flows, interaction energies, stationary states
\end{abstract}

\input{introduction}
\input{gradientflow}
\input{minmax}
\input{stationary}
\input{numerics}
\input{conclusion}

\input{acks}

\sloppy
\printbibliography

\appendix
\input{AGradFlow}
\input{Asphericalcoords}
\input{AStationary}

\end{document}

%% file: definitions.tex
\newcommand{\defeq}{:=}

\newcommand{\abs}[1]{\left|#1\right|}

\newcommand{\eps}{\varepsilon}

\renewcommand{\leq}{\leqslant}
\renewcommand{\geq}{\geqslant}
\renewcommand{\phi}{\varphi}

\let\oldsum\sum
\renewcommand{\sum}{\oldsum\nolimits}

\newcommand{\ergp}{{\mathcal{E}}}
\newcommand{\erg}{{\ergp_D}}
\newcommand{\Velo}{\mathsf{V}}

\DeclareMathOperator*{\argmax}{arg\,max}

\renewcommand{\d}{\,\mathrm{d}}

\let\div\relax
\DeclareMathOperator{\div}{div}
\newcommand{\grad}[1]{\nabla #1}

\newcommand{\uS}{\mathcal{S}}

\newcommand{\id}{{\operatorname{Id}}}

\newcommand{\R}{\mathbb{R}}
\let\S\relax
\newcommand{\S}{{\cal S}}

\definecolor{cornflowerblue}{rgb}{0.39, 0.58, 0.93}
\definecolor{sky}{RGB}{117, 187, 253}
\newcommand{\sk}[1]{{
		#1}}

\newcommand{\dpr}[2]{#1\cdot #2}
\newcommand{\Id}{\id}

%% file: introduction.tex
\section{Introduction}

Transformer architectures and the associated (self-)attention dynamics gained strong interest recently due to the success of artificial intelligence relying on them in several applications. Examples include large language models such as  GPT-4 \cite{achiam2023gpt}, multimodal large language models such as Vision Language Transformers~\cite{wu2023multimodal, fields2023vision}, text-to-image generation like Stable Diffusion \cite{esser2024scaling}, and protein folding with AlphaFold \cite{abramson2024accurate, jumper2021highly}, which won the Nobel prize in Chemistry in 2024.

The practical success of transformers and (self-)attention dynamics calls for developing detailed mathematical understanding which started recently in, e.g. \cite{vuckovic2020mathematical,sander2022sinkformers,geshkovski2023mathematical,calvello2024continuum,nguyen2024primal,wright2021transformers,criscitiello2024synchronization,alcalde2024clustering,geshkovski2024measure,kan2025ot,viswanathan2025geometry,abella2024asymptotic,alcalde2025exact}.

An interesting viewpoint on such dynamics is to interpret it as an interacting particle system \cite{lu2020understanding,dutta2021redesigning,sander2022sinkformers}, which allows for natural continuous-time and mean-field limits. The latter approach already provided valuable insights into feed-forward neural networks and their training dynamics (cf. \cite{chizat2018global,ding2021global}). In the context of transformers, this viewpoint also provides interesting (so far formal~\cite{geshkovski2023mathematical}) connections to gradient flows and the minimization of an interaction energy for the particle measures. The latter is a topic of great recent interest due to various applications in biology and social interactions. Indeed, the self-attention dynamics in transformers shares certain mathematical similarities with models used in opinion formation, which also exhibit similar emergence of clusters in certain cases \cite{hegselmann2002opinion,gomez2012bounded,piccoli2021generalized}. \color{black} In this work we focus on cluster formation in the infinite time horizon. However, we note that the formation of metastable states is of special interest. For the case of isotropic interaction, metastability was studied in \cite{bruno2024emergence,geshkovski2024dynamic}. \color{black}

 
In this paper we proceed upon the work in \cite{geshkovski2023mathematical} on analyzing transformer dynamics with layer normalization, focusing in particular on the case when the underlying dynamics has a gradient flow structure.  Indeed, the continuum limit of the self-attention dynamics leads to a Wasserstein-type gradient flow for probability measures on the unit sphere ${\cal S}$ of the form 
\begin{equation}
    \partial_t \mu_t = \nabla_{\cal S} \cdot (\mu_t m_{\mu_t} \nabla_{\cal S} {\cal E}'(\mu_t)),
\end{equation}
where $\nabla_{\cal S}$ and $\nabla_{\cal S}\cdot$ are the tangential gradient and divergence, respectively, and $m_\mu=\frac{1}{{\cal E}'(\mu)}$ is a non-local mobility.
The underlying energy in this case is of the form 
\begin{equation}\label{eq:energy}
    {\cal E}(\mu) =  \int_{\cal S} \int_{\cal S} e^{\dpr{x}{Dy}} \d \mu(x) \d \mu(y) ,
\end{equation}
with $D$ being a symmetric matrix and $\mathcal{E}'$ denotes its first variation. Since $D$ is symmetric and hence diagonalizable, we can equivalently assume that $D$ is a diagonal matrix, since we can use an orthogonal diagonalization and a corresponding transfer of variables to the eigenvectors, which leaves the unit ball unchanged. This will be used in several instances to simplify notation.   It also permits a more detailed study of stationary patterns, in particular minimizers and maximizers of the energy.

Compared to the existing literature on such gradient flows there are three distinct features that motivate our study, namely
\begin{itemize}
    \item restriction of the dynamics to the unit sphere (a consequence of the layer normalization);

    \item  non-local mobility (a consequence of the self-attention mechanism), which is related to but still distinctly different from other variations of Wasserstein gradient flows  studied recently (cf. \cite{burger2023covariance,duncan2023geometry,li2021hessian,lisini2012cahn});

    \item multiplicative coupling of states in the interaction energy, as opposed to commonly used interaction potentials depending only on the difference of the states (cf., e.g.~\cite{burger2008large,canizo2024discrete,carrillo2014global,carrillo2017geometry,shu2024wasserstein,simione2015existence}).
\end{itemize}

We make the gradient flow formally 
studied in \cite{geshkovski2023mathematical} rigorous, showing that  the transport distance with non-local mobilities is well defined, studying energy dissipation properties of the associated gradient flow, and describing the large-time behavior of the dynamics, specifically the convergence to stationary solutions, at least along subsequences. We further carry out a detailed study of energy minimizers and maximizers of ${\cal E}$ (extending the previously studied case of $D$ being a multiple of the identity) as well as stationary points of the energy in a Wasserstein setting, which we prove to be equivalent to stationary solutions of the dynamics. For the energy minimizers we obtain an interesting picture depending on the structure of $D$:
\begin{itemize}
    \item If there is a positive eigenvalue that is the eigenvalue of maximal absolute value, then a Dirac delta concentrated in the direction of a corresponding eigenvalue is a maximizer.

    \item If the smallest eigenvalue is negative, then only a Dirac delta concentrated in the direction of a corresponding eigenvalue is a minimizer.

    \item If the smallest eigenvalue is zero, then any measure concentrated on the nullspace of $D$ is a minimizer.

    \item Dirac deltas concentrated in directions of arbitrary eigenvectors are stationary points. We also find some convex combinations of Dirac deltas being stationary points.

    \item If the smallest eigenvalue is positive, we conjecture that the minimizer of the energy has full support on the unit sphere. To obtain some insight, we carry out a second-order asymptotic analysis of the minimizers for $D$ being a small perturbation of the identity.
\end{itemize}
We support our theoretical findings by several computational experiments and  investigate the cases when the energy minimizers or maximizers cannot be characterized explicitly. 

The rest of this work is organized as follows. In the remainder of the introduction, we recapitulate the simplified softmax transformer model introduced in \cite{sander2022sinkformers}, with additional layer normalization as considered in \cite{geshkovski2023mathematical}. In \cref{sec:gradientflow}, we provide a rigorous derivation of the gradient flow induced by the considered model. \cref{sec:energy-min-max,sec:stationary} are dedicated to characterizing optimizers, or, respectively, stationary points of the studied energy. We support our findings by numerical experiments in \cref{sec:numerics} and summarize our results in \cref{sec:conclusion}.

%
%



\subsection{Self-Attention}
%
%
    Transformer architectures \cite{vaswani2017attention} were  developed in the field of Natural Language Processing. Here an input is usually a sentence, which is decomposed into a sequence of tokens (e.g. words or syllables). Each token (possibly along with its position in the sentence) is represented as a vector in a high-dimensional vector space. Apart from a conventional feed-forward component, the main feature of a transformer layer is the so-called attention mechanism.  This mechanism implements the interaction between tokens and was first introduced in \cite{bahdanau2014neural} in the context of Neural Machine Translation as an alternative to encoder-decoder approaches whose performance often deteriorates for large input lengths due to the use of latent representations of fixed dimensions.\looseness=-1

Like \cite{geshkovski2023mathematical}, we will focus on a simple yet widely used form of attention, the so-called self-attention. It can be formalized as follows: consider an input sequence $X = [X_i]_{i = 1}^N \in \R^{N \times n}$, where each $X_i \in \R^{n}$ represents an $n$-dimensional token and $N$ denotes the number of tokens. The self-attention matrix $A \in \R^{N\times N}$ is given by 
\begin{align}\label{eq:attention_matrix}
    A_{ij} = \frac{\exp(X_i \cdot DX_j)}{\sum_{k = 1}^N \exp(X_i \cdot DX_k)},
\end{align}
where we assume $D \in \R^{n\times n}$ to be symmetric. The latter property does not necessarily hold  for learned parameters in transformer architectures, but we expect the symmetric part to determine the asymptotic behavior of the self-attention dynamics. Since the symmetry of $D$ allows one to interpret the dynamics as a gradient flow corresponding to a certain interaction energy, as observed in \cite{geshkovski2023mathematical}, it will allow us to analyze the asymptotic behavior for this subclass; the study of the general case is left for future research. \color{black}An important example for non-symmetric interaction is given by masked attenion, which can be used to model causality. We refer to \cite{castin2024howsmooth,castin2025unified,karagodin2024clustering} for a mean-field interpretation of such dynamics. \color{black}

By definition, the matrix $A$ is stochastic, i.e. each of its rows  is a probability vector. Roughly speaking, the attention matrix determines how strong a token is influenced by each other token. To determine \textit{how} tokens influence each other, another matrix $V \in  \R^{n\times n}$, called the value matrix, is used. The influence of $X_j$ on $X_i$ can then be written as $A_{ij}\,VX_j$ and the self-attention layer $\mathcal{A}: \R^{N \times n} \rightarrow \R^{N \times n}$ is given by
\begin{align}\label{eq:selfattention_discrete}
    \mathcal{A}(X) = \left[X_i + \sum_{j = 1}^{N} A_{ij}\, VX_j\right]_{i  = 1}^N.
\end{align}
For our purposes, we assume $V = D$ or $V = -D$ since in this case one can show that the particles move along a gradient flow. The general case is subject of future work.

\subsection{Normalization method}
The normalization of intermediate values is a common practice in  machine learning models. In the context of neural networks, so-called batch normalization \cite{ioffe15batch} is a popular method to prevent gradients from blowing up and thus to stabilize (and to improve) the training. Since this form of normalization uses information from the entire training batch, \cite{ba2016layernormalization} proposes layer normalization (LayerNorm), which translates the mean of an intermediate vector to zero and divides it by its standard deviation, and therefore does not depend on any other vector in the batch. While the original implementation of the transformer \cite{vaswani2017attention} uses LayerNorm, some of the more recent publications (e.g. Llama, \cite{touvron2023llamaopenefficientfoundation}) use a simplified version called Root Mean Square Layer Normalization (RMSNorm) proposed in \cite{zhang2019root}. Up to a multiplication with learned weights $[g_i]_{i = 1}^n$, called gain parameters, RMSNorm performs a projection onto the unit sphere $\mathcal{S}^{n-1}$ (where in the following we will suppress the superscript and simply write ${\cal S}$). More precisely, for $x \in \R^{n}$ we write
\begin{align*}
    \operatorname{RMSNorm}(x)_i = g_i\frac{x_i}{\|x\|_2},
\end{align*}
where in practice a division by zero is circumvented by adding a small value $\epsilon>0$ in to $\|x\|_2$. In our setting, we can assume the norm to be strictly positive as we consider the dynamics in continuous time. Following the setting of \cite{geshkovski2023mathematical}, we focus on RMSNorm with fixed gain parameters $g_i = 1$ for all $i = 1,\dots n$ and denote the projection onto the unit sphere for $x \in \R^n \setminus \{0\}$  by
\begin{align*}
    \Pi(x) = \frac{x}{\|x\|_2}.
\end{align*}

\subsection{Simplified transformer layer and time-continuous dynamics}
Combining the attention layer with a normalization layer, we arrive at the following update step
\begin{align*}
    X \leftarrow \Pi(\mathcal{A}(X)),
\end{align*}
where the projection is applied vector-wise to each row of $\mathcal{A}(X)$. For the sake of our analysis, we will deviate from typical practical implementations of transformers and  consider the architecture to be a composition of such layers which all share the same matrices $D$ and $V$ in \eqref{eq:attention_matrix} and \eqref{eq:selfattention_discrete}. In \cite{geshkovski2023mathematical}, it was  proposed to study the continuum limit of these updates. This approach has become a popular tool for analyzing Residual Neural Networks (ResNets) \cite{he2016resnets}: as discussed from various perspectives, e.g. in \cite{e2017dynamical, haber2017stable, chen2018neural, thorpe2023deep}, the skip connections (i.e. the residual components) of the ResNet architecture make it possible to interpret it as a forward Euler discretization of an ordinary differential equation. Introducing a time variable $t > 0$ and a small time increment $\Delta t > 0$, we get
\begin{align}\label{eq:expEuler}
    X_i(t + \Delta t) = \Pi\left(X_i(t) + \Delta t\,\sum_{j = 1}^{N} A_{ij}(t)\, VX_j(t)\right), \quad i = 1,\hdots, N.
\end{align}
At this point, the residual component is hidden in the attention layer and cannot easily be extracted since the projection is nonlinear. In the continuous time limit $\Delta t \rightarrow 0$, remembering that $\Pi(x) = x$ for any $x \in \mathcal{S}$, we arrive at the following system of differential equations
\begin{align}\label{eq:timederivativeproj}
    \dot{X_i}(t) = \left\langle \nabla_x \Pi(X_i(t)),  \sum_{j = 1}^{N} A_{ij}(t)\, VX_j(t)\right\rangle,
    \quad i = 1,\hdots, N,
\end{align}
where the spatial derivatives are understood as derivatives in $\R^n$.
With a simple computation one can further show that for any $x \in \mathcal{S}$ and $z \in \R^{n}$ it holds that
\begin{align*}
    \langle \nabla_x \Pi(x) , z \rangle = P_x^\perp(z), 
\end{align*}
where, following \cite{geshkovski2023mathematical}, we define $P_x^\perp(z) = z - \dpr{x}{z} \,x$. Substituting this into \eqref{eq:timederivativeproj}, we arrive at the following dynamics
\begin{subequations}
\label{eq:dynamics-particles}
    \begin{numcases}
        {}\dot{X_i}(t) = P_{X_i(t)}^\perp\left(\sum_{j = 1}^{N} A_{ij}(t)\, VX_j(t)\right), \label{eq:timederivativeorth}\\
        X_i(0) = X_{0,i} \in \mathcal{S}
    \end{numcases}
\end{subequations}
which serve as a starting point of \cite{geshkovski2023mathematical}.

\subsection{Interpretation as an evolution of measures}\label{sc:MeasInt}
Instead of studying the dynamics of distinct particles, \cite{geshkovski2023mathematical} propose to view~\eqref{eq:dynamics-particles} as an evolution of an empirical measure
\begin{align*}
    \mu_t = \frac{1}{N}\sum_{i = 1}^{N} \delta_{X_i(t)}.
\end{align*}
The right-hand side of \eqref{eq:timederivativeorth} can be understood as an integral with respect to $\mu_t$; for a generic probability measure $\mu$ this can be written as a measure-dependent velocity field
\begin{align} \label{eq:velocityfield}
    \mathsf{V}[\mu](x) = \frac{P^{\perp}_x\left(\int_{\mathcal{S}}e^{\dpr{x}{Dy}}\,Vy \,\d \mu(y)\right)}{\int_{\mathcal{S}}e^{\dpr{x}{Dy}}\,\d \mu(y)},
\end{align}
and \eqref{eq:timederivativeorth} turns into $\dot{X_i}(t) = \mathsf{V}[\mu_t](X_i(t))$. With this notion, we recover the weak continuity equation formulated in \cite{geshkovski2023mathematical}: for any test function $\phi \in C^1(\mathcal{S} \times [0,T])$ one has
\begin{align}\label{eq:conteq_empirical}
    \frac{d}{dt} \int_{\mathcal{S}} \phi(t,x) \,\d \mu_t(x) &= \frac{d}{dt} \frac{1}{N} \sum_{i = 1}^N  \phi(t,X_i(t)) \nonumber \\ &= \frac{1}{N} \sum_{i = 1}^N \partial_t \phi(t, X_i(t)) + \big\langle \nabla_x \phi(t,X_i(t)), \mathsf{V}[\mu_t](X_i(t))\big\rangle \\
    & = \int_{\mathcal{S}} \partial_t \phi(t,x) + \big\langle \nabla_x \phi(t,x), \mathsf{V}[\mu_t](x)\big\rangle \,\d \mu_t(x), \nonumber
\end{align}
where, in this case, the spatial derivatives of $\phi$ have to be understood as derivatives on $\mathcal{S}$.

Similarly, \cite{geshkovski2023mathematical} propose the interaction energy \eqref{eq:energy},
which for an empirical measure $\mu_t$
reduces to 
\begin{align*}
    \mathcal{E}(\mu_t) = \sum_{i,j = 1}^{N} e^{\dpr{X_i(t)}{DX_j(t)}}.
\end{align*}
In this discrete case, a straightforward application of the chain rule and a reordering of the terms yields
\begin{align*}
    \frac{d}{dt} \mathcal{E}(\mu_t) = 2\sum_{i= 1}^N \dpr{\left( \sum_{j = 1}^N e^{\dpr{X_i(t)}{DX_j(t)}} DX_j(t) \right)}{\dot{X_i}(t)}.
\end{align*}
Under our assumption that the value matrix is given by $V = \pm D$, we see that, up to an application of $P^{\perp}_{X_i(t)}$ and a division by $\sum_{j = 1}^N e^{\dpr{X_i(t)}{DX_j(t)}}$, the term in the brackets is given by $\dot{X_i}(t)$. Since $P_x^\perp(z) \cdot z = P_x^\perp(z) \cdot P_x^\perp(z)$ for any $x \in \mathcal{S}$, $z \in \R^n$, we have that
\begin{align*}
    \frac{d}{dt} {\cal E}(\mu_t) = \pm 2 \sum_{i= 1}^N \|\dot{X}_i(t)\|^2\,  \sum_{j = 1}^N e^{\dpr{X_i(t)}{DX_j(t)}} \substack{\geq\\\leq} 0,
\end{align*}
and hence the energy $\mathcal{E}$ increases ($V = D$) or decreases ($V = -D$) monotonously along the trajectory of $\mu_t$.
A formal derivation of the above formulae for general probability measures on smooth manifolds is provided in \cref{sec:gradientflow}. 


Let us mention that problems with similar energies as ${\cal E}$ have been studied in the past. 
The most prominent is an interaction energy with respect to  $D$ with a nonlocal interaction kernel depending on $x-y$. Choosing the kernel as Gaussian with covariance matrix $D^{-1}$ (which makes sense only if $D$ is positive definite) results into
\begin{align}\label{eq:inter}
{\cal E}^{\text{inter}}(\mu) = \int_{\cal S} \int_{\cal S} e^{-\frac{1}{2} (x-y) \cdot (D(x-y))} \d \mu(x) \d \mu(y).
\end{align}
For $D = \pm \operatorname{Id}$ the minimizers and maximizers of \eqref{eq:energy} and \eqref{eq:inter} are equivalent as $\mp\frac{1}{2} (x-y) \cdot (x-y) = \mp\frac{1}{2}(x \cdot x +  y \cdot y ) \pm x\cdot y = \mp 1 + \pm x\cdot y $ for all $x,y \in {\cal S}$. The important difference between \eqref{eq:energy} and \eqref{eq:inter} is the rotation-(in)variance of the interaction functions $e^{ x \cdot (Dy)}$ and $e^{-\frac{1}{2} (x-y) \cdot (D(x-y))}$. In the general case this is not true, but we shall use an analogy to the interaction energy to rewrite
$$ {\cal E}(\mu)=e^\lambda \int_{\cal S} \int_{\cal S} e^{-\frac{\lambda}2 |x|^2 -\frac{\lambda}2 |y|^2 + x \cdot Dy} \d \mu(x) \d \mu(y) = e^\lambda \int_{\cal S} \int_{\cal S} e^{-\frac{\lambda}2 |x-y|^2 + x \cdot ((D-\lambda \Id)y)} \d \mu(x) \d \mu(y) . $$

\subsection{Understanding $x \cdot Dy$ on the sphere}

For our further analysis it is crucial to understand the implications of restricting the problem to the unit sphere and the behavior of the bilinear form $x \cdot Dy$ on it.
For $D = \id$ it is clear that the minimizer of $f_y(x) = x\cdot Dy$ is given by $x = -y$ and the maximizer by $x = y$. This changes for a general $D$ and as a result the minimizer of the energy \eqref{eq:energy} is not given by the uniform distribution on $\S$ anymore.
For a diagonal matrix $D$ the maximizer/minimizer of $f_y$ for a fixed $y \in \S$ with $Dy \neq 0$ is given by $x_{\pm}  = \pm \frac{Dy}{\|Dy\|}$. Therefore, we know that $x \cdot Dy = 0$ if and only if $x \cdot x_\pm = 0$ (same for $>$ and $<$).
For $Dy = 0$, we already have $f_y(x) = 0$ for any $x \in \S$, i.e. each point is a minimizer, maximizer and orthogonal to $y$ w.r.t. $D$. 
A further consequence is that 
\begin{align*}
    \max_{x,y \in \cal\S} x\cdot Dy = \max_{y \in \cal\S} \frac{Dy\cdot Dy}{\|Dy\|} = \max_{y \in \cal\S} \|Dy\| = |\lambda|, 
\end{align*}  
where $\lambda$ denotes the eigenvalue of maximum absolute value of $D$.
\color{black}
We further note that all of the following results on minimizers/maximizers as well as stationary points of $\erg$ can be generalized to probability measures  concentrated on an ellipsoid instead of a sphere. To see this, we consider the ellipsoid
\begin{align*}
    C\mathcal{S} = \left\{x \in \R^n:\, \|C^{-1}x\| = 1\right\},
\end{align*} where $C\in \R^{n\times n}$ is invertible,
and the corresponding energy 
\begin{align*}
    \mathcal{E}_D^C(\mu) = \int_{C\mathcal{S}}\int_{C\mathcal{S}} e^{\dpr{x}{Dy}}\,\d\mu(x)\,\d\mu(y).
\end{align*}
Since $C$ is invertible, any measure $\mu$ is uniquely determined by the pushforward measure $\nu = C^{-1}_\#\mu$, as $\mu = C_\#\nu$. Thus, we can rewrite the energy as
\begin{align*}
     \mathcal{E}_D^C(\mu) = \int_{\mathcal{S}}\int_{\mathcal{S}} e^{\dpr{Cx}{DCy}}\,\d\nu(x)\,\d\nu(y) = \mathcal{E}_{C^TDC}(\nu),
\end{align*}
and equivalently optimize the energy $\mathcal{E}_{C^TDC}$ on the sphere. 
A special case that leads to measures  concentrated on an ellipsoid corresponds to RMSNorm normalization with non-vanishing gain parameters $g_i \neq 0$. In this case, the ellipsoid is given by $G\mathcal{S}$, where $G$ is a diagonal matrix with entries $[g_i]_{i = 1}^n$.
\color{black}

%% file: gradientflow.tex
\section{Gradient flow}\label{sec:gradientflow}
As shown 
above, the particle dynamics can be \enquote{lifted} by  the use of empirical measures to the space of probability measures $\mathcal{P}(\mathcal{S})$ over the sphere. As mentioned in \cite[Remark 3.3]{geshkovski2023mathematical} for arbitrary probability measures the connection between the partial dynamics and a corresponding continuity equation can be made by a mean field limit approach.
Hence, instead of the particle dynamics one can study the continuity equation
\begin{align}
    \partial_t \mu+\div( \mathsf{V}[\mu] \mu)=0 &\quad \text{on } [0,T]\times \mathcal{S} ,  \label{eq:continuityequation} \\
    \mu_{|t=0}=\mu(0)& \quad \text{on }\mathcal{S} ,  \nonumber
\end{align}
with the velocity field given by \eqref{eq:velocityfield},
which holds in the sense of distributions. Note, that in this section we scale the energy by a factor of $1/2$ to be consistent with \cite{geshkovski2023mathematical}. It was remarked in \cite[Chapter  3.3]{geshkovski2023mathematical} that for $V=\pm D$ the energy 
$$\mathcal{E}(\mu)=\pm \frac{1}{2} \int_{\mathcal{S}}\int_{\mathcal{S}} e^{\dpr{x}{Dy}}\,\d\mu(x)\,\d\mu(y)$$
is monotone along these dynamics and the PDE~\eqref{eq:continuityequation} can be interpreted as a gradient flow for a modified optimal transport distance.
However, as the authors of~\cite{geshkovski2023mathematical} acknowledge, there is a gap in the literature that prevents them from making this observation rigorous.

In this section we aim to close this gap. We show that $\mathcal{P}(\mathcal{S})$ equipped with this new distance is a geodesic space with properties similar to the classical $2$-Wasserstein space and prove that solutions of \eqref{eq:continuityequation} are curves of maximal slope of $\mathcal{E}$ with respect to this distance and thus satisfy the energy dissipation equality
\begin{align*}
    \frac{\d}{\d t} \mathcal{E}(\mu_t)=- \int_\mathcal{S} \int_\mathcal{S} e^{\dpr{x}{Dy}}\d\mu_t(y)\,  |\Velo [\mu_t](x)|^2\d\mu_t(x) \quad \text{for a.e. }t.
\end{align*}
Lastly, we study the long-time behavior of the dynamics and show that subsequences of the flow converge to stationary points of the energy $\mathcal{E}$.

Let us mention that the basic analysis of this section related to the novel transport distance can be generalized in a rather straight-forward way to the more general case of $D$ being non-symmetric and can thus provide the basis for a future analysis of the non-gradient flow case with  $V$ arbitrary and $D$ non-symmetric.

\subsection{Continuity equation on manifolds}
Let $M$ be a compact $n$-dimensional Riemannian manifold without a boundary, e.g. the sphere $\mathcal{S}\subset \R^n$. 
The tangent bundle $TM= \sqcup_{x\in M} T_xM$ is given by the disjoint union of all tangent spaces of all $x\in M$. We denote by $\mathcal{P}(M)$ the space of Borel probability measures on $M$, equipped with the standard narrow topology (e.g. \cite[Chapter  5.1]{ambrosio2008gradient}). The symbol  $\rightharpoonup$ is used to indicate convergence in this topology. 
Let $I=(0,T)$ be an open interval, $\mu: t \rightarrow \mu_t \in \mathcal{P}(M)$ a narrowly continuous curve and $V: (x,t)\in M\times I \mapsto v_t(x) \in T M$ a Borel velocity field such that $\int_0^T \int_M  |v_t(x)|\d\mu_t\d t<\infty$.
The continuity equation holds in the sense of distributions if
\begin{align}\label{eq:CE}
    \int_{(0,T)} \int_M \partial_t \varphi(x,t)+ \langle \mathcal{D} \varphi(x,t)  ,v_t(x)\rangle\d\mu_t\d t=0,\quad \forall \varphi\in C_c^1(M\times(0,T)).
\end{align}
Here $\mathcal{D}$ denotes the differential on the manifold $M$. Sometimes we will use $\mathcal{D}_x$ to clarify with respect to which variable the differential is taken.
We define the set of solutions to the continuity equation as follows
\begin{align*}
    CE(0,T)\coloneqq \left\{(\mu,v): \quad
    \begin{gathered}
        \mu:I \mapsto \mathcal{P}(M) \text{ is narrowly continuous}, \\
        \int_0^T \int_M  |v_t(x)|\d\mu_t\d t<\infty,\\
        (\mu,v) \text{ satisfy the continuity equation}
    \end{gathered}
    \right\}.
\end{align*}
Further, we define $CE(0,T; \nu\rightarrow \eta)$ as the subset $(\mu, v)$  such that $\mu_0=\nu$, $\mu_T=\eta$. For more details, we refer to \cref{app:ContEq}.

\subsection{Distance}
To interpret \cref{eq:continuityequation} as a gradient flow on $\mathcal{P}(M)$ we need to modify the well-known dynamic formulation of the $2$-Wasserstein distance \cite{benamou2000computational} and introduce the following mobility 
\begin{align*}
    m_\mu(x)={\int_{M} K(x,y)\d\mu(y)}.
\end{align*}
With this, the modified transport distance between $\mu_0,\mu_1\in \mathcal{P}(M)$ is defined as follows (see  \cite[Section 3.4.2]{geshkovski2023mathematical})
\begin{align}\label{eq:OpTrDi}
    W_{m,2}^2(\mu_0,\mu_1)=\inf\left\{ \int_0^1 \int_{M} m_{\mu_t}(x)|v_t(x)|^2 \d\mu_t(x)\d t: (\mu,v)\in CE(0,1;\mu_0\rightarrow\mu_1) \right\}.
\end{align}
  
For $K \equiv 1$ we recover the classical $2$-Wasserstein distance. The dynamics \eqref{eq:continuityequation} corresponds to the kernel $K(x,y)=e^{\dpr{x}{Dy}}$, but for the sake of generality we carry out the analysis for a more general class of kernels $K$.

\begin{assumption}\label{ass:kernel}
 The kernel $K(x,y) \in C(M\times M)$ is continuous and there exists a constant $C>0$ such that $K(x,y)\geq C$ for all $x,y\in M$.
\end{assumption}

\begin{rem}
    The assumption that $K$ is bounded from below is vital for our analysis and covers the cases of interest in this paper. Nonetheless, it would be interesting to see whether this assumption can be relaxed. For example, instead of a compact manifold $M$, we could consider $\mathbb{R}^d$ as the underlying space and take $K$ to be a Gaussian or a bounded confidence kernel $K(x,y)=\mathsf{1}_{\abs{x-y}\leq 1}$ as studied in \cite{deffuant2000mixing}.
\end{rem}
As the next theorem shows the infimum in \eqref{eq:OpTrDi} is actually  attained by some $(\mu,v)\in CE(0,1;\mu_0\rightarrow\mu_1)$.
The proof can be found in \cref{app:Geo}.

\begin{thm}[Existence of minimizers]\label{eq:Geo}
For every pair $\mu_0,\mu_1 \in \mathcal{P}(M)$ with $W_{m,2}(\mu_0,\mu_1)<+\infty$ there exists a couple $(\mu,v)\in CE(0,1)$ such that 
\begin{align*}
    W_{m,2}^2(\mu_0,\mu_1)=\int_0^1 \int_M m_{\mu_t}(x) |v_t(x)|^2\d\mu_t(x)\d t.
\end{align*}
Furthermore, such minimizers can be equivalently characterized as those of 
\begin{align}\label{eq:GeoRep}
    W_{m,2}(\mu_0,\mu_t)=\inf\left\{ \int_0^T  \left(\int_M m_{\mu_t}(x) |v_t(x)|^2\d\mu_t(x)\right)^{\frac12}\d t\ : \, (\mu,v)\in CE(0,T; \mu_0\rightarrow \mu_T)\right\}
\end{align}
\end{thm}
Using the theorem above it is easy to show that $W_{m,2}$ is a distance on $\mathcal{P}(M)$.
\begin{thm}\label{cor:Top}
The space $\mathcal{P}(M)$ equipped with $W_{m,2}$ is a complete metric space and its topology  is equivalent to the one induced by the $2$-Wasserstein distance which, since $M$ is compact, is equivalent to the topology of narrow convergence.
\end{thm}

\begin{proof}
First, we check that $W_{m,2}$ is a distance. Indeed, \emph{(i)}~symmetry follows from simply rescaling time by $\tilde{t}: t\in[0,T] \mapsto  T-t\in[0,T]$; \emph{(ii)}~definiteness: Since $m_{\mu_t}$ is bounded from below, $W_{m,2}(\mu,\nu)=0$ implies that $v_t=0$ for $\bm{\mu}$-a.e. $(p,t)\in M \times (0,T)$. Thus by  \cref{eq:BInt} $\mu=\nu$;  \emph{(iii)}~the triangle inequality follows from the characterization \cref{eq:GeoRep} and the gluing property from \cref{pro:Pro}.
To show the equivalence of the distances we observe that by \cref{ass:kernel}, $K(x,y)\geq C$ and since $M\times M$ is compact and $K(x,y)$ is continuous we can also find a $\tilde{C}$ such that $K(x,y)\leq \tilde{C}$. This implies that
\begin{align*}
    \frac{1}{C}W_2(\mu,\nu)\leq  W_{m,2}(\mu,\nu)\leq \tilde{C} W_2(\mu,\nu)<+\infty \quad \forall\mu,\nu\in\mathcal{P}(M)
\end{align*}
and the distances are equivalent. Since $(\mathcal{P}(M), W_{2})$ is complete, $(\mathcal{P}(M), W_{m,2})$ has to be complete as well.
\end{proof}

 Let us recall that in a general metric space $(X,d)$ a curve $\gamma: [0,T]\rightarrow X$ is called absolutely continuous if there exists a function  $m\in L^1(0,T)$ such that
\begin{align}\label{eq:AbsCont}
    d(\gamma_s,\gamma_r)\leq \int_s^r m(t)\d t  \quad \forall s,r\in [0,T] \text{ with }s\leq r.
\end{align}
For an absolutely continuous curve $\gamma(t)$ its metric derivative is defined by 
\begin{align*}
    |\dot{\gamma}|(t)\coloneqq \lim_{h\rightarrow 0} \frac{d(\gamma_{t+h},\gamma_t)}{h}
\end{align*}
and it exists for a.e. $t\in(0,T)$. It can be shown that $|\dot{\gamma}|$ is minimal in the sense that for all $m(t)$ satisfying \cref{eq:AbsCont}  it holds that $|\dot{\gamma}|(t)\leq m(t)$ for a.e. $t\in(0,T)$.
The next lemma, which is proven in \cref{app:AC}, characterizes absolutely continuous curves in $(\mathcal{P}(M),W_{m,2})$. 

\begin{lemma}\label{lm:AC}
Let $\mu_t$ be an absolutely continuous curve w.r.t. $W_{2,m}$. Then there exists a Borel velocity field $(v_t)_{t\in(0,T)}$ such that $(\mu,v)\in CE(0,T)$ and $$\left( \int_M m_{\mu_t}(x) |v_t(x)|^2\d\mu_t(x)\right)^{1/2}= |\dot{\mu}|(t)\quad \text{for a.e. }t\in(0,T).$$
Conversely, if  $(\mu,v)\in CE(0,T)$ and $\int_0^T\left( \int_M m_{\mu_t} |v_t|^2\d\mu_t\right)^{1/2}\d t<+\infty$ then $t\rightarrow\mu_t$ is absolutely continuous and 
$$ |\dot{\mu}|(t)\leq \left( \int_M m_{\mu_t}(x) |v_t(x)|^2\d\mu_t(x)\right)^{1/2}\quad \text{for a.e. }t\in(0,T).$$
\end{lemma}

A metric space is called a length space if $$d(x,y)=\inf \int_0^1 |\dot\gamma|(t)\d t,$$
where the infimum is taken over all absolutely continuous curves $\gamma:[0,1]\rightarrow X$ with $\gamma(0)=x$ and $\gamma(1)=y$. If this infimum is obtained by a minimal curve, also called geodesic, we say that $(X,d)$ is a geodesic space. As it turns out, the minimal curves obtained in \cref{eq:Geo} are such geodesics. This can be immediately be deduced from \cref{eq:ConstGeo} and the definition of the metric velocity,
\begin{cor}
    The space $(\mathcal{P}(M), W_{m,2})$ is a geodesic space. 
\end{cor}

\subsection{Gradient flows of the interaction energy}

Let  $W(x,y) \in C^1(M\times M)$ be a symmetric interaction kernel. 
The interaction energy is given by 
\begin{align*}
    \mathcal{E(\mu)}\coloneqq \frac{1}{2} \int_{M \times M} W(x,y)\d\mu(x)\d\mu(y).
\end{align*}
Let us consider the following inverse  duality map 
\begin{align*}
    \mathcal{J}_2: x\in TM_p^* \mapsto |x|_*  \argmax_{y\in TM_p: |y|= 1} x(y).
\end{align*}
Since all tangent spaces are finite-dimensional this map is well defined. The application of $\mathcal{J}_2$ to a  1-form on $M$ (in particular, a differential of a function) yields a velocity field on $M$. 
Below we show that gradient flows of the energy $\mathcal{E}$ with respect to the metric $W_{m,2}$ are given by weak solutions to PDEs of the form
\begin{align}\label{eq:GradFlowPDE}
    \partial_t \mu+\div\left( \frac{1}{ m_{\mu}} J_2(\mathcal{D} W[\mu])\mu\right)=0,
\end{align}
where $W[\mu](x)= \int_M W(x,y)\d\mu(y)$. For $M=\mathcal{S}$, $K(x,y)=e^{\dpr{x}{Dy}}$ and $W(x,y)=\pm e^{\dpr{x}{Dy}}$ equation \cref{eq:GradFlowPDE} corresponds precisely to \cref{eq:continuityequation} if $V=\pm D$.
The sole difference between \cref{eq:GradFlowPDE} and classical Wasserstein gradient flows is the presence of of the factor $\frac{1}{m_\mu}$. It arises since the modified transport distance punishes movement of particles with a high mobility $m_\mu(x)$. When we interpret $K(x,y)$ as an interaction kernel between particles, those particles interacting strongly with others are slowed down while particles with low interaction are sped up.

\begin{lemma}[Chain rule] \label{lm: ChainRule}
Let $t\rightarrow\mu_t$ be an absolutely continuous curve in $W_{2,m}$. Then $t\mapsto \mathcal{E}(\mu_t)$ is absolutely continuous and
   \begin{align}\label{eq:CheinRule}
    \frac{\d}{\d t} \mathcal{E}(\mu_t)= \int_M \langle \mathcal{D}W[\mu_t](x),v_t(x)\rangle\d\mu_t(x) \quad \text{for a.e}\quad t\in (0,T). 
\end{align} 
\end{lemma}
\begin{proof}
Let us consider an absolutely continuous curve $(\mu,v)\in CE(0,1; \mu\rightarrow \nu)$ and
the function $\eta: (x,t)\in M\times [0,T] \mapsto\frac{1}{2}\int_M W(x,y)\d\mu_t(y)$.
In the case when $\eta\in C^1(M\times [0,T])$ we could use it as a test function in \cref{eq:BInt} and immediately obtain
\begin{align*}
    \mathcal{E}(\mu_T)-\mathcal{E}(\mu_0)&=\int_M \eta(x,T)\d\mu_T(x)-\int_M \eta(x,0)\d\mu_0(x)\\
    &=\int_0^T\int_M \partial_t \eta(x,t)\d\mu_t(x)+\int_M \langle \mathcal{D} \eta(t,x),v_t(x)\rangle \d\mu_t(x)\d t\\
    &= \int_0^T\int_{M}\int_M \langle \mathcal{D}_x W(x,y),v_t(x)\rangle\d\mu_t(y)\d\mu_t(x)\d t<+\infty.   
\end{align*}
The finiteness follows from that fact that we can bound $|\mathcal{D}_{x} W(x,y)|_*$ uniformly on $M \times M$. 
In the general case we have to use a rather lengthy time mollification argument, see \cref{sc:ProofChain}. 
\end{proof}
Equation \cref{eq:CheinRule} is reminiscent of the classical chain rule $\frac{\d}{\d t} F(x(t))=\nabla F(x(t))\cdot \dot x(t)$ for a function $F:\R^d \rightarrow \R$ and a curve $x:[0,T] \rightarrow \R^d$.
The velocity field $v_t$ can be viewed as the \enquote{derivative} of the curve $\mu_t$, while $\mathcal{D}W[\mu_t]$ is the corresponding \enquote{gradient} of the interaction energy. Using this chain rule we can 
estimate how fast the energy can decrease along a curve~$\mu_t$. Therefore, curves reaching this bound dissipate the energy as fast as possible and satisfy the so-called energy dissipation equality.
\begin{lemma} \label{lm: DisE}
 For any absolutely continuous w.r.t. $W_{2,m}$ curve $(\mu_t)_{t\in (0,T)}$ we have that
 \begin{align}\label{eq:EnergyDis}
     \mathcal{E}(\mu_T)-\mathcal{E}(\mu_0)+ 
     \frac{1}{2}\int_0^T\int_M { m_{\mu_t}} |v_t|^2\d\mu_t\d t  +\frac{1}{2}\int_0^T\int_M  \frac{1}{m_{\mu_t}} |\mathcal{D} W[\mu_t]|_*^2\d\mu_t\d t \geq 0.
 \end{align}
 Moreover, we have equality if and only if $(\mu_t)_{t\in (0,T)}$ is a weak solution to \cref{eq:GradFlowPDE}.
\end{lemma}
\begin{proof}
We can estimate the right hand sight of \cref{eq:CheinRule} by Hölder's and Young's inequality 
\begin{align*}
    \int_M \mathcal{D} W[\mu_t](v_t)\d\mu_t\geq -\sqrt{ \int_M { m_{\mu_t}} |v_t|^2\d\mu_t }\sqrt{\int_M { \frac{1}{m_{\mu_t}}} |\mathcal{D} W[\mu_t]|_*^2\d\mu_t }\\
    \geq -\frac{1}{2}\int_M { m_{\mu_t}} |v_t|^2\d\mu_t  -\frac{1}{2}\int_M  \frac{1}{m_{\mu_t}} |\mathcal{D} W[\mu_t]|_*^2\d\mu_t.
\end{align*}
Integrating both sides of \cref{eq:CheinRule} from 0 to T we obtain \cref{eq:EnergyDis}. Moreover, equality holds if and only if for a.e. $t$ and $\mu_t$-a.e. we have $v_t=  \frac{1}{m_{\mu_t}} J_2  (\mathcal{D} W[\mu_t]$). Hence, $\mu_t$ is a weak solution to~\cref{eq:GradFlowPDE}.
\end{proof}

\subsection{Metric gradient flows}
Let us put the previous calculations into the context of curves of maximal slope \cite[Chapter 1]{ambrosio2008gradient} which can be viewed as a way to generalize gradient flows to general metric spaces.
We assume $(X,d)$ to be a complete metric space. Let $\mathcal{E}: X \rightarrow \R$. 
A function $g:X \rightarrow [0,+\infty ]$ is called a strong upper gradient of $\mathcal{E}$ if for any absolutely continuous curve $x:[0,T]\rightarrow X$ the concatenation $g\circ x$ is Borel and
 \begin{align*}
  |\mathcal{E}(x(t))-\mathcal{E}(x(s))|\leq \int_s^t g(x(r)) |\dot{x}|(r)\d r \quad \forall 0\leq s\leq t \leq T.    
 \end{align*}
 If $\mathcal{E}(x(t))$ is non-increasing in $t$ then the application of Young's inequality yields
 \begin{align*}
     \mathcal{E}(x(t))-\mathcal{E}(x(s))+\frac{1}{2} \int_s^t g(x(r))^2+|\dot{x}|(r)^2 \d r\geq0 \quad \forall 0\leq s\leq t\leq T.
 \end{align*}
 This observation allows us to define curves of maximal slope as those  that decrease the energy as fast as possible.

 \begin{definition}[Curve of maximal slope] An absolutely continuous curve $x:[0,T]\rightarrow X$ is called a curve of maximal slope of $\mathcal{E}$ with respect to its strong upper
 gradient $g$ if $t\mapsto E(x(t))$ is non-increasing and
 \begin{align*}
     \mathcal{E}(x(t))-\mathcal{E}(x(s))+\frac{1}{2} \int_s^t g(x(r))^2+|\dot{x}|(r)^2 \d r\leq0 \quad \forall 0\leq s\leq t\leq T.
 \end{align*}
 \end{definition}

 \begin{lemma}
 The map
\begin{align*}
   g: \mu \mapsto \sqrt{\int_M  \frac{1}{m_{\mu}} |\mathcal{D} W[\mu_t]|_*^2\d\mu}
\end{align*}     
is a strong upper gradient of $\mathcal{E}$ and solutions of \eqref{eq:GradFlowPDE} coincide with curves of maximal slope of $\mathcal{E}$ with respect to the strong upper gradient $g$.
 \end{lemma}
 \begin{proof}
  For an absolutely continuous w.r.t. $W_{2,m}$ curve $\mu_t$ we can find, by \cref{lm:AC}, a velocity field $(v_t)_{t\in(0,T)}$ such that $(\mu,v)\in CE(0,T)$ and $$\left( \int_M m_{\mu_t} |v_t|^2\d\mu_t\right)^{1/2} = |\dot{\mu}|(t)\quad \text{for a.e. }t\in(0,T).$$
  Then the chain rule \cref{lm: ChainRule} yields
  \begin{align*}
      | \mathcal{E}(\mu_t)-\mathcal{E}(\mu_s)|\leq \int_s^t | \langle \mathcal{D} W[\mu_t] , v_r\rangle  |\d r\leq \int_s^t  g( \mu_r) |\dot{\mu}|(r)\d r 
  \end{align*}
  and $g$ is a strong upper gradient. The coincidence of solutions of \eqref{eq:GradFlowPDE} and curves of maximal slope follows from \cref{lm: DisE}.
 \end{proof}

\newcommand{\localinclude}[1]{}
\localinclude{%
\subsection{$\lambda$-convexity ?}
We provide a very rough estimate for lambda convexity.
{\color{black} the manifold needs to be a geodesic space !}\todo{TR: there are still open ToDOs here.}
\begin{lemma}
    Let $W(x,\tilde{x})$ be jointly $\lambda$-convex along to arbitrary absolutely continuous curves $\{x_t\}_{t\in[0,T]}$  and $\{x_t\}_{t\in[0,T]}$ on $M$ with $\lambda<0$, i.e.:
    $$ W(x_t,\tilde{x}_t)\leq(1-t) W(x_0,\tilde{x}_0)+tW(x_T,\tilde{x}_T)-\frac{\lambda}{2}(1-t)t\left(d_M^2(x_0,x_T)+d_M^2(\tilde{x}_0,\tilde{x}_T) \right) $$
    then $\mathcal{E}$ is $\tilde{\lambda}$-convex for $\tilde{\lambda}=2 C \lambda$ along arbitrary AC curves.
\end{lemma}
\begin{proof}
We again exploit\todo{TR: Do we have a symbol for the integration domain, where there are currently dots?}
\begin{align*}
\mathcal{E}(\mu_t)
&=
\int_{M\times M} W(x,\tilde{x})\d\mu_t\otimes \mu_t= 
\int_{M\times M} W(x,\tilde{x}) d(e_t)_\# \tilde{\eta}\otimes (e_t)_\# \tilde{\eta}\\
&=
\int_{...\times ...} W(x(t),\tilde{x}(t)) d \tilde{\eta}\otimes  \tilde{\eta}\\
&\leq \int_{...\times ...} (1-t) W(x_0,\tilde{x}_0)+tW(x_T,x_T)-\frac{\lambda}{2}(1-t)t\ \left(d_M^2(x_0,x_T)+d_M^2(\tilde{x}_0,\tilde{x}_T)\right) d \tilde{\eta}\otimes  \tilde{\eta}\\
&=
(1-t) \mathcal{E}(\mu_0)+t\mathcal{E}(\mu_T)-\frac{\lambda}{2}(1-t)t 2\int_{...\times ...} d_M^2(x_0,x_T)d\tilde{\eta}\\
&\leq
...-\frac{\lambda}{2}(1-t)t 2\int_{...} \left( \int_0^T |\dot x_{geo}|(t)\d t\right)^2d\tilde{\eta}\\
&\leq
... -\frac{\lambda}{2}(1-t)t 2\int_{...}  \int_0^T |\dot x_{geo}|^2(t)\d t d\tilde{\eta}=... -\lambda(1-t)t W_2^2(\mu_0,\mu_T)
\end{align*}
\end{proof}
\todo{LW: how does uniqueness follow ???}
\todo{TR: What is $x_{geo}$?}
\todo{TR: Some parts of the above calculation aren't entirely clear to me}
\begin{cor}[Uniqueness]
    
\end{cor}%
} 

\subsection{Energy dissipation and large-time behavior}

Due to the missing \sk{geodesic} convexity properties of the energy, we cannot expect convergence of the evolution to a unique minimizer in the large-time limit. However, we can obtain some weaker results by further analyzing the energy dissipation property  
\begin{equation}
\label{eq:EnergyDisIn}
    {\cal E}(\mu_t) + \frac{1}2 \int_0^t \int_M m_{\mu_s}(x) |\nabla {\cal E}'(\mu_s)|^2\d\mu_s(x)\, \d s  \leq {\cal E}(\mu_0).
\end{equation}
As $s \rightarrow \infty$, we can pick narrowly  convergent subsequences of $\mu_s$ (i.e. converging weakly-star in the Banach space of Radon measures). Moreover, the entropy dissipation inequality above implies
$$  \int_0^\infty \int_M m_{\mu_s}(x) |\nabla {\cal E}'(\mu_s)|^2\,\d\mu_s(x)\, \d s < \infty ,$$
hence along suitable subsequences the entropy dissipation 
$$ D(s) = \int_M m_{\mu_s}(x) |\nabla {\cal E}'(\mu_s)|^2\d\mu_s(x)$$ 
converges to zero since it is non-negative and bounded. To establish the existence of subsequences converging to stationary solutions, we need to identify the limit in suitable spaces. Under appropriate regularity assumptions on the interaction kernel $W$ (satisfied, for example, for the exponential kernel), this is a direct consequence of the Arzelà--Ascoli theorem.

\begin{lemma}\label{lem:212}
Let $M$ be a compact manifold without a boundary, $W \in C^{1,\alpha}(M \times M)$ for some $\alpha > 0$ and symmetric. Moreover, let $\mu^n$ be a sequence of probability measures on $M$. Then the sequences
$$
m_{\mu^n}= \int_M W(\cdot, y)\d\mu^n(y)
\quad \text{and} \quad 
\nabla {\cal E'}(\mu^n) = \int_M \nabla_x W(\cdot, y)\d\mu^n(y) 
$$
have uniformly convergent subsequences. If $\mu^n$ converges narrowly to $\mu^*$, then $m_{\mu^n}$ converges uniformly to $m_{\mu^*}$ and $\nabla {\cal E'}(\mu^n)$ converges uniformly to $\nabla {\cal E'}(\mu^*).$
\end{lemma}

\cref{lem:212} combined with the entropy dissipation inequality \sk{\cref{eq:EnergyDisIn}} yields the following  result.
\begin{cor}\label{cor:convergence_subseq}
Let $M$ be a compact manifold without a boundary,  $W \in C^{1,\alpha}(M \times M)$ for some $\alpha > 0$ and symmetric. Then each weak solution $\mu_t$ of \eqref{eq:continuityequation} with the velocity field \eqref{eq:velocityfield} 
has a narrowly convergent subsequence $\mu_{t_n}$ as $t_n \rightarrow \infty$ whose limit is a stationary solution. 
\end{cor}\color{black}
The following example connects the general results of this section with the transformer dynamics. 
\begin{example}
The transformer dynamics for a finite number of particles described by \cref{eq:dynamics-particles} with $V = \pm D$ correspond to the choice  $M=\mathcal{S}$, $K(x,y)=e^{\dpr{x}{Dy}}$, and $W(x,y)= \pm e^{x\cdot Dy}$. As discussed in \cref{sc:MeasInt}, the corresponding empirical measures $\mu_t$ fulfill the continuity equation \cref{eq:conteq_empirical}. Thus, they solve \cref{eq:continuityequation} in the weak sense with the velocity field \cref{eq:velocityfield} and all requirements of \cref{cor:convergence_subseq} are fulfilled. Therefore, there exists a subsequence of $\mu_t$ that converges narrowly to a stationary solution of the interaction energy $\erg$ defined in \cref{eq:energy}.
\end{example}

This section establishes the relation between the particle model \cref{eq:dynamics-particles} and gradient flows of interaction energies for the special cases $V = \pm D$. The energy dissipation property \cref{eq:EnergyDis} and convergence property from \cref{cor:convergence_subseq} motivate the study of stationary solutions of the energy $\erg$, which we carry out in \cref{sec:energy-min-max,sec:stationary}.
\color{black}
We shall start with minimizers and maximizers.

%% file: minmax.tex
\section{Explicit energy minimizers and maximizers} \label{sec:energy-min-max}

In this section, we compute explicit minimizers and maximizers of the energy $\erg{}$ (from \cref{eq:energy}, i.e. without the factor $1/2$) in different scenarios depending on the properties of the interaction matrix $D$. We make the dependence on the matrix $D$ explicit by employing it as a subscript of the energy. The case $D = \operatorname{Id}$ has been already covered in \cite[Proposition 3.4]{geshkovski2023mathematical} where it is stated that a measure is a maximizer if and only if it is a Dirac delta placed at any point on the sphere and a minimizer if and only if it is the uniform distribution. As we show below, for more general matrices the position of optimal Diracs depends strongly on the eigenvalues of the matrix $D$. We further derive a symmetry condition for minimizers of energies with a positive definite interaction matrix $D$. This property yields an alternative, simpler proof that the uniform distribution is the only minimizer for $D = \operatorname{Id}$.

\subsection{Maximal eigenvalue and related maximizers or minimizers}
Like for $D = \operatorname{Id}$, there are several cases in which the minimizers or maximizers of the energy $\erg{}$ are given by Diracs concentrated at a single point. We start with the maximizers when the largest eigenvalue of $D$ is also an eigenvalue of largest absolute value (or, respectively, minimizers  when the smallest eigenvalue of $D$ is also an eigenvalue of largest absolute value). 

\begin{thm} \label{thm:Diracmaxmin}
Let $\lambda$ be an eigenvalue of maximal absolute value of $D$ and  {$Z_{\lambda} \subseteq \mathcal{S}$ the set of associated normalized eigenvectors.} If $\lambda > 0$ then $\mu^* = \delta_{z}$ with $z \in Z_{\lambda}$ are the only maximizers of the energy $\erg{}$. If $\lambda < 0$ then $\mu^* = \delta_{z}$ with $z \in Z_{\lambda}$  are  {the only} minimizers. 
\end{thm}
\begin{proof}
We consider the case $\lambda > 0$; the case $\lambda <0$ can be treated similarly.
For all $x,y \in {\cal S}$ we have $e^{x \cdot D y} \leq e^{ \lambda}$ with equality if  {and only if }$x=y=\pm z$. Thus,
$$ \erg{}(\mu)  =  \int_{\cal S} \int_{\cal S} e^{x \cdot Dy}\d\mu(x)\d\mu(y) \leq  
\int_\uS \int_\uS e^{\lambda}\d\mu(x)\d\mu(y)
= e^\lambda = \erg{}(\mu^*),$$
 {where the inequality is strict if $\mu$ is not concentrated on an eigenvector associated to $\lambda$.}
\end{proof}
An example of the above setting is maximizing the energy for $D = \operatorname{Id}$~\cite[Proposition 3.4]{geshkovski2023mathematical}, where the authors make a connection between the existence of concentrated  maximizers and the so-called mode collapse of transformers often observed in practice. For a positive definite $D \neq \operatorname{Id}$, \cref{thm:Diracmaxmin} shows that the set of maximizers is not only restricted to Dirac measures, but that it is actually finite. We summarize this insight in the following example and refer to \cref{sec:num_diracmaxis} for an illustrating numerical example. 
\begin{example}
    If $D = \Id$ then $\mu^* = \delta_{z}$ is a maximizer of the energy $\mathcal{E}_{\Id}$ for any $z \in {\cal S}$. Similarly, for $D = - \operatorname{Id}$, $\mu^* = \delta_{z}$ is a minimizer  for any $z \in {\cal S}$.
If $D \neq \operatorname{Id}$ is positive definite then $\mu^* = \delta_{z}$ is  a maximizer of  $\erg{}$ only if $Dz = \lambda z$ and $\lambda$ is the largest eigenvalue of $D$. Similarly, for a negative definite $D \neq \operatorname{Id}$, $\mu^* = \delta_{z}$ is a minimizer only if $Dz = \lambda z$ and $\lambda$ is the smallest eigenvalue of $D$.
\end{example}
In the remainder of this section we study minimizers for matrices that do not fulfill the conditions of \cref{thm:Diracmaxmin}.

\subsection{Minimizers for indefinite matrices }
We now generalize the statement in \cref{thm:Diracmaxmin} to minimizers of energies where the matrix $D$ has at least one non-positive eigenvalue. In particular, that we do not assume that the smallest eigenvalue is the eigenvalue of maximal absolute value. A key property is the following result that gives a lower bound on the energy in terms of the smallest eigenvalue of $D$.

\begin{lemma}\label{lem:lowerboundindefinite}
Let $\bar{x}$ be the expected value of $x$ under $\mu$, i.e., 
$ \bar{x} := \int_{\cal S} x\d\mu(x)$.
Then
\begin{equation} \label{eq:firstlowerboundonE}
\erg{}(\mu) \geq e^{\dpr{\bar{x}}{D \bar{x}}}.
\end{equation}
If $D$ is not positive definite and $\lambda_{\min}$ is its smallest eigenvalue,
it further holds that 
\begin{equation} \label{eq:secondlowerboundonE}
\erg{}(\mu) \geq e^{\lambda_{\min}}.
\end{equation}
\end{lemma}
\begin{proof}
We use the convexity of exponential functions of the form $x \mapsto e^{x \cdot a}$ and $y \mapsto e^{b \cdot y}$ for arbitrary $a,b \in \R^n$, which, with two applications of Jensen's inequality, implies
\begin{align}\label{eq:JensenJensen}
\erg{}(\mu) &= \int_{\cal S} \int_{\cal S} e^{\dpr{x}{D y}}\d\mu(y)\d\mu(y) \int_{\cal S} e^{\dpr{x}{D \bar{x}}}\d\mu(x) \geq e^{\dpr{\bar{x}}{D \bar{x}}}.
\end{align}
Since, further, $
    \dpr{\bar{x}}{D\bar{x}} \geq \lambda_{\min}\, \|\bar{x}\|^2$
and $0\leq\|\bar{x}\|\leq1$, the monotonicity of the exponential function gives us 
\begin{align*}
    \erg{}(\mu) \geq e^{\min\{\lambda_{\min}, 0\}}.
\end{align*}
If $D$ is not positive definite, we know that $\lambda_{\min} \leq 0$ and the above inequality reduces to \eqref{eq:secondlowerboundonE}.
\end{proof}
A direct consequence of \cref{lem:lowerboundindefinite} for indefinite matrices is that a Dirac measure that is concentrated on an eigenvector corresponding to the smallest eigenvalue is a minimizer of the energy. If the smallest eigenvalue is negative, we can even show that all minimizers are of this form. In the case of a vanishing smallest eigenvalue, it is necessary and sufficient that the measure is concentrated on the null space of $D$.
\begin{thm}\label{thm:minmaxindefinite}
 Consider a matrix $D$ that is not positive definite with the smallest eigenvalue $\lambda_{\min} \leq 0$.
 If $\lambda_{\min} < 0$, a measure minimizes the energy if and only if it is a Dirac measure placed at an eigenvector corresponding to $\lambda_{\min}$.
 If $\lambda_{\min} = 0$,
 a measure minimizes the energy if and only if it is concentrated on the null space of $D$.
\end{thm}%
\begin{proof}
    We first assume $\lambda_{\min} < 0$. If follows directly from \eqref{eq:secondlowerboundonE} that every Dirac measure concentrated on an eigenvector corresponding to $\lambda_{\min}$ is a minimizer. We further see that $\dpr{\bar{x}}{D\bar{x}} = \lambda_{\min}$ if only if $\bar{x}$ is an eigenvector corresponding to $\lambda_{\min}$ and $\|\bar{x}\| = 1$. This can only hold for Dirac measures. Thus, there are no other minimizers. 
    
    For $\lambda_{\min} = 0$, it also follows directly from \eqref{eq:secondlowerboundonE} that every measure concentrated on the null space of $D$ minimizes the energy. However, $\dpr{\bar{x}}{D\bar{x}} = \lambda_{\min}$ holds for all measures that fulfill $\bar{x} = 0$. Still, the estimate \eqref{eq:JensenJensen}, obtained using Jensen's inequality, is only an equality if  $\dpr{x}{Dx} = \dpr{\bar{x}}{D\bar{x}} = 0$ for $\mu$-a.e. $x\in\S$. Therefore, all minimizers are concentrated on the null space of $D$.  
\end{proof}
\begin{rem}\label{rem:ndd}
    In general, \cref{thm:minmaxindefinite} does  not transfer to maximizers for matrices $D$ that are not negative definite. To see this, consider $D$ with the largest eigenvalue $\lambda_{\max} \geq 0$, the smallest eigenvalue $\lambda_{\min} < 0$ and corresponding eigenvectors $z_{\min}$ and $z_{\max}$. If further $e^{\lambda_{\max}} < \cosh(\lambda_{\min})$ it holds that
    \begin{align*}
        \erg{}(\delta_{z_{\max}}) = e^{\lambda_{\max}} < \cosh(\lambda_{\min}) = \erg{}\left(\frac{\delta_{z_{\min}} +\delta_{-z_{\min}}}{2}\right)
    \end{align*}
    and thus, $\delta_{z_{\max}}$ is not a maximizer. In the special case $\lambda_{\max} = 0$, the above inequality holds for all measures concentrated on the null space of $D$ and all $\lambda_{\min}  < 0$.
\end{rem}
 At this point we further note that the above strategy does not work for analyzing minimizers for positive definite interaction matrices $D$. In this case, \cref{lem:lowerboundindefinite} only gives us $\erg{}(\mu) \geq e^0 = 1$, but also $\dpr{x}{Dx} > 0$ for all $x \in \S$, so the inequality is strict for all measures $\mu \in \mathcal{P}(\S)$.  
 


\subsection{Symmetry property for positive definite matrices}
The remainder of this section gives the first characterization of minimizers of the energy when the interaction matrix is positive definite. More precisely, we can show that in this case all minimizers are symmetric and the symmetry axes are determined by the eigenvectors of $D$. The first step towards this  is to show that the energy $\erg{}$ is strictly convex if $D$ is positive definite.
\begin{lemma} \label{lem:convexity}
    If $D$ is positive semi-definite (resp. positive definite) then $\erg{}$ is convex (resp. strictly convex).
\end{lemma}
\begin{proof}
Since $\erg{}$ is quadratic, convexity (resp. strict convexity) follows from the non-negativity (resp. positivity) of the quadratic form
$${\cal F}(\mu) = \int_{\cal S} \int_{\cal S}e^{\dpr{x}{Dy}}\d\mu(x)\d\mu(y) $$
for arbitrary
signed Radon measures $\mu$, e.g. \cite[Proposition 2.11]{bilyk2022positive}. For $D$ positive semi-definite there exists a unique positive semi-definite matrix square root $D^{1/2}$ and we can use the transformation $T(x) = D^{1/2} x$. 
We
denote by $T_\# \mu $ the push-forward of $\mu$ by $T$, so that 
\begin{align*}{\cal F}(\mu) &= \int_{T({\cal S})} \int_{T({\cal S})}e^{x \cdot y} \d T_\# \mu(x) \d T_\# \mu(y) \\ &=
\int_{T({\cal S})} \int_{T({\cal S})}e^{ -\frac{1}2|x - y|^2} e^{\frac{1}2 |x|^2}\d T_\# \mu(x)  e^{\frac{1}2 |y|^2}\d T_\# \mu(y).
\end{align*}
Let $\d\eta = e^{\frac{1}2 |x|^2}\d T_\# \mu(x)$, then 
$${\cal F}(\mu) =  \int_{T(\S)} \int_{T(\S)} e^{ -\frac{1}2|x - y|^2}\d\eta(x)\d\eta(y).$$
The fact that the Gaussian kernel is positive definite (e.g. \cite{fasshauer2011positive}) yields that $\mathcal{F}(\mu) > 0$ unless $\nu$ vanishes. This can only happen if  $\mu = 0$ or, in case of a semi-definite matrix $D$, if $\mu$ is concentrated on the null space $\mathcal{N}(D)$ and $\mu(\mathcal{N}(D)) = 0$. This yields the assertion.
\end{proof}
\color{black}
\begin{rem}
The previous convexity result does not 
guarantee the convergence of the gradient flow \cref{eq:GradFlowPDE} to a global minimizer of $\mathcal{F}$. For such results usually a slightly different notion of convexity is required, the so-called geodesic convexity. The following example shows that besides the case of $D$ being a multiple of the identity  we do not have geodesic convexity for the classical $2$-Wasserstein distance. We do not expect any improvements for our modified optimal transport distance.
\end{rem}
\begin{example}\label{ex:geodesicconv}
We consider a simple counterexample in  \sk{${\cal S}^1$ (equipped with the spherical distance)} to show that $\mathcal{F}$ is not convex along $2$-Wasserstein geodesics. Choose
$$D=\begin{pmatrix}
  2&0\\
  0&1
\end{pmatrix} \quad \text{and the curve} \quad \gamma:t\in[0,1]\mapsto \begin{pmatrix}
    \cos(-\frac{\pi}{4}+t\frac{\pi}{2})\\
    \sin(-\frac{\pi}{4}+t\frac{\pi}{2})
\end{pmatrix}.$$
 Then $\mu_t\coloneqq\delta_{\gamma(t)}$ is a constant-speed geodesic in the $2$-Wasserstein space connecting $\delta_{\gamma(0)}$ and $\delta_{\gamma(1)}$. Clearly, the map $[0,1] \ni t \mapsto\mathcal{F}(\gamma(t))$ is not convex, since $$\mathcal{F}(\gamma(0))=\mathcal{F}(\gamma(1))=e^{1.5}<e^{2} = \mathcal{F}\left(\gamma\left(\frac{1}{2}\right)\right).$$ Such a counterexample can always be constructed as long as $D$ has two different eigenvectors. \sk{\cref{lem:convexity} does not contradict this counterexample, however, as it only implies the convexity of }$$[0,1] \ni t \mapsto \mathcal{F}((1-t) \mu_0 + t\mu_1).$$\end{example}\color{black}Having established convexity, we can show that reflecting a measure along the eigenvectors of $D$ and then normalizing it does not increase the energy. Moreover, if $D$ is positive definite and $\mu$ is not symmetric with respect to all eigenvectors of $D$, one can always construct a symmetric measure with a smaller energy.
%

\begin{lemma} \label{lem:symmetry}
Let $z$ be an eigenvector related to an eigenvalue $\lambda$ of a positive semi-definite matrix $D$. For a measure $\mu$, we define $\tilde{\mu}$ as
\begin{align*}
\tilde \mu := \frac{1}2 \left( \mu + {H_z}_\#\mu\right),
\qquad
H_z(x) = x - 2 (x \cdot z) z,
\end{align*}
where $H_z$ denotes a reflection.
Then $\erg{}(\tilde \mu) \leq \erg{}(\mu) $ and the inequality is strict if $D$ is positive definite and $\tilde \mu \neq \mu$.
\end{lemma}
\begin{proof}
Since $e^{x \cdot D y} = e^{H_z(x) \cdot D H_z(y)}$, it is straightforward to see that  $ \erg{}(\mu) = \erg{}({H_z}_\#\mu)$. The (strict) convexity of the energy yields the assertion. 
\end{proof}
As a direct consequence, we obtain a symmetry property of minimizers for positive definite~$D$.

\begin{cor}\label{cor:psd-symmetric}
If $D$ is positive definite then each minimizer is symmetric with respect to its eigenvectors.
\end{cor}

If $D$ is a positive multiple of the identity, one can easily show using the above result that the uniform distribution is the unique energy minimizer. This has been shown already in \cite[Proposition 3.4]{geshkovski2023mathematical} using properties of 
Gegenbauer polynomials~\cite[Proposition 2.2]{bilyk2016geodesic}. The symmetry property from \cref{cor:psd-symmetric} gives an alternative---and straightforward---proof of this fact. 

\begin{prop}\label{prop:uniformmin}
    If $D=\lambda \, \id$ for $\lambda >0 $ then the uniform distribution is the unique energy minimizer.
\end{prop}
\begin{proof} 
If $\mu$ is not uniform, we can find a unit vector $z$ such that  with $H_z$ as in Lemma \ref{lem:symmetry} we have
$$\tilde \mu =\frac{1}2 \left( \mu + {H_z}_\#\mu\right) \neq \mu. $$
However, for $D=\lambda \id$ every unit vector is an eigenvector and  \cref{lem:symmetry} implies that
$ \erg{}(\tilde \mu) < \erg{}(\mu)$. Hence, the uniform distribution is the only minimizer of the energy.
\end{proof}
\begin{rem}\label{rem:ndb}
    The statement in \cref{prop:uniformmin} does not transfer to maximizers for negative multiples of the identity. 
    To see this,  consider $D = \lambda \operatorname{Id}$ with $\lambda < 0$ and let $\mu_0$ denote the uniform distribution on $\S$. The symmetry of $\mu_0$  yields
    \begin{align*}
        \erg{}(\mu_0) = 2\int_{\S^+}\int_{\S^+}e^{\lambda\dpr{x}{y}}+e^{-\lambda\dpr{x}{y}}\,\d\mu_0(x)\,\d\mu_0(y) = 4\int_{\S^+}\int_{\S^+}\cosh(\lambda \dpr{x}{y})\,\d\mu_0(x)\,\d\mu_0(y), 
    \end{align*}
    where $\S^+:= \{x\in\S: \,x_1 > 0\}$. Since $|\dpr{x}{y}| < 1$ $\mu_0\times\mu_0$-almost everywhere on $\S^+\times\S^+$ the integrand can be strictly bounded from above by $4\cosh(\lambda)$. Since $\mu_0(\S^+) = 1/2$ it follows that
    \begin{align*}
        \erg{}(\mu_0) < \cosh(\lambda) = \erg{}\left(1/2\,(\delta_{z} + \delta_{-z})\right),
    \end{align*}
    with $z \in \S$.
   ´Therefore, $\mu_0$ cannot be a maximizer of $\erg{}$.
\end{rem}
\begin{rem}\label{rem:ndc}
    The above argument can be used to show that for arbitrary $D$, one has
    \begin{align*}
        \erg{}(\mu) \leq \erg{}\left(\frac{\delta_{z} +\delta_{-z}}{2}\right)
    \end{align*}
    for all symmetric measures $\mu$ if and only if $z$ is an eigenvector that corresponds to the eigenvalue of largest absolute value. In the upcoming section, we use
    this insight to show that such measures are maximizers of $\erg{}$ for negative semi-definite $D$.
\end{rem} 
If $D$ has non-positive eigenvalues,  \cref{thm:Diracmaxmin,thm:minmaxindefinite} still show that all minimizers are invariant with respect to reflections $H_z$, where $z$ corresponds to a positive eigenvalue. However, if $D$ has negative eigenvalues, such reflections can increase the energy when they are applied to general, non-minimizing measures. This is illustrated by the following example.
\begin{example}
\def\lm{\lambda}
Consider the two-dimensional case with $D = \text{diag}(\lm, 1)$ and $\lm<0$.
%
%
For any $\theta \in [0,2\pi)$, denote by $\delta_\theta$ the Dirac delta placed at $(\cos(\theta), \sin(\theta))$. Fix $\phi\in [0,2\pi)$ and let
\begin{align*}
\mu = \frac{1}{2}(\delta_\phi + \delta_{\pi + \phi}).
\end{align*}
In the two-dimensional setting, the symmetrization is given by
\begin{align*}
\tilde{\mu} = \frac{1}{4}(
\delta_\phi + \delta_{\pi + \phi}
+
\delta_{-\phi} + \delta_{\pi - \phi}
).
\end{align*}
Denoting, for convenience, $\cos(\phi) = c$, we have
\begin{align*}
\erg{}(\mu) -
\erg{}(\tilde{\mu}) = 
\frac{1}{2}
\bigg(
\cosh\left|(\lm -1)c^2 + 1\right|
-
\cosh{\left|(-\lm - 1)c^2 + 1\right|}
\bigg).
\end{align*}
Since $t\mapsto \cosh(t)$ is strictly increasing  for $t\geq 0$, we  get that $\erg{}(\mu) \leq \erg{}(\tilde{\mu})$ since
\begin{align*}
\left|(\lm -1)c^2 + 1\right| = \left|-|\lm|c^2 + 1-c^2\right|
\leq |\lm|c^2 + 1-c^2 = \left||\lm|c^2 + 1-c^2\right| = \left|(-\lm - 1)c^2 + 1\right|
\end{align*}
for any $0 \leq c \leq 1$ and $\lm \leq 0$, and the inequality is strict if and only if $0 < c < 1$ and $\lm < 0$. 
\end{example}

\subsection{Maximizers for negative semi-definite matrices}\label{sec:maxnd}
There is no apparent way to use the  proof strategy from the previous Section for showing that maximizers for negative definite matrices are symmetric since the kernel $(x,y) \mapsto e^{\dpr{x}{Dy}}$ is not negative definite for a negative definite $D$. However, we can show that the quadratic form ${\cal F}$ used to prove \cref{lem:convexity} is non-positive for anti-symmetric measures. This yields a symmetry property of maximizers for negative semi-definite matrices.  
\begin{lemma}\label{lem:ndsymm}
Let $D$ be a negative semi-definite matrix and $\mu$ a measure on the sphere. Define $\tilde{\mu}$ as
\begin{align*}
    \d\tilde{\mu}(x) = \frac{1}{2} (\d\mu(x) + \d\mu(-x)).
\end{align*}
Then $\erg{(\tilde{\mu})} \geq \erg{(\mu)}$ and the inequality is strict if  $\tilde{\mu} \neq \mu$ and either $D$ is negative definite or $\tilde{\mu} = \mu$ on the null space $\mathcal{N}(D)$.
\end{lemma}

\begin{proof}
\def\tmu{\tilde{\mu}}%
We denote by $N(x)=-x$ the negation and define
\begin{align*}
\mu^+:=\mu,\quad \mu^-:=N_\#\mu,\quad \zeta:=1/2\,(\mu^- - \mu^+).
\end{align*} 
This yields that $\d\zeta(-x) = 2(\d\mu(-x) -\d\mu(x)) = -\d\zeta(x)$ and
\begin{align*}
\erg(\zeta) =& 
\int_\uS\int_\uS e^{\dpr{x}{Dy}} \d\zeta(x) \d\zeta(y) = 
\int_{\uS^+}\int_{\uS^+} e^{\dpr{x}{Dy}} \d\zeta(x) \d\zeta(y)  + 
\int_{\uS^+}\int_{\uS^+} e^{\dpr{x}{Dy}} \d\zeta(-x) \d\zeta(-y)\\
+& 
2\int_{\uS^+}\int_{\uS^+} e^{-\dpr{x}{Dy}} \d\zeta(-x) \d\zeta(y)
=
2
\int_{\uS^+} \int_{\uS^+} e^{\dpr{x}{Dy}} - e^{\dpr{-x}{Dy}} \d\zeta(x) \d\zeta(y) = - \ergp_{-D}(\zeta).
\end{align*}
Since $-D$ is positive semi-definite, the proof of \cref{lem:convexity} shows that $ \ergp_{-D}(\zeta)\geq 0$ and thus $\erg(\zeta)\leq 0$. The inequality is strict if $\zeta \neq 0$ and either $D$ is negative definite or $\zeta$ is concentrated on $\mathcal{N}(D)^\perp$. The symmetry of the kernel yields $\erg(\mu^-) = \erg(\mu^+)$. Further, substituting $\mu^+ = \tilde{\mu} + \zeta$ and $\mu^- = \tilde{\mu} - \zeta$ we see that
\begin{align*}
\erg(\tmu) &= \frac{1}{4}\erg(\mu^+) + \frac{1}{4}\erg(\mu^-) + \frac{1}{2} \erg(\mu^+, \mu^-)\\ 
&= \frac{1}{2}\erg(\mu) + \frac{1}{2} \erg(\tilde{\mu}+\zeta, \tilde{\mu}-\zeta) = \frac{1}{2}\erg(\mu)+ \frac{1}{2} \erg(\tilde{\mu}) -\frac{1}{2}\erg{(\zeta)}.
\end{align*}
Reordering the terms leads to 
\begin{align*}
    \erg{(\tmu)} = \erg{(\mu)} - \erg{(\zeta)} \geq \erg{(\mu)}.
\end{align*} 
From the conditions on $\zeta$ and $D$ that lead to $\erg < 0$ we derive that the above inequality is strict if $\tilde{\mu} \neq \mu$ and either $D$ negative definite or $\tilde{\mu} = \mu$ on $\mathcal{N}$. 
\end{proof}
\begin{cor}\label{cor:ndsymm}
Let $\mu^*$ be a maximizer of $\erg{}$ for a negative definite $D$. Then  $\d\mu^*(x) = \d\mu^*(-x)$.
\end{cor}
This symmetry property is the missing ingredient for showing that the discrete measures introduced in \cref{rem:ndb,rem:ndc} are maximizers for negative semi-definite matrices $D$. 
\begin{thm}\label{thm:maxnd}
    Let $D$ be negative semi-definite and $\lambda_{\min} < 0$ its smallest eigenvalue. Then a measure $\mu$ maximizes $\erg{}$ if and only if $\mu^* = 1/2\, (\delta_{z}+\delta_{-z})$
    where $z \in \S$ is an eigenvector associated to $\lambda_{\min}$. 
\end{thm}
\begin{proof}
By \cref{lem:ndsymm} it suffices to consider $\mu$ satisfying $\d\mu(x) = \d\mu(-x)$. Denoting $\S^+:=\{x\in \S: x_1 > 0\}$ and using the symmetry property of $\mu$, with the arguments  from \cref{rem:ndb} we have 
\begin{align*}
    \erg{(\mu)} \leq \cosh{\lambda_{\min}} = \erg{(\mu^*)}, 
\end{align*}
where equality is only obtained if $|\dpr{x}{Dy}| = \lambda_{\min}$ holds $\mu\times\mu$-almost everywhere on $\S^+\times\S^+$. Since $\mu$ is symmetric, this is equivalent to $\mu = \mu^*$. For a negative definite $D$ we already know from \cref{cor:ndsymm} that there are no other measures that maximize $\erg$. In the negative semi-definite case, we have that any $\mu$ that fulfills $\erg{(\mu)} = \cosh{\lambda_{\min}}$ has to be concentrated on $\mathcal{N}(D)^\perp$ and, therefore, also in this case there are no other maximizers.
\end{proof}

%% file: stationary.tex
\section{Energy variation and stationary points}\label{sec:stationary}

In order to study stationary points or local maximizers / minimizers, it is useful to consider the first and second variations of the energy on the Wasserstein space of probability measures on the sphere, as studied previously for Vlasov-type interactions. e.g. the mean-field aggregation equation, cf. \cite{Burger2013,carrillo2017geometry,gomez2024beginner}. 
The first variation of $\erg{}$ is given by
\begin{equation}\label{eq:first-variation}
\d\erg{}(\mu;V)=\frac{\d}{\d t} \erg{}(\mu_t)|_{t=0}
\end{equation}
where $\mu_t$ satisfies 
\begin{equation}\label{eq:transport-equation} 
\partial_t \mu_t + \nabla \cdot  ( \mu_t P^\perp_x V) = 0, \quad \mu_0=\mu,
\end{equation}
where $P^\perp_x= \Id - x x^T$ is the projection to the tangent space of the unit ball at $x$. Here, the velocity field $V$ is an arbitrary Lipschitz function on $\R^n$; by the projection $P^\perp_x$ we restrict it further to admissible velocities that keep the distribution on the unit sphere. 

The following weak formulation, where $\varphi$ is a continuously differentiable test function, will be useful later 
$$ \frac{\d}{\d t} \int_{\cal S} \varphi(x)\, \d\mu_t(x) = 
\int_{\cal S} P^\perp_x\nabla \varphi(x) \cdot V(x) \,\d\mu_t(x).$$
Similarly to the first variation, the second variation of $\erg{}$ can be defined as 
\begin{equation}\label{eq:second-variation}
\d^2\erg{}(\mu;V.W)=\frac{\d}{\d t} \d\erg{}(\mu_t,W)|_{t=0}
\end{equation}
if the derivative on the right-hand side exists. The computation of the first variation is completely analogous to the case of the aggregation equation (cf. \cite{Burger2013}) and thus omitted here.
\begin{lemma}\label{lem:firstvar} 
For  any Lipschitz continuous vector field $V$, the first variation of the energy $\erg{}$ in the direction $V$ exists and is given by 
\begin{equation}
    \d\erg{}(\mu;V) = \int_{\cal S} \int_{\cal S} e^{x \cdot (Dy)}  P^\perp_xDy \cdot V(x)\d\mu(x)\d\mu(y) . 
\end{equation}
\end{lemma}

It is straightforward to see that the first variation vanishes at the extremal points of the energy:
\begin{prop}
    Let $\mu^*$ be a minimizer or maximizer of the energy. Then $\d\erg{}(\mu;V) = 0$ for all Lipschitz vector fields $V$.
\end{prop}
\begin{proof}
Let $\mu^*$ be the initial value for the transport equation~\eqref{eq:transport-equation}. 
For Lipschitz-continuous vector fields there is a unique solution $\mu_t$ of the transport equation and for all times $t>0$ it is an admissible distribution on the sphere. Hence, if $\mu$ is a minimizer then
$$ \erg{}(\mu^*) \leq \erg{}(\mu_t)$$
for all $t >0$, which implies that $\d\erg{}(\mu^*;V) \leq 0$ in the limit $t \downarrow 0$. Since $V$ is arbitrary and $\d\erg{}$ is linear in $V$, we have that $\d\erg{}(\mu;V) = 0$. The case of a maximizer is treated in the same way, with an opposite inequality initially.
\end{proof}

The connection between the transformer dynamics and the energy variations in Wasserstein spaces is readily established in the following

\begin{lemma}\label{lem:varstationary} A probability measure $\mu$ is a stationary solution of \eqref{eq:continuityequation} with the velocity field \eqref{eq:velocityfield}  if and only if $\d\erg{}(\mu;W)=0$
 for all Lipschitz continuous $W$.
\end{lemma}


Similarly to \cref{lem:firstvar}, one can obtain an expression for the second variation.

\begin{lemma} \label{lem:secondvar} 
For $V,W$ being Lipschitz continuous, the second variation of the energy $\erg{}$ in the directions $V$, $W$ exists and is given by 
\begin{equation*}
\d\erg{}(\mu;V,W) = \int_{\cal S} \int_{\cal S} e^{\dpr{x}{Dy}}  \left( (P^\perp_xDy \cdot V(x)) (P^\perp_xDy \cdot W(x)) +(Dy)^T \nabla (P^\perp_xV(x)) \right)\d\mu(x)\d\mu(y) . 
\end{equation*}
\end{lemma}

\subsection{Energy variation at concentrated distributions} \label{sec:energy-var-deltas}
From \cref{lem:firstvar} we see that any measure $\mu$ that fulfills
\begin{align}\label{eq:sufficientStationary}
    \int_\S e^{\dpr{x}{Dy}} P^{\perp}_xDy \,\d\mu(y) = 0 \qquad \text{for $\mu$-almost all $x \in \mathcal{S}$},
\end{align}
is  a stationary point of $\erg{}$. Here and in the following, with a slight abuse of notation, we denote the $0$-vector by $0$. For concentrated measures, the above condition is also necessary and moreover rather easy to verify, as we see in what follows.
We first show that single Dirac measures can only be stationary points if they align with an eigenvector of the matrix $D$.
\begin{lemma}
A Dirac measure $\mu^*=\delta_z$ is a stationary point of $\erg{}$ if and only if $z$ is an eigenvector of $D$.
\end{lemma}
\begin{proof}
The first variation is given by 
$$ \d\erg{}(\mu^*;V) = - e^{z\cdot Dz} P_z Dz \cdot V(z). $$
Since $V(z)$ is an arbitrary vector, $\mu^*$ is a stationary point if an only if
$$ 0 = P^\perp_z D z = D z - (z^T D z) z,$$
which holds if and only if $z$ is an eigenvector of $D$.
\end{proof}
Intuitively speaking, $P^\perp_zDz = 0$ means that the force emerging from the interaction of a particle located at eigenvector $z$ with itself is orthogonal to the tangent space of $\S$ at point $z$ and is thus canceled out by the projection. The same effect can be observed for convex combinations of a Dirac measure and its reflection.

\begin{lemma}\label{lem:statcomb}
For any $t\in [0,1]$ we have that $t\delta_z + (1-t)\delta_{-z} $ is a stationary point of $\erg{}$ if and only if $z$ is an eigenvector of $D$.
\end{lemma}
\begin{proof}
Using the expression in \cref{lem:firstvar}, we obtain for any Lipschitz continuous $V$, using the abbreviation $\iota = \dpr{z}{Dz}$, that
\begin{align*}
\d\erg{}(t \delta_z + (1-t)\delta_{-z}; V) 
&= 
t^2\, e^{\iota} P^\perp_z Dz\, V(z) + 
(1-t)^2\, e^\iota P_{-z}^\perp D(-z)\, V(-z)
\\
&+ t (1-t)\, e^{-\iota} P_{-z}^\perp Dz\, V(z) +
t (1-t)\, e^{-\iota} P_z^\perp D(-z)\, V(-z).
\end{align*}
We first observe that for any $x,y$ one has that $P_x^\perp y = P_{-x}^\perp y = -P_x^\perp(-y)$. By comparing the coefficients in the above equation, we obtain that 
\begin{gather*}
\d\erg{}(t)\delta_z + (1-t)\delta_{-z}; V)  = 0\quad \text{for all } V \text{ Lipschitz}\quad \Leftrightarrow\quad
P^\perp_z Dz = 0\\
\Leftrightarrow
Dz - (\dpr{z}{Dz}) z = 0 
\quad \Leftrightarrow
\quad z\text{ is an eigenvector}. \qedhere
\end{gather*}
\end{proof}
For the symmetric case $t = 1/2$ in the above lemma, we can further show that any convex combination of such stationary points is again a stationary point.
\begin{lemma}
Let $Z_D$ be a finite subset of eigenvectors of $D$ such that $\dpr{w}{z} = 0$ for all $z \in Z_D\backslash\{w\}$. Then for any choice of parameters $t: Z_D \rightarrow \R^{+}_0$ such that $\sum_{z \in Z_D} t(z) = 1$ the following measure is a stationary point of $\erg{}$
\begin{align*}
    \mu = \frac12 \sum_{z \in Z_D} t(z
    ) (\delta_z + \delta_{-z}).
\end{align*}
\end{lemma}
\begin{proof}
    We prove the statement by showing that \eqref{eq:sufficientStationary} holds. For any $w \in Z_D$ it holds that 
    \begin{align*}
        P^\perp_w Dw = - P^\perp_w Dw = 0,
    \end{align*}
    since $Z_D$ only contains eigenvectors of $D$. On the other hand, since we also require  $\dpr{w}{z} = 0$ for all $z \in Z_D\backslash\{w\}$ it follows that $\dpr{z}{Dw} = -\dpr{z}{Dw} = 0$ and therefore
    \begin{align*}
        e^{\dpr{w}{Dz}} = e^{-\dpr{w}{Dz}} 
    \end{align*}
    for all $z \in Z_D\backslash\{w\}$. In total this yields 
    \begin{align*}
        \int_\S e^{\dpr{w}{Dy}} P^{\perp}_xDy \,\d\mu(y) = \sum_{z \in Z_D} t(z) \left(e^{\dpr{w}{Dz}} - e^{-\dpr{w}{Dz}}\right) P^\perp_w(Dz) = 0
    \end{align*}
    for all $w \in Z_D$ and thus also for $\mu$-almost all $w \in \S$.
    \end{proof}
The above proof strategy works only for Dirac measures aligned with the eigenvectors of $D$. However, there exist other discrete measures that are stationary points, as the following example shows. For the sake of simplicity, we restrict ourselves to the two-dimensional case with a positive definite matrix $D$ and a symmetric combination of four Dirac measures. We further assume that $D$ is diagonal; the case of a general symmetric $D$ can be treated similarly with a rotation argument.
\begin{lemma}\label{lem:symmetricpeaks}
Let $n= 2$, $\phi \in [0,2\pi)$ and $D$ be diagonal and positive definite. A discrete measure \begin{align}\label{eq:fourpeaks}
    \mu_{\phi} = \frac{1}{|X_\phi|} \sum_{x \in X_{\phi}} \delta_{x}, \quad \text{where} \quad
    X_\phi = \{X(\phi), X(\pi - \phi), X(\pi + \phi), X(2\pi - \phi)\},
\end{align}
is a stationary point of $\erg{}$ if and only if either $\phi \in\{0,\pi/2, \pi\}$ or
\begin{align}\label{eq:tanhcriterion2d}
    \frac{\tanh(\lambda_1 \cos^2\phi)}{\tanh(\lambda_2 \sin^2\phi)} = \frac{\lambda_2}{\lambda_1},
\end{align}
where $\lambda_1,\lambda_2$ denote the diagonal entries of $D$. For any choice of $\lambda_1,\lambda_2 > 0$ there exists exactly one $\phi\in (0,\pi/2)$ that fulfills \eqref{eq:tanhcriterion2d}. 
 \end{lemma}
 \begin{proof}
Without loss of generality we prove the statement for $\phi \in [0,\pi/2]$, since otherwise it holds that $(\psi\mod 2\pi) \in [0,\pi/2]$ for a $\psi \in \{\pi - \phi, \pi + \phi, 2\pi - \phi\}$, and thus $\mu_{\phi} = \mu_{\psi}$.

 It follows directly from \cref{lem:statcomb} that $\mu_{\phi}$ is a stationary point if $\phi \in\{0,\pi/2\}$. Therefore, it remains to show that  $\mu_{\phi}$ is a stationary point if and only if \eqref{eq:tanhcriterion2d} is fulfilled. This means that we have to see when there exists a Lipschitz continuous $V$ such that $\d\erg{}(\mu_{\phi}, V) \neq 0$.

We first fix $x \in \mathcal{S}$ and consider 
\begin{multline}
    \label{eq:oneintvector}
    \int_{\mathcal{S}} e^{\dpr{x}{Dy}} P^\perp_{x}{Dy}\, \d\mu_{\phi}(y) 
    = \frac{1}{4} \left( \left(e^{\dpr{x}{DX(\phi)}}-e^{-\dpr{x}{DX(\phi)}}\right)P^\perp_xDX(\phi) \right. \\ \left. + \left(
    e^{\dpr{x}{DX(\pi - \phi)}}-e^{-\dpr{x}{DX(\pi - \phi)}}\right)P^\perp_xDX(\pi-\phi)\right).
\end{multline}
Since $n = 2$, we can further write $P^\perp_xy = \dpr{x^\perp}{y}\,x^\perp$, where $x^\perp = (-x_2,x_1)^T$. We factor out $x^\perp$ to rewrite \eqref{eq:oneintvector} as $E(x; \mu_\phi) \,x^\perp$ with
    \begin{align*}
        E(x; \mu_\phi) =  (1/2) \left(\sinh(\dpr{x}{DX(\phi)})\,\dpr{x^\perp}{DX(\phi)} + \sinh(\dpr{x}{DX(\pi - \phi)})\,\dpr{x^\perp}{DX(\pi-\phi)}\right).
    \end{align*}
    \cref{lem:firstvar} now gives us that
    \begin{align*}
        \d\erg{}(\mu_{\phi}, V) = \sum_{x \in X_\phi} E(x;\mu_\phi)\,\dpr{x^\perp}{V(x)},
    \end{align*}
    which can  become zero for all admissible $V$ if and only if $E(x;\mu_\phi) = 0$ for all $x \in X_\phi$. Due to the symmetry properties of our measures $\mu_\phi$, it further holds that $E(x;\mu_\phi)$ is constant on $X_\phi$; therefore, it suffices to consider  $x = X(\phi)$. Remembering that $X(\phi) = (\cos\phi, \sin\phi)^T$ we derive
    \begin{align*}
        2E(X(\phi); \mu_\phi) = & \sinh(\lambda_1\, \cos^2\phi + \lambda_2\, \sin^2\phi) \, (-\lambda_1 + \lambda_2)\,\sin\phi\cos\phi\\
        +&\sinh(-\lambda_1\, \cos^2\phi + \lambda_2\, \sin^2\phi) \, (\lambda_1+\lambda_2)\,\sin\phi\cos\phi.
    \end{align*}
    Since $\phi \in (0,\pi/2)$, the factor $\sin\phi\cos\phi$ cannot vanish and  the zeros of $E(X(\phi); \mu_\phi)$ coincide with those of 
    \begin{align}\label{eq:sinhzeros}
        &\sinh(\lambda_1\, \cos^2\phi + \lambda_2\, \sin^2\phi) \, (-\lambda_1 + \lambda_2)\,+\sinh(-\lambda_1\, \cos^2\phi + \lambda_2\, \sin^2\phi) \, (\lambda_1+\lambda_2)\\
        =&  \sinh(\lambda_1 + (-\lambda_1+\lambda_2)\, \sin^2\phi) \, (-\lambda_1 + \lambda_2)\,+\sinh(-\lambda_1 + (\lambda_1+\lambda_2)\, \sin^2\phi) \, (\lambda_1+\lambda_2)\nonumber.
    \end{align}
    This function obtains its minima at $(\phi\mod 2\pi) \in \{0, \pi\}$ and its maxima at $(\phi\mod 2\pi) \in \{\pi/2, 3\pi/2\}$ and strictly increases or decreases, respectively, in between. Substituting  these points into \eqref{eq:sinhzeros}, we see that the minima are strictly negative and the maxima are strictly positive since  $\lambda_1,\lambda_2 > 0$. Therefore, there exists exactly one zero in the interval $(0, \pi/2)$.  Using the hyperbolic identity $\sinh{(x+y)} = \sinh{x}\cosh{y} + \cosh{x}\sinh{y}$ in~\eqref{eq:sinhzeros} we arrive at the criterion \eqref{eq:tanhcriterion2d}.
\end{proof}
\begin{rem}\label{rem:peaks-magnitude}
Importantly, the angle $\phi$ that fulfills \eqref{eq:tanhcriterion2d} depends not only on the ratio of the  eigenvalues of $D$ but also on their magnitude since they appear separately within the hyperbolic tangent. 
\end{rem}
Although the ratio of the eigenvalues does in general not determine the angle $\phi$ that fulfills \eqref{eq:tanhcriterion2d}, we can still make a qualitative prediction based on the ratio. The left-hand side of~\eqref{eq:tanhcriterion2d} decreases monotonously for $\phi \in [0,\pi/2)$; for $\lambda_1 = \lambda_2$, the condition is fulfilled for $\phi = \pi/4$. Therefore, the condition is fulfilled by some $\phi \in [0,\pi/4)$ if $\lambda_2 > \lambda_1$ and by some $\phi \in (\pi/4, \pi/2]$ if $\lambda_1 > \lambda_2$. The numerical experiments in \cref{sec:num_minposdef} show that the measures characterized by \eqref{eq:tanhcriterion2d} are not only stationary points but also minimizers among empirical measures consisting of at most four Dirac measures. In the remainder of this section, we aim to characterize minimizers for positive definite matrices $D$ in arbitrary dimensions $n \geq 2$.
\subsection{Energy variation at the uniform distribution}
To characterize minimizers for positive definite $D$, we start by identifying the cases when the uniform distribution is a stationary state. As we show in the following lemma, this can only be the case if the strength of the interaction does not depend on the direction, i.e. the eigenvalues of $D$ all have the same absolute value.  
\begin{lemma}\label{lem:stationaryuniform}
The uniform distribution $\mu = \frac{1}{|\mathcal{S}^{n-1}|}\mathcal{H}^n$ is a stationary point of $\erg{}$ if and only if all eigenvalues $(\lambda_i)_{i = 1}^n$ of $D$ have the same absolute value, i.e. $|\lambda_i| = \lambda$ for some $\lambda \in \R$.
\end{lemma}
\begin{proof}
To keep the notation simple we treat here the case $n = 2$, leaving the general proof for $n > 2$ to \cref{app:stationuni}. 
Let us fix $x\in \S$ and determine $\phi \in [0,2\pi)$ such that $Dx/\|Dx\| = (\cos\phi, \sin\phi)^T$. Consider the integral
\begin{align*}
   \int_{\S} e^{x\cdot Dy} P^\perp_x Dy \, \d\mathcal{H}^{2}(y) &= \int_{0}^{2\pi} e^{\|Dx\| \cos(\psi - \phi)} P^\perp_x\left(D(\cos\psi, \sin\psi)^T\right) \,\d\psi = (*),
\end{align*}
which can be rewritten with a change of variables $\theta = \psi - \phi$ as follows (recall that $P^\perp_x= \Id - x x^T$)
\begin{align*}
   (*) &= \int_{0}^{2\pi} e^{\|Dx\| \cos\theta} \left(\cos\theta\left(D^2x/\|Dx\| - \|Dx\|x\right) + \sin\theta \,(Dx/\|Dx\|)^\perp\right) \,\d\theta\\
   &= \left(D^2x/\|Dx\| - \|Dx\|x\right) \underbrace{\int_{0}^{2\pi} e^{\|Dx\| \cos\theta} \cos\theta \,\d\theta}_{> 0} + (Dx/\|Dx\|)^\perp \underbrace{\int_{0}^{2\pi} e^{\|Dx\| \cos\theta} \sin\theta \,\d\theta}_{ = 0}.
\end{align*}
From the above derivations we see that $(*) = 0$ if and only if $x$ is an eigenvector of $D^2$. This holds true for $\mu$-almost all $x \in \mathcal{S}$ if and only if $|\lambda_1| = |\lambda_2|$. This automatically yields $\d\erg{}(\mu,V) = 0$ if $|\lambda_1| = |\lambda_2|$. It remains to show that this is also a necessary condition. 

Without loss of generality, we assume that $|\lambda_1| > |\lambda_2|$, where $\lambda_1$ and $\lambda_2$ are the eigenvalues corresponding to the eigenvectors $z_1$ and $z_2$, respectively. Then, $\left(D^2x/\|Dx\| - \|Dx\|x\right) \cdot z_2$ is strictly negative on the set
\begin{align*}
    A = \{x \in \S \, | \, (x\cdot z_1) \in (|\lambda_2/\lambda_1|, 1),\, (x\cdot z_2) > 0\}.
\end{align*}
Since $\mu(A) > 0$ we can find a Lipschitz continuous $V$ such that $V \cdot z_1 = 0$ for $\mu$-a.e. on $\mathcal{S}$ and
\begin{align*}
    V(x)\cdot z_2 \begin{cases}
        > 0 &\text{for a.e. } x \in A \\
        = 0 &\text{for a.e. } x \in \mathcal{S}\backslash A.
    \end{cases}
\end{align*}
For all such $V$ it holds that $\d\erg{}(\mu, V) > 0$, which concludes the proof.
\end{proof}
Since we already know that minimizers for $D$ with at least one negative eigenvalue are Dirac measures, we can conclude that the uniform distribution is only a minimizer for $D = \operatorname{Id}$.
\begin{cor}
    The uniform distribution $\mu = \frac{1}{|\mathcal{S}^{n-1}|}\mathcal{H}^n$ minimizes $\erg{}$ if and only if $D = \lambda \operatorname{Id}$ for $\lambda \geq 0$.
\end{cor}
\begin{proof}
We only need to show that there are no other matrices $D$ such that $\erg{}$ is minimized by $\mu$; the other direction has been treated in \cref{prop:uniformmin}.
The measure $\mu$ can only be a minimizer if it is a stationary point.
By \cref{lem:stationaryuniform}, this implies that all eigenvalues of $D$ have to have the same absolute value. If such $D$ has at least one negative eigenvalue, it is also the smallest eigenvalue. Thus, by \cref{thm:Diracmaxmin}, the only minimizers are Dirac deltas placed at eigenvectors corresponding to the negative eigenvalue. 
\end{proof}

\subsection{Perturbation of the identity}\label{sec:perturabation-identity}
It is not clear whether an explicit computation of stationary points for an arbitrary positive definite matrix $D$ with at least two distinct eigenvalues is possible, but some insight can be gained with asymptotic analysis. We consider the following perturbed energy 
\begin{align*}
    \ergp{}_{\eps}(\mu) := \int_{\S}\int_{\S} e^{\dpr{x}{(\operatorname{Id} + \eps M)y}}\,\d\mu(x)\,\d\mu(y) ,
\end{align*}
where  $M$ is a diagonal matrix and $|\eps| \ll 1$ is a small parameter. Using the second-order Taylor expansion of the exponential function, we can write
\begin{align}\label{eq:expansionInt}
\ergp{}_{\eps}(\mu) \approx\,& \erg{}(\mu) + \eps \int_{\S} \int_{\S} e^{\dpr{x}{y}} \dpr{x}{My} \,\d\mu(x)\, \d\mu(y) + \eps^2 \int_{\S} e^{\dpr{x}{y}} (\dpr{x}{My})^2 \,\d\mu(x)\, \d\mu(y).
\end{align} For $\eps = 0$ we know that the unique minimizer $\mu_0$ is the uniform distribution on the sphere. 
Therefore, we use the following second-order asymptotic ansatz
\begin{equation}\label{eq:expansion-mu}
    \mu_\eps \defeq \mu_0 + \eps \nu + \eps^2 w, \quad \int_\S \d\nu = \int_\S \d w = 0.
\end{equation}
We stress that here we consider the energy as a function on the space of signed Radon measures on the sphere $\cal M(\cal S)$ with the total variation norm and not on the space of probability measures $\cal P(\cal S)$ with the Wasserstein metric as in \cref{sec:energy-var-deltas}. For this reason, the perturbation here is a measure and not a vector field (cf.~\eqref{eq:first-variation}). 

Substituting \eqref{eq:expansion-mu} into \eqref{eq:expansionInt} and neglecting higher-order terms we derive
\begin{align*}
    \ergp{}_{\eps}(\mu_\eps) - \ergp{}_{\eps}(\mu_0) \approx\, \eps\,\erg{}(\mu_0,\nu) + \eps^2\,\erg{}(\mu_0, w) + \eps^2\,\erg{}(\nu) + 2\eps^2 \int_{\S} \int_{\S} e^{\dpr{x}{y}} \dpr{x}{My} \,\d\mu_0(x)\d\nu(y).
\end{align*}
Since further $y\mapsto \int_\S e^{\dpr{x}{y}}\,d\mu_0(x)$ is constant on $\S$, it follows that 
\begin{align*}
    \erg{}(\mu_0,\nu) = C(n) \int_\S \, \d\nu = 0
\qquad \text{and} \qquad 
    \erg{}(\mu_0,w) = C(n) \int_\S \, \d w = 0.
\end{align*}
In particular, we see that the term $\eps^2\omega$ from \eqref{eq:expansion-mu} does not contribute to the second-order expansion of the energy. Therefore, minimizing $\ergp{}_{\eps}$ over all possible $\mu_\eps$ satisfying \eqref{eq:expansion-mu} is equivalent to minimizing
\begin{align*}
    \tilde{\ergp{}}_\eps(\nu) := \eps^2\left( \erg{}(\nu) + 2\int_{\S} \int_{\S} e^{\dpr{x}{y}} \dpr{x}{My} \,\d\mu_0(x)\d\nu(y)\right)
\end{align*}
over all signed measures $\nu$ with $\nu(\mathcal{S}) = 0$. 
The first variation in the direction  $\nu'$ satisfying $\int_\S\,\d\nu' = 0$ is given by
\begin{align}\label{eq:firstvarnu}
    \d\Tilde{\ergp{}}_\eps(\nu, \nu') = 2 \eps^2\left(\int_\S \int_\S e^{\dpr{x}{y}}\, \d\nu(x)\,\d\nu'(y) +  \int_{\S} \int_{\S} e^{\dpr{x}{y}} \dpr{x}{My} \, \d\mu_0(x)\d\nu'(y)\right).
\end{align}
Our goal is now to find an optimal measure $\nu$, such that its first variation vanishes in any direction $\nu'$ such that $\int_\S\,\d\nu' = 0$. To do so, we will need the following two technical lemmas. \color{black}To make the definition of the uniform distribution on the sphere rigorous, we denote by $\mathcal{H}^n$ the $n$-dimensional Hausdorff measure and write $\mathcal{S}^{n-1}$ instead of $\mathcal{S}$. \color{black}
\begin{lemma}\label{lem:technicalx}
 Let $n \geq 2$ and $\mu_0 = \frac{1}{|\mathcal{S}^{n-1}|} \mathcal{H}^n$. It holds that
    \begin{align}\label{eq:intx}
        \int_{\mathcal S^{n-1}} e^{x\cdot y} x \,\d\mu_0(x) = C_1\,y
    \end{align}
    for any $y \in \mathcal{S}^{n-1}$, where the constant $C_1$ is positive and depends only on the dimension $n$.
\end{lemma}
\begin{proof}
For the sake of simplicity, here we present the (more intuitive) proof for $n = 2$, leaving the general case $n > 2$ to \cref{app:stationtechx}.
We write $x = (\cos\phi, \sin\phi)^T$ and $y = (\cos\psi, \sin\psi)^T$ and derive that
\begin{align*}
    2\pi \int_{\mathcal S} e^{x\cdot y} x \,\d\mu_0(x) &= \int_0^{2\pi} e^{\cos(\phi-\psi)}(\cos\phi, \sin\phi)^T\,\d\phi = \int_0^{2\pi} e^{\cos\theta} (\cos(\psi+\theta), \sin(\psi+\theta))^T\,\d\theta \\ &=  (\cos\psi, \sin\psi)^T\int_0^{2\pi} e^{\cos\theta}\cos\theta \,\d\theta + (-\sin\psi, \cos\psi)^T \int_0^{2\pi} e^{\cos\theta} \sin\theta\,\d\theta,
\end{align*}
where we use the coordinate transform $\theta = \phi-\psi$ and two trigonometric identities to separate the summands inside sine and cosine.
Since $\int_S e^{\cos\theta} \sin\theta\,\d\theta = 0$, this  yields \eqref{eq:intx} with
\begin{equation*}
    C_1 = \frac{1}{2\pi}\int_0^{2\pi} e^{\cos\theta}\cos\theta\,\d\theta > 0. \qedhere
\end{equation*}
\end{proof}
\begin{lemma}\label{lem:technicalxx}
     Let $n \geq 2$ and $\mu_0 = \frac{1}{|\mathcal{S}^{n-1}|} \mathcal{H}^n$. It holds that for any $y \in \mathcal{S}^{n-1}$
    \begin{align}\label{eq:intxx}
        \int_{\mathcal S^{n-1}} e^{x\cdot y} x_i^2 \,\d\mu_0(x) = C_2\,y_i^2+ C_3, \qquad \text{$1 \leq i\leq n$},
    \end{align}
    where the constants $C_2$ and $C_3$ are positive and depend only on the dimension $n$.
\end{lemma}
\begin{proof}
For the sake of simplicity, we again present the proof for $n = 2$; the general case $n > 2$ is treated in \cref{app:stationtechxx}. Using the same arguments as in the previous proof, we derive 
\begin{align*}
    &2\pi \int_{\mathcal S} e^{x\cdot y} x^2 \,\d\mu_0(x) =  \int_0^{2\pi} e^{\cos\theta} (\cos^2(\psi+\theta), \sin^2(\psi+\theta))^T\,\d\theta\\
    &= (\cos^2\psi, \sin^2\psi)^T\int_0^{2\pi} e^{\cos\theta}\cos^2\theta \,\d\theta + (\sin^2\psi, \cos^2\psi)^T\int_0^{2\pi} e^{\cos\theta}\sin^2\theta \,\d\theta,
\end{align*} where the mixed terms containing $\cos\theta\sin\theta$ vanish due to symmetry. Further, since $\cos^2\psi + \sin^2\psi = 1$ we can write 
\begin{align*}
    (\sin^2\psi, \cos^2\psi)^T = (1,1)^T - (\cos^2\psi, \sin^2\psi)^T.
\end{align*}
This  yields \eqref{eq:intxx} with positive constants
\begin{equation*}
    C_2 = \frac{1}{2\pi} \int_0^{2\pi} e^{\cos\theta} (\cos^2(\theta) - \sin^2(\theta))\,\d\theta, \qquad C_3 = \frac{1}{2\pi} \int_0^{2\pi}e^{\cos\theta}\sin^2(\theta)\,\d\theta. 
    \qedhere
\end{equation*}
\end{proof}
\cref{lem:technicalx} allows us to rewrite the second summand in \eqref{eq:firstvarnu} such that it contains $\dpr{y}{My}$. Using \cref{lem:technicalxx}, we can then deduce that, up to constants, the measure $-(\dpr{x}{Mx})\,\mu_0(x)$ is a stationary point of $\Tilde{\ergp{}}_\eps$.
\begin{thm}\label{thm:perturbidentity}
The measure
\begin{align*}
    \d\nu^*(x) = \left(\alpha\,\dpr{x}{Mx}+\beta\right)\d\mu_0(x), \quad \text{where \; $\alpha = -C_1/C_2$ \; and \; $\beta= - \int_{\S} \alpha\,\dpr{x}{Mx}\d\mu_0(x)$},
\end{align*}
fulfills $\int_\S\d\nu^* = 0$ and $\d\ergp{}_{\eps}(\nu^*, \nu') = 0$ for all $\nu'$ satisfying $\int_\S \d\nu' = 0$.
\end{thm}
\begin{proof}
From the definition of $\beta$ and $\int_\S\d\mu_0 = 1$ it follows that $\int_\S \d\nu^* = 0$. With
    \cref{lem:technicalx} we write the optimality condition derived from \eqref{eq:firstvarnu} as
    \begin{align*}
    \int_\S \int_\S e^{\dpr{x}{y}}\, \d\nu(x)\,\d\omega(y)  = - C_1 \int_{\S} y\cdot My\,\d\omega(y).
    \end{align*}
    Substituting $\nu^*$ into the left-hand side and using \cref{lem:technicalxx} we get
    \begin{align*}
         \int_\S \int_\S e^{\dpr{x}{y}}\, \d\nu^*(x)\,\d\omega(y) &= \int_\S \alpha \,\left(C_2\,\dpr{y}{My} + \operatorname{Tr}(M)\,C_3\right)\,\d\omega(y)+\beta\int_\S \int_\S e^{\dpr{x}{y}}\,\d\mu_0(x)\,\d\omega(y)\\
         &= \alpha \,C_2\int_{\S} \dpr{y}{My}\,\d\omega(y),
    \end{align*}
    where all terms that do not depend on $y$, including $\int_\S e^{\dpr{x}{y}}\,\d\mu_0(x)$, vanish due to $\int_\S\,\d\omega = 0$. Substituting $\alpha = -C_1/C_2$ completes the proof.
\end{proof}
\cref{thm:perturbidentity} gives us the following intuitive characterization. The measure $\mu_\eps$ that optimizes the perturbed energy is obtained by taking mass from the uniform distribution where $(\dpr{x}{Mx})$ is large and adding it where $(\dpr{x}{Mx})$ is small. In other words, we expect minimizers of the energy $\erg{}$ with a positive definite matrix $D$ to have more mass in regions that correspond to small eigenvalues of $D$ than in regions that correspond to large ones. This intuition is in line with the results of the particle approximation in \cref{fig:min}. Furthermore, in \cref{fig:densityall} we also observe that the density obtained in \cref{eq:expansion-mu} with the measure $\nu^*$ from above can indeed be seen as a first order approximation for small values of $\eps$.

%% file: numerics.tex
\section{Numerical examples}\label{sec:numerics}

To illustrate the obtained theoretical results we perform a series of numerical experiments using a particle approximation of the energy~\cref{eq:energy} with with an ensemble of $N$ particles $X=(X_1, \ldots, X_N)$,
%
\begin{align*}
\erg(\mu_N(X)),\qquad\text{where}\qquad 
\mu_N(X) = \frac{1}{N}\sum_{i=1}^N \delta_{X_i}.
\end{align*}
\color{black}
We consider the following  particle flow, introduced in \cite{geshkovski2023mathematical},
\begin{align*}
\dot{X_i}(t) = P^\perp_{X_i(t)} 
\left( 
\pm
\frac{1}{J_i(X)}
\sum_{j=1}^N e^{\dpr{X_i(t)}{DX_j(t)}} DX_j(t)
\right)
\end{align*}
with normalization factors $J_i(X)$. If we choose the constant normalization 
\begin{align}\label{eq:normconst}
J_i(X)=N, 
\end{align}
 this corresponds merely to a step-size rescaling of a standard gradient descent scheme for $\erg$, which is called the (USA) flow in \cite{geshkovski2023mathematical}. Choosing the normalization as the partition function,
\begin{align}\label{eq:normpart}
J_i(X) = \sum_{j=1}^N e^{\dpr{X_i(t)}{DX_j(t)}},
\end{align}
corresponds more closely to the self-attention dynamics and is labelled the (SA) flow in \cite{geshkovski2023mathematical}. In what follows, we mostly use the normalization in \cref{eq:normpart}, highlighting minor differences between the two formulations as appropriate. %
We use the explicit Euler discretization from~\cref{eq:expEuler} with step size $\tau>0$ to obtain the following update
\begin{align}\label{eq:discusa}
X_i(t + \tau) =
\Pi
\left(
X_i(t) \pm 
\frac{\tau}{J_i(X)}
\sum_{j=1}^N e^{\dpr{X_i(t)}{DX_j(t)}} DX_j(t)
\right).
\end{align}
\color{black}%
\begin{rem}
For $N=1$ and this scheme reduces to the following power iteration in the limit $\tau\to\infty$
\begin{align*}
X_1(t + \tau) =
\Pi(DX_1(t)).
\end{align*}
In this regard, the iteration~\eqref{eq:discusa} can be seen as a method for approximating the largest eigenvalue and the corresponding eigenvector. We leave further analysis of this connection to future work. 
\end{rem}

The source code for the experiments here is available at 
\url{https://github.com/TimRoith/TransformerDynamics}
and uses \texttt{Python} \cite{van1995python}, mainly building upon the packages \texttt{NumPy} \cite{harris2020array}, \texttt{SciPy} \cite{2020SciPy-NMeth}, and \texttt{PyTorch} \cite{paszke2019pytorch}.
\subsection{Maximizers for positive definite matrices}\label{sec:num_diracmaxis}
To validate our results on maximizers, we first consider a simple setup of one-particle system, $N=1$. We choose $\tau=0.075$ and run~\eqref{eq:discusa} for $1500$ iterations. \color{black} We only report the results for the adaptive normalization \cref{eq:normpart}, those for the constant normalization \cref{eq:normconst} being essentially the same. \color{black} For $D=\Id$ we know that every single Dirac is a maximizer, which is indeed observed in \cref{fig:maxid}. Here, each random initialization on the sphere leads to a different final state. In fact, in this case there is no evolution at all and the particle stays at its initial position. 
If $D$ is positive definite and has a strictly largest eigenvalue $\lambda_\text{max}$, \cref{thm:Diracmaxmin} shows that only Diracs at eigenvectors $z_\text{max}$ corresponding to $\lambda_\text{max}$ are maximizers. This can be observed in \cref{fig:maxpert} where the final state is either at $z_\text{max}$ or $-z_\text{max}$. 
\begin{figure}
\begin{subfigure}[t]{.47\textwidth}%
\centering%
\includegraphics[width=.8\textwidth,trim={.9cm 0cm 0cm 0cm},clip]{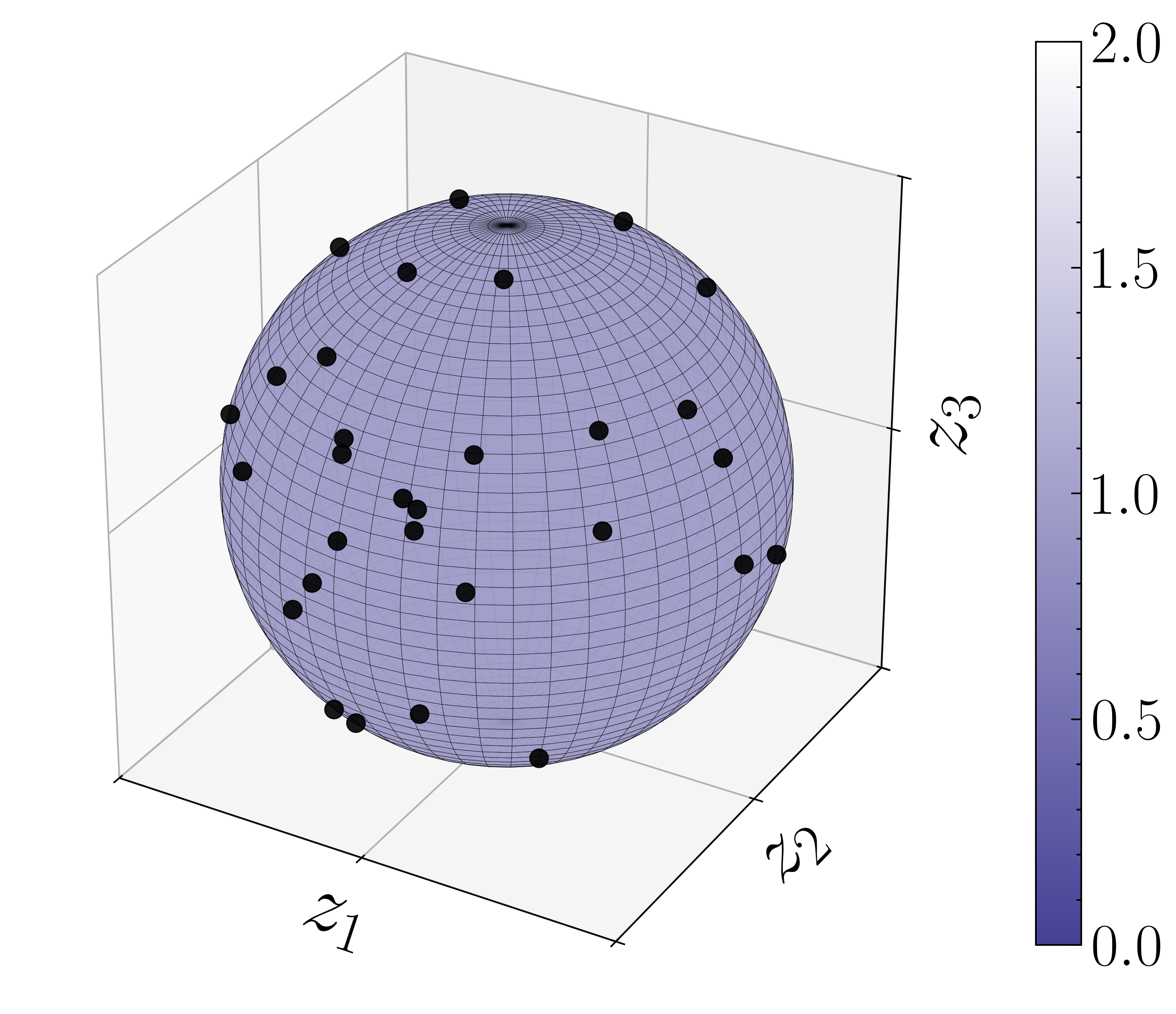}\caption{For $D=\Id$ every single Dirac is a maximizer. We show the results for $30$ different initializations}\label{fig:maxid}
\end{subfigure}\hfill%
\begin{subfigure}[t]{.47\textwidth}%
\centering%
\includegraphics[width=.8\textwidth,trim={.9cm 0cm 0cm 0cm},clip]{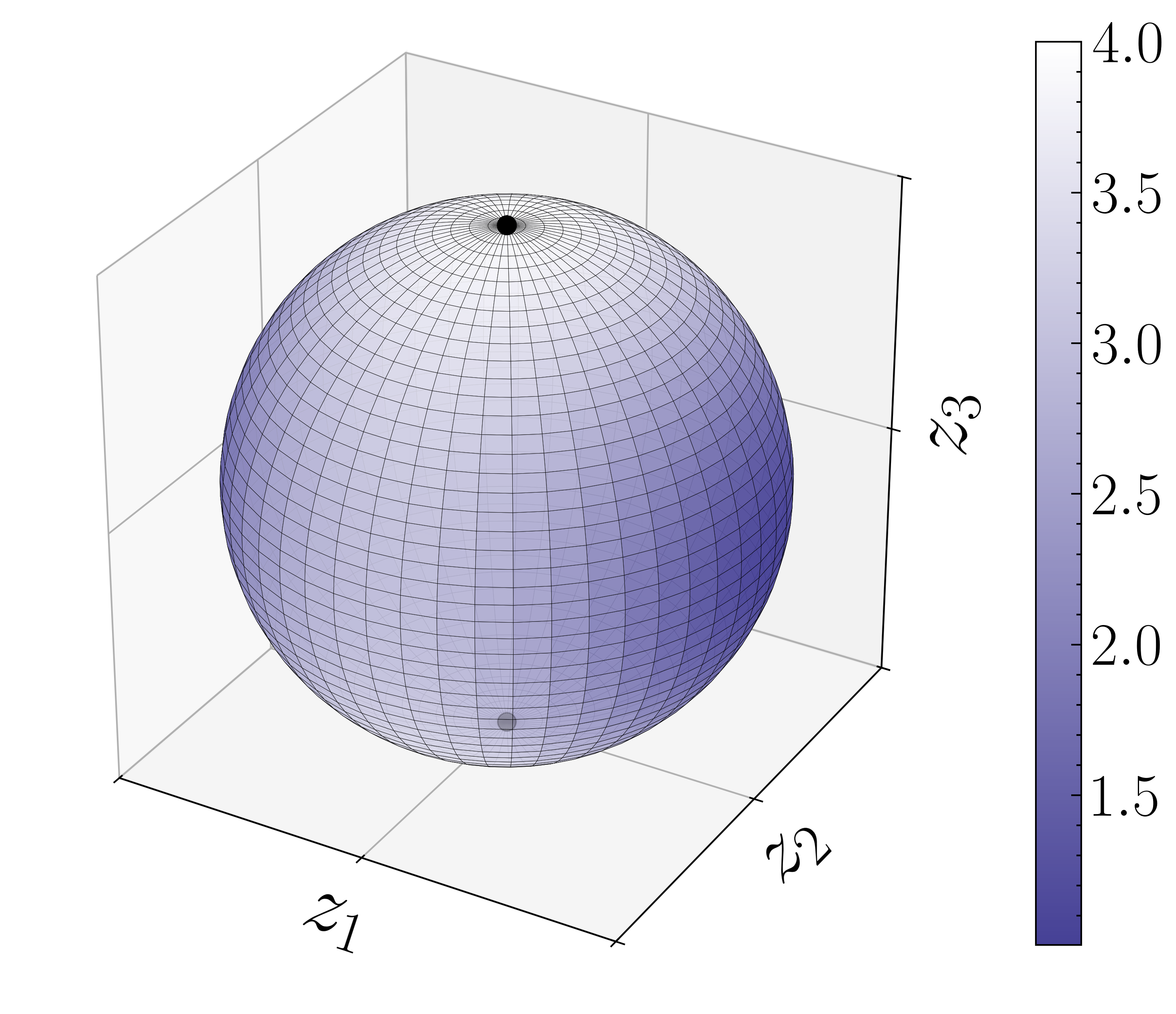}\caption{For $D=\text{diag}(1, 3, 4)$ the final state is either $(0,0,1)$ or $(0,0,-1)$.}\label{fig:maxpert}
\end{subfigure}
\caption{Discrete maximizers on the sphere for $N=1$ particles. The color indicates the value of ${x\cdot Dx}$ at each point on the sphere.}\label{fig:max}
\end{figure}

For multiple particle systems with $N>1$, \cref{lem:statcomb} suggests that also linear combinations of an eigenvector with its negative are stationary points. These linear combinations are not maximizers, but their basin of attraction depends on the eigenvalues of the matrix. In \cref{fig:clusters} (left) we plot the probability (i.e. the proportion of random initializations) of converging to a single cluster vs. two clusters as function of the eigenvalues. We fix $\lambda_1=1$ and vary $\lambda_2$ between $1$ and $1.5$. Note that, as discussed in \cref{lem:symmetricpeaks,rem:peaks-magnitude}, the actual values of the eigenvalues matter and not just their ratio. For $\lambda_2 \sim 1$ the probability of converging to a single cluster is high, whereas for larger values $\lambda_2 \gtrsim 1.4$ most trajectories converge to two clusters. \color{black} The results in \cref{fig:clusters} were obtained with the  adaptive normalization \cref{eq:normpart}; however, we observed the same quantitative behavior with the constant normalization \cref{eq:normconst}.\color{black}

\begin{figure}[t]
\begin{subfigure}[b]{.55\textwidth}
\begin{tikzpicture}
\begin{axis}[
        ybar stacked,
        ylabel={Count},
        xlabel={largest eigenvalue $\lambda_2$},
        ymin=0,
        legend style={at={(0.5,1.10)},
        anchor=north,legend columns=-1},
        width=\textwidth,
        height=.8\textwidth,
        bar width=0.1cm,
        axis x line*=bottom,
        axis y line*=left,
    ]  
    \addplot +[ybar]table [x index=0,y index=1] {results/max_split/CC.txt};
    \addplot +[ybar,     
            postaction={
            pattern=north west lines,
            pattern color = red
            }
    ]table [x index=0,y index=2] {results/max_split/CC.txt};
    \legend{Single Cluster, Two Clusters}
\end{axis}
\end{tikzpicture}
\end{subfigure}\hfill%
\begin{subfigure}[b]{.45\textwidth}
Single cluster (maximizers):\\
\hbox to \textwidth{\hfill%
\includegraphics[width=.33\textwidth]{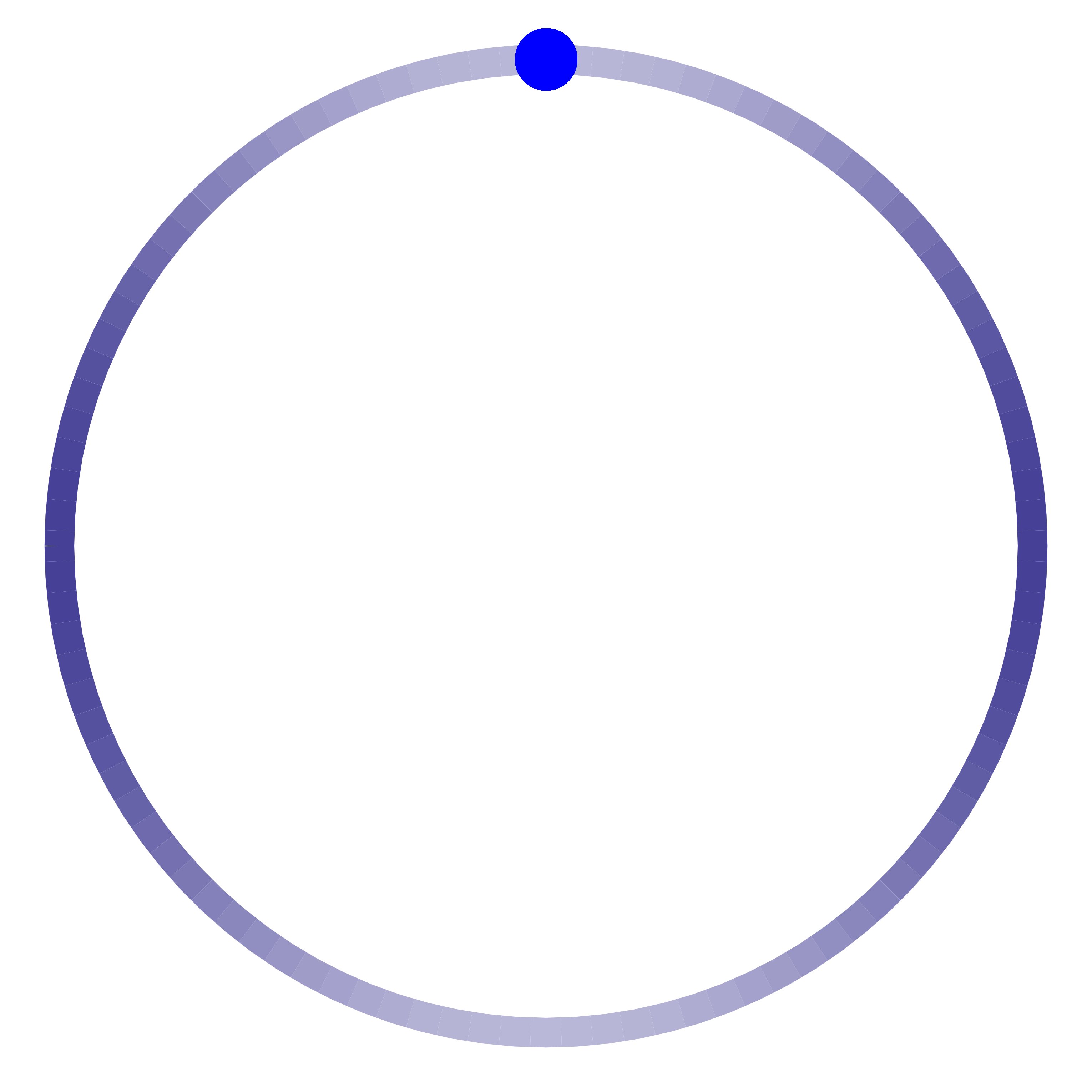}\hfill%
\includegraphics[width=.33\textwidth]{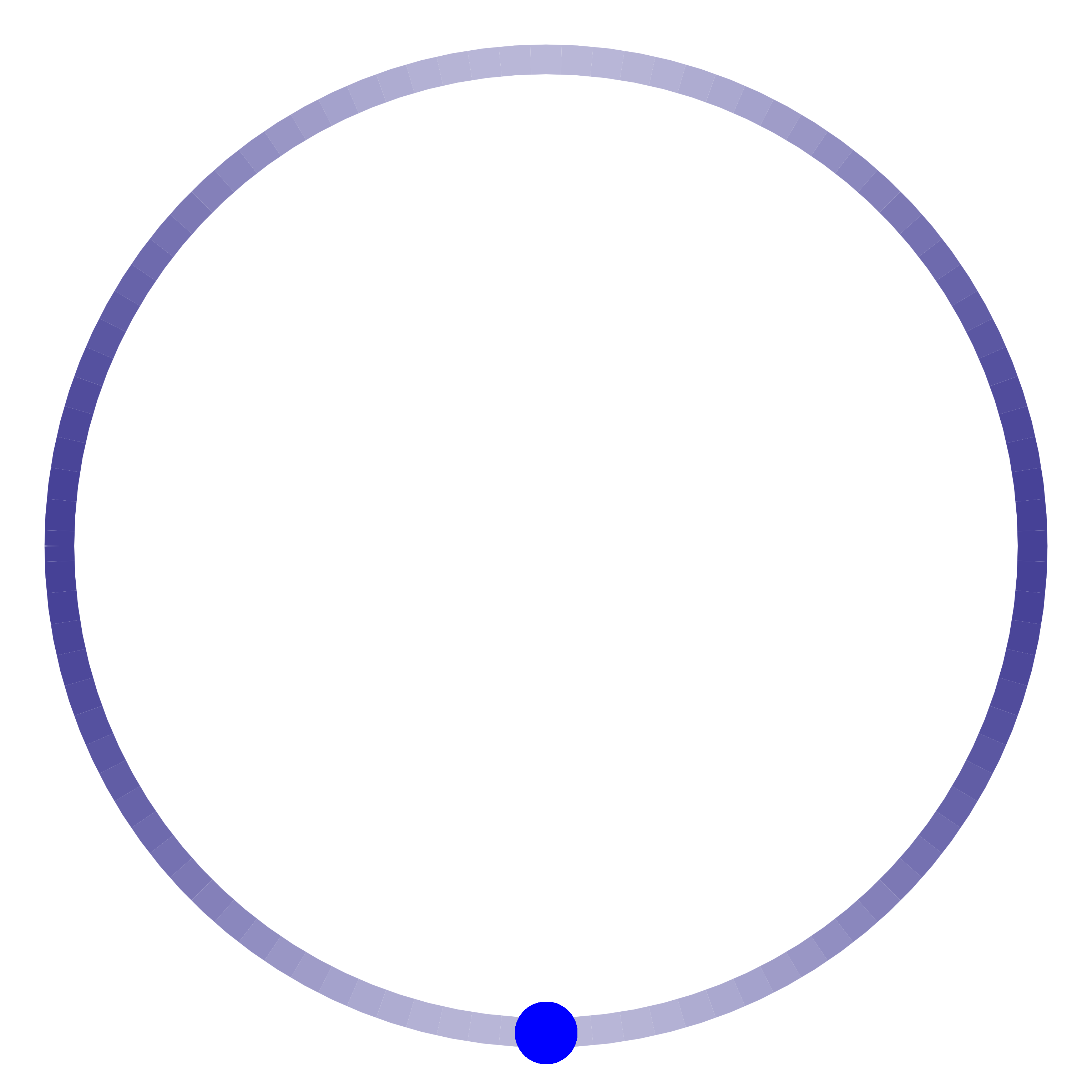}\hfill%
}\\
Two clusters (suboptimal stationary point):\\
\hbox to \textwidth{\hfill%
\includegraphics[width=.33\textwidth]{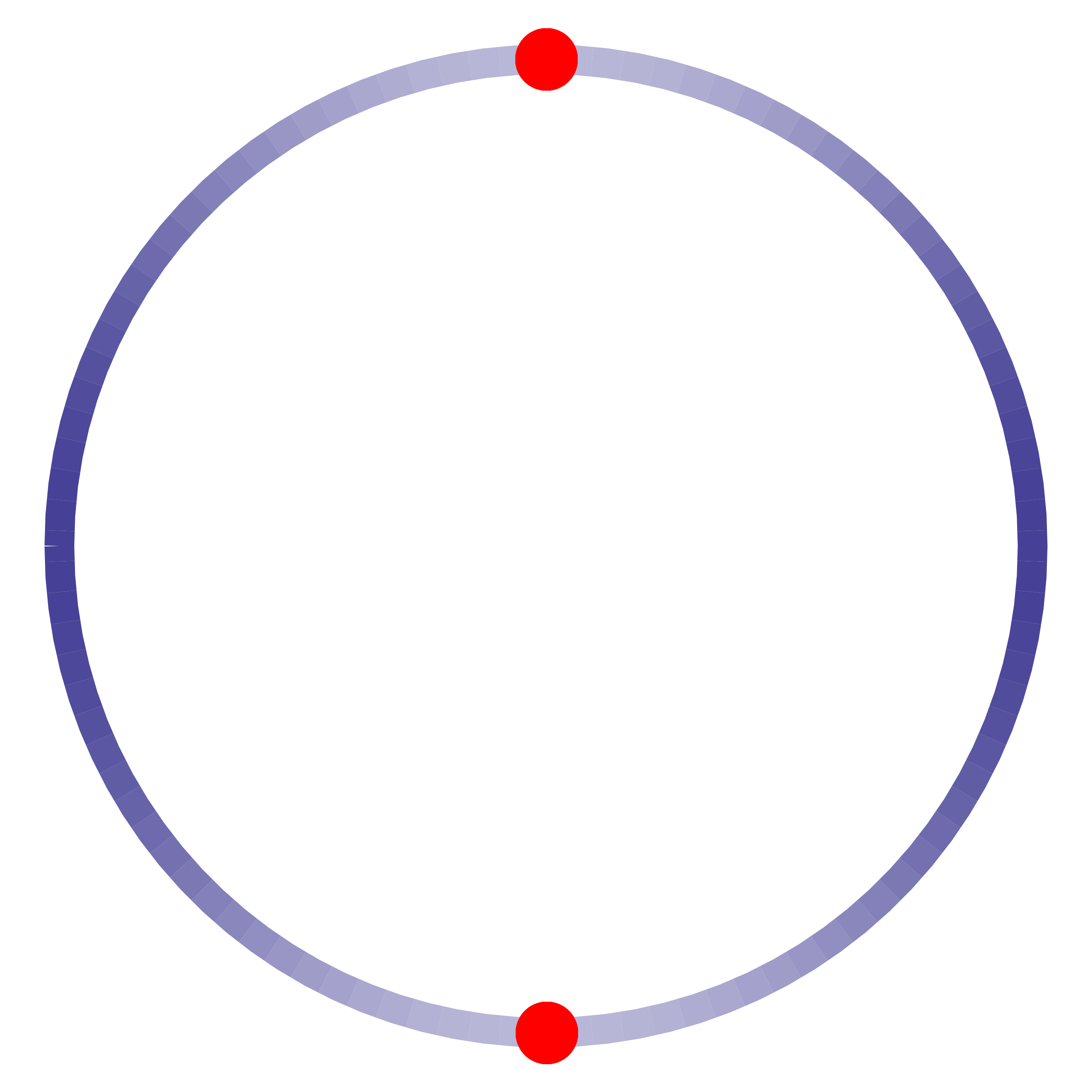}\hfill}%
\vspace{3em}
\end{subfigure}%
\caption{We study the trajectories for a symmetric positive definite matrix $D=\text{diag}(1, \lambda_2)$ with $\lambda_2 \in [1.,1.5]$ and $100$ different initializations using $100$ particles. We evaluate the number of clusters at the final iteration with the $k$-means implementation of the \texttt{SciPy} package \cite{2020SciPy-NMeth}. The center of each cluster is close to an eigenvector corresponding to an eigenvalue of maximal absolute value. For $\lambda_2 \approx 1$, the evolution converges to the optimal state with a single cluster (blue, solid), while for bigger values it tends to get stuck in the suboptimal stationary state with two clusters (red, hatched) from \cref{lem:statcomb}.}
\label{fig:clusters}
\end{figure}
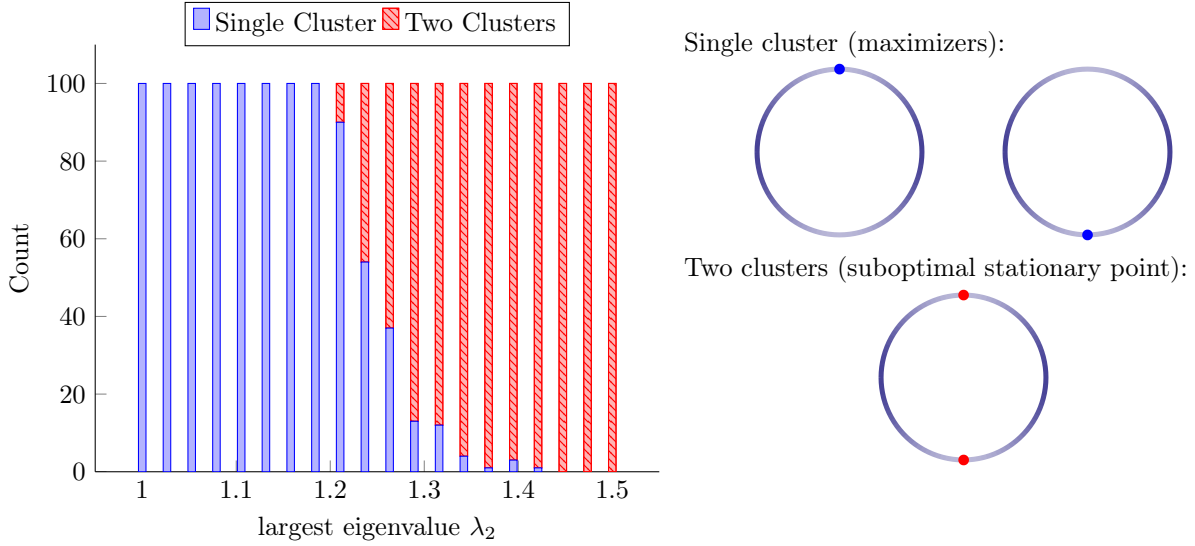
%
%
%
\subsection{Minimizers for positive (semi-) definite matrices}\label{sec:num_minposdef}

We now study discrete minimizers for positive definite matrices. In  \cref{fig:min} we show how the matrix $D$ influences the particle configuration which the scheme in \cref{eq:discusa} converges to. \color{black} Here, too, we used the  adaptive normalization \cref{eq:normpart}; the results for the constant one \cref{eq:normconst} are largely the same. \color{black}

\begin{figure}%
\begin{subfigure}[t]{.24\textwidth}%
\centering%
\includegraphics[width=\textwidth,trim={.9cm 0cm 0cm 0cm},clip]{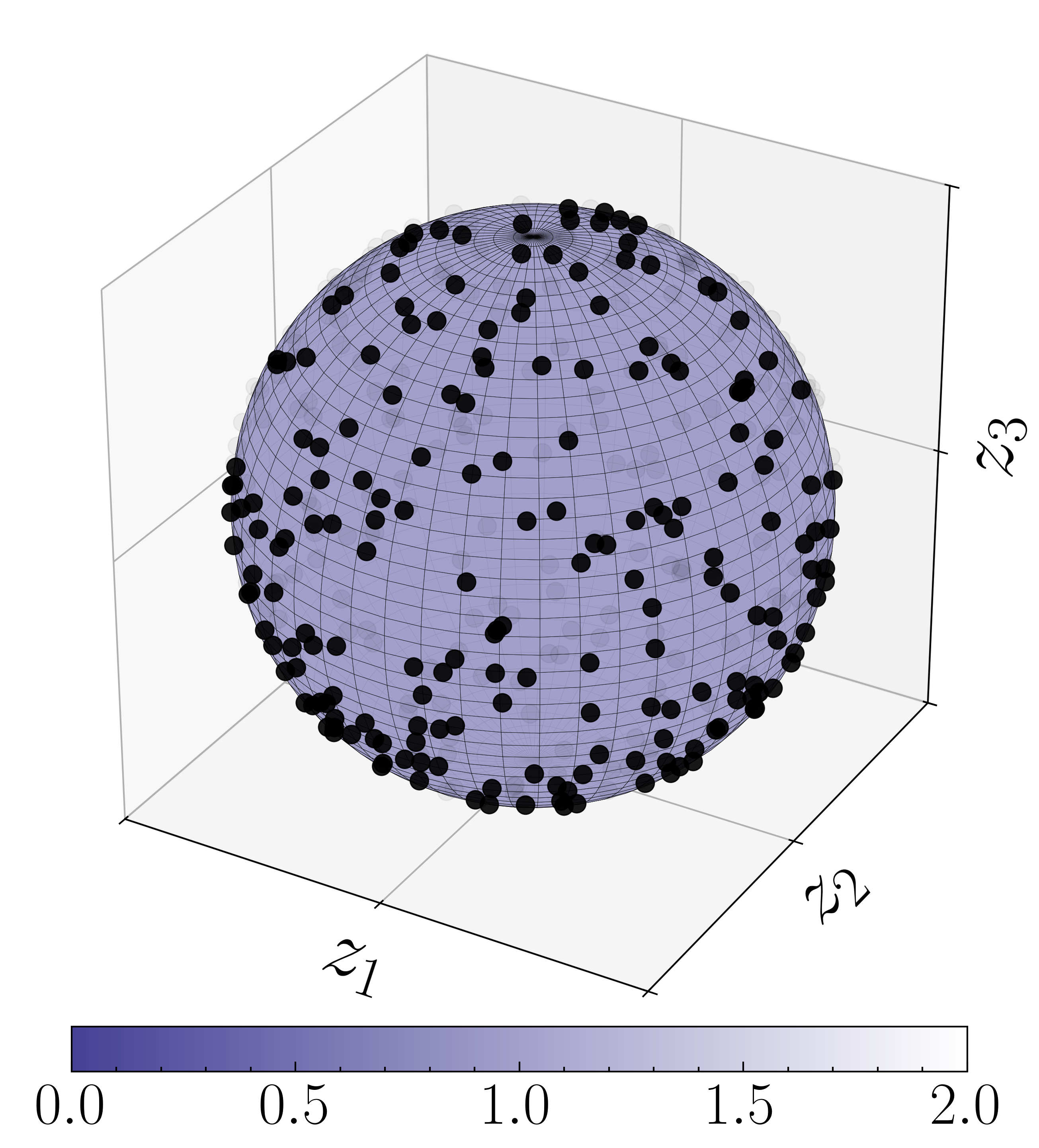}\captionsetup{width=.9\linewidth}%
\caption{$D=\Id$.
}
\end{subfigure}\hfill%
\begin{subfigure}[t]{.24\textwidth}%
\centering%
\includegraphics[width=\textwidth,trim={.9cm 0cm 0cm 0cm},clip]{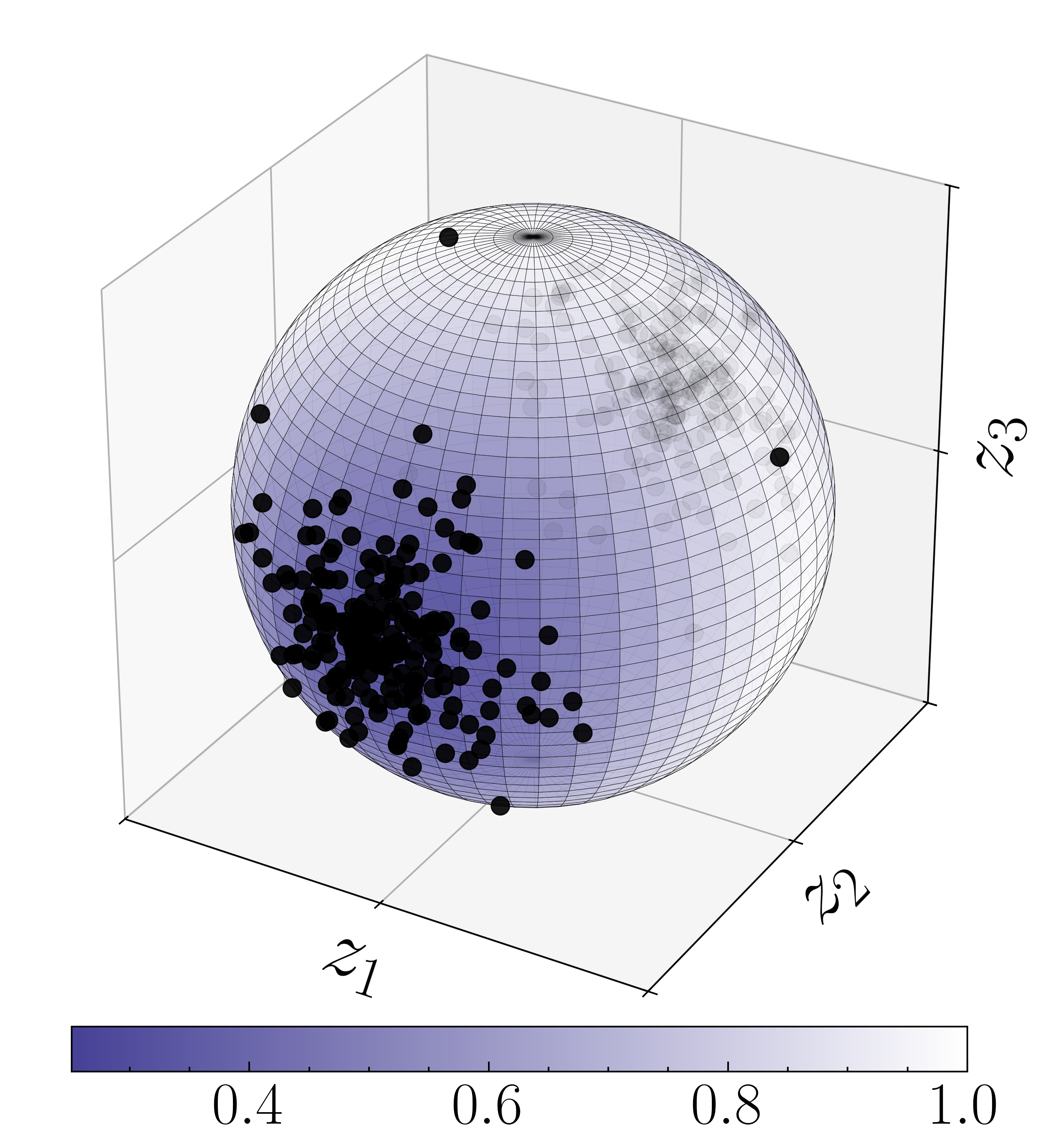}
\caption{$D=\text{diag}(1, 0.25, 1)$}
\end{subfigure}\hfill%
\begin{subfigure}[t]{.24\textwidth}%
\centering%
\includegraphics[width=\textwidth,trim={.9cm 0cm 0cm 0cm},clip]{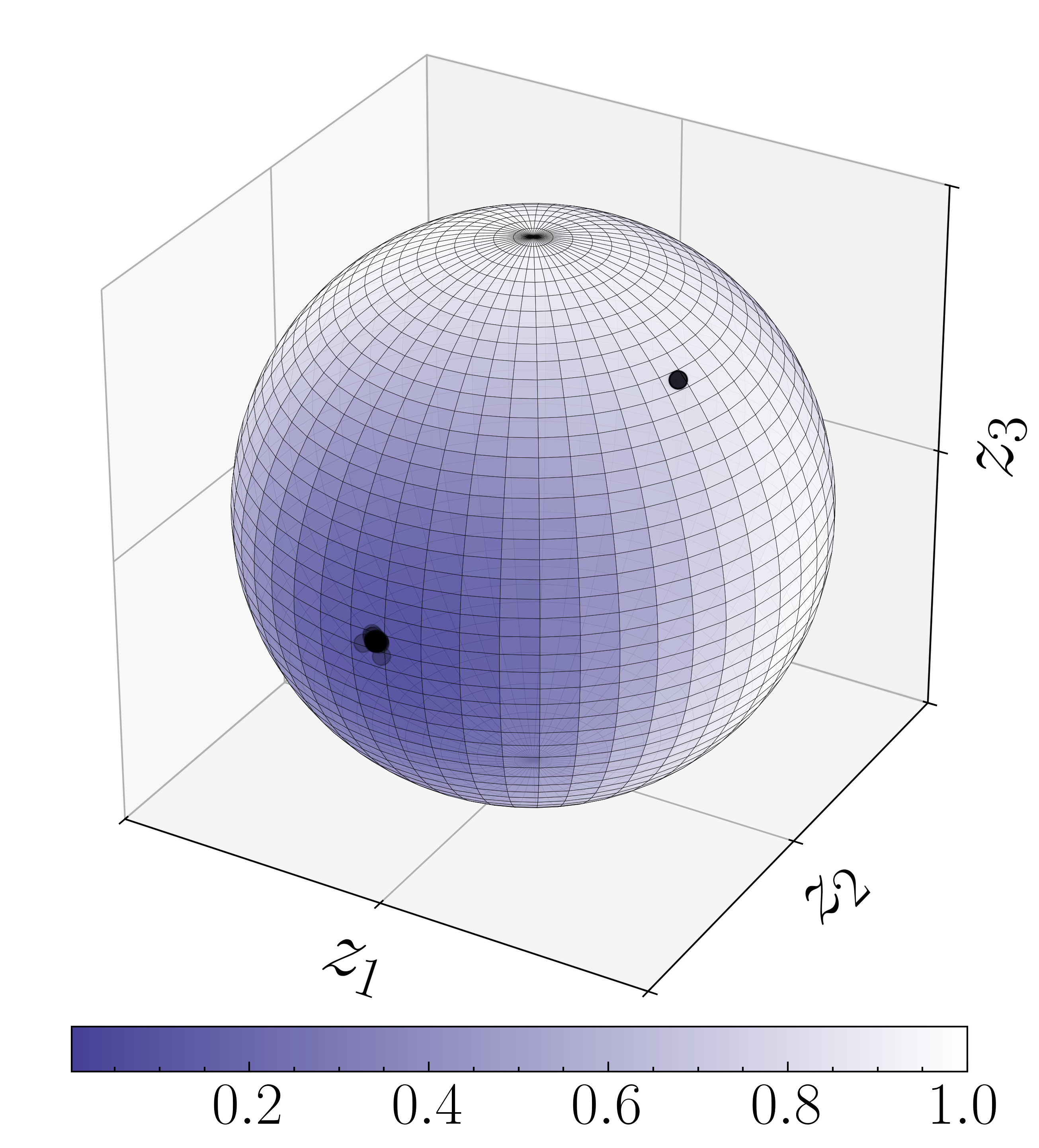}
\captionsetup{width=.9\linewidth}%
\caption{$D=\text{diag}(1, 0, 1)$}
\end{subfigure}\hfill%
\begin{subfigure}[t]{.24\textwidth}%
\centering%
\includegraphics[width=\textwidth,trim={.9cm 0cm 0cm 0cm},clip]{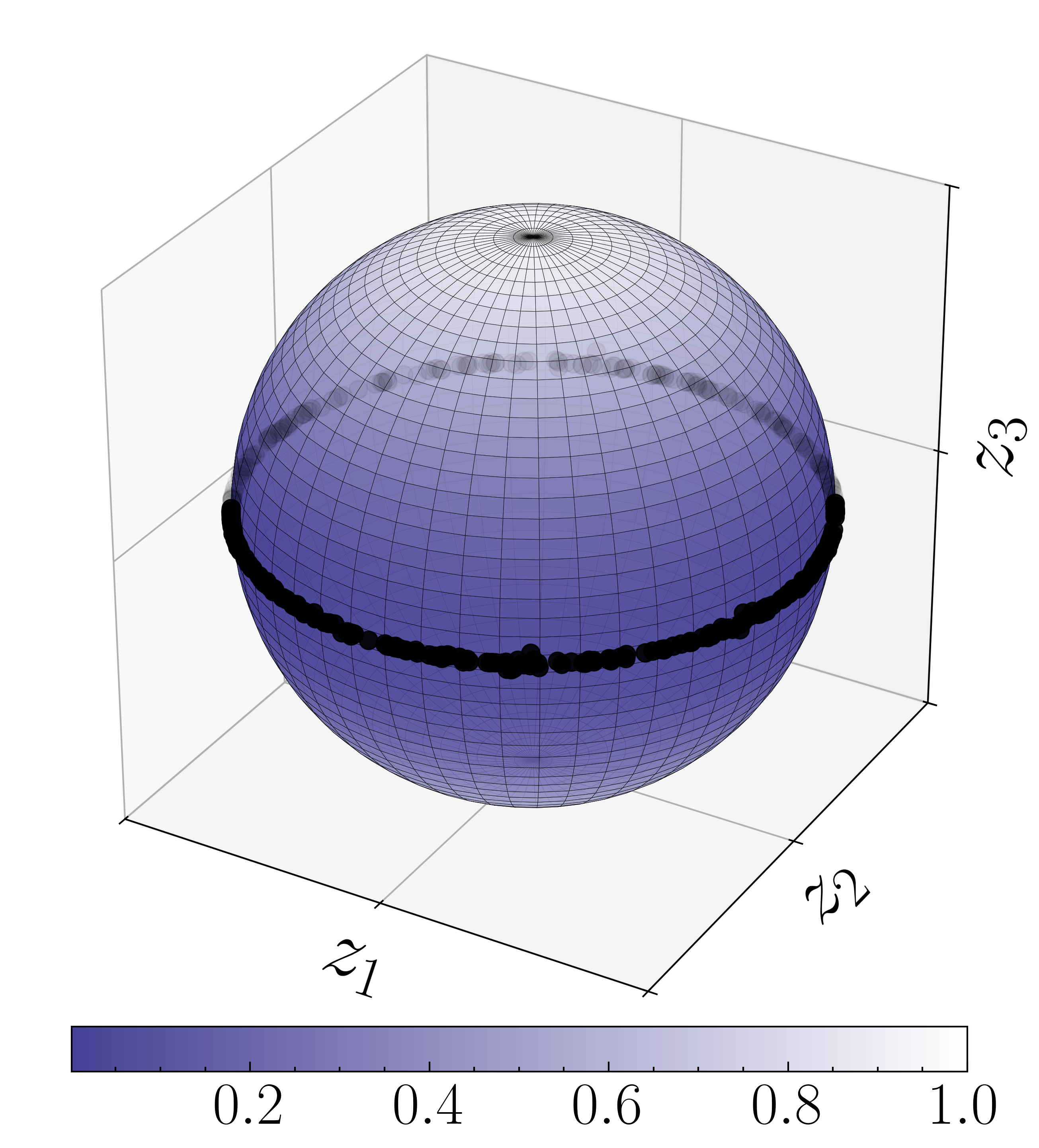}\captionsetup{width=.9\linewidth}\caption{$D=\text{diag}(0,0,1)$}
\end{subfigure}
\caption{Final states for the minimization scheme after $10000$ steps with $N=400$ particles. The color indicates the value of $\dpr{x}{Dx}$ at each point on the sphere. In (a) the uniform distribution is the minimizer of the energy. In (b)  the particles do not form clusters at single Diracs but rather follow a smooth distribution on the sphere. In (c) any configuration with $(X_i)_1 = (X_i)_3 =0$ for all~$i$ is a minimizer. In (d) any configuration with $(X_i)_3 =0$ for all~$i$ is a minimizer.
}\label{fig:min}
\end{figure}%
Furthermore, in \cref{fig:fourpart} we illustrate the results of \cref{lem:symmetricpeaks} for matrices $D=\text{diag}(1, \lambda_2)$ with varying values $\lambda_2\in[0.5,8]$. We initialize $N=4$ particles as
\begin{align}\label{eq:fourpartinit}
X_i = X(\phi_i) \qquad \text{with} \quad \phi_i = (i-1)\cdot \pi + \pi/4
\qquad \text{for} \quad i=1,\ldots,4
\end{align}
and  let the scheme \cref{eq:discusa} run for $10000$ iterations. From the final particle state, we compute the value $\tanh(\cos^2\phi_i)/\tanh(\lambda_2 \sin^2\phi_i)$ for each particle separately; \cref{lem:symmetricpeaks} tell us that this should be equal to $\lambda_2$ for the minimizer. In \cref{fig:fourpart} we observe that this holds true for the particle configurations computed with the discrete scheme. However, if the step size is too big compared to the value $\lambda_2$, the system instead converges to the two-cluster stationary point from \cref{fig:clusters}. \color{black} Here, we notice a slight difference between the two normalizations. The  adaptive normalization \cref{eq:normpart} allows choosing bigger step sizes compared to the constant normalization \cref{eq:normconst}, enabling faster convergence to the large-time limit.\color{black}
\begin{figure}
\centering%
\begin{tikzpicture}
\begin{axis}[width=.82\textwidth,
height=.5\textwidth,grid=major,
legend pos=outer north east,
legend cell align={left}, xlabel=$\lambda_2$,
ylabel=$\tanh(\cos^2\phi_i)/\tanh(\lambda_2 \sin^2\phi_i)$,
axis x line*=bottom,
axis y line*=left,
x label style={anchor=west},
y label style={anchor=south},
xmin=-1,
name=aa
]
%
%
\addplot[line width=1.5pt, loosely dashed, color=black,opacity=1., mark=*,mark options={solid}]  table[x index=0, y index=1,restrict x to domain=0:6]{results/tanh_min/tanh_minusa0.2.csv};
\label{e1}
\addplot[line width=1.5pt, dotted, color=red,opacity=1.,mark=*, mark options={solid}]  table[x index=0, y index=1,restrict x to domain=0:8]{results/tanh_min/tanh_minsa0.35.csv};
\label{e2}
\addplot[line width=1.5pt, color=blue,opacity=1., mark=*]  table[x index=0, y index=1,restrict x to domain=0:8]{results/tanh_min/tanh_minsa0.2.csv};
\label{e3}
\tikzstyle{every pin}=[
    draw=blue,
    fill=white,
    thick
]
\tikzset{every pin edge/.style={draw=blue, thick}}
\node [
coordinate,
pin={[pin distance=1cm]above:{\includegraphics[width=1.2cm]{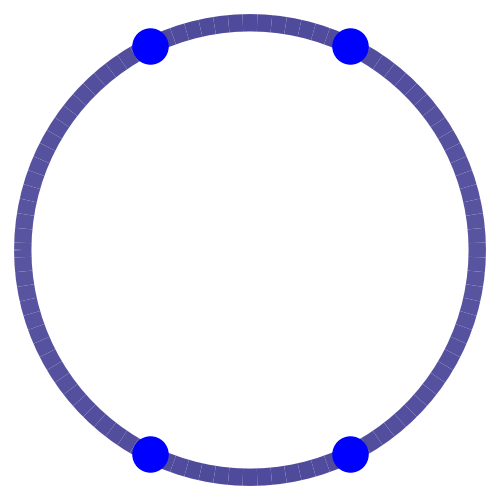}}},
] at (axis cs:0.5,0.5)   {};
\node [
coordinate,
pin={[pin distance=2cm]above:{\includegraphics[width=1.2cm]{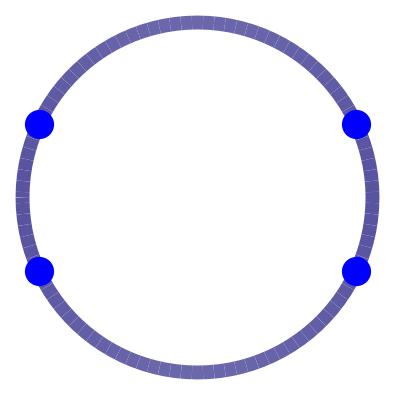}}},
] at (axis cs:2,2)   {};
\node [
coordinate,
pin={[pin distance=.5cm]90:{\includegraphics[width=1.2cm]{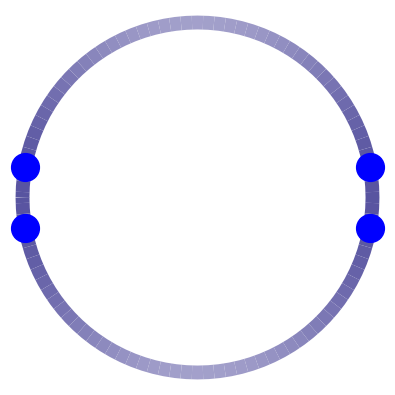}}},
] at (axis cs:5,5)   {};
\node [
coordinate,
pin={[pin distance=.6cm]below:{\includegraphics[width=1.2cm]{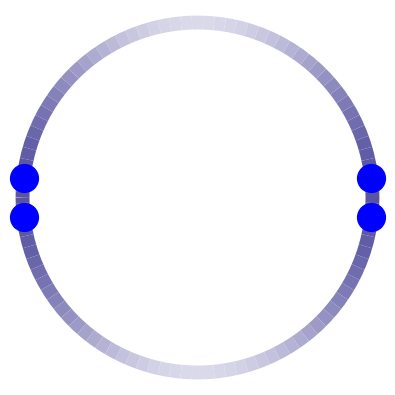}}},
] at (axis cs:8,8)   {};
\node [
coordinate,
pin={[pin distance=.5cm, draw=red, pin edge={red, thick}]above:{\includegraphics[width=1.2cm]{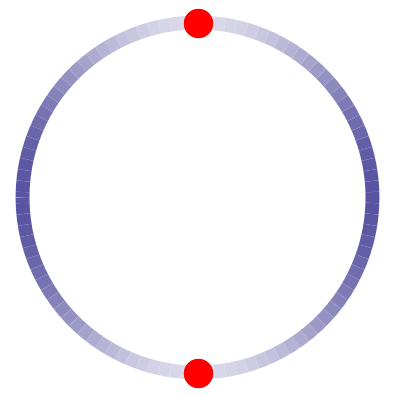}}},
] at (axis cs:8,0)   {};
\end{axis}
\matrix [
 matrix of nodes,
 nodes={anchor=west},
 anchor=north west,
 at={([shift={(2pt,-2pt)}]aa.north east)},
 fill=white,
 draw,
 inner sep=2pt,
 row sep=2pt
] {
\ref{e1} \cref{eq:normconst}, $\tau=0.2$ \\
\ref{e3} \cref{eq:normpart}, $\tau=0.2$ \\
\ref{e2} \cref{eq:normpart}, $\tau=0.35$ \\
};
\end{tikzpicture}
\caption{We consider minimizers for the matrix $D=\text{diag}(1,\lambda_2)$. Starting with the initial configuration described in \cref{eq:fourpartinit} we compute the mean of $\tanh(\cos^2\phi_i)/\tanh(\lambda_2 \sin^2\phi_i)$ over all particles.
For a small step size, the resulting curve is very close to the identity, as predicted by \cref{lem:symmetricpeaks}. If $\lambda_2\, \tau$ is too big, the dynamics  converge to a suboptimal stationary point. \color{black}{We also compare the normalizations~\eqref{eq:normconst} and~\eqref{eq:normpart}. We see that with the same step size $\tau=0.2$, the adaptive normalization~\eqref{eq:normconst} yields faster convergence than the constant one~\eqref{eq:normpart}.
}}\label{fig:fourpart}
\end{figure}

We further investigate the validity of the asymptotic solution from \cref{thm:perturbidentity} in the two-dimensional case. Here we deviate from the particle approximation and instead discretize the interval $[-\pi,\pi)$ with $N$ equidistant grid points $\Theta\in[-\pi,\pi]^N$ and the associated points on the sphere $x_1,\ldots,x_N\in\mathcal{S}^1$. In this setting, we then aim to minimize
\begin{align}\label{eq:mirror}
\tilde{E}_\varepsilon(m) = 
\sum_{i,j=1}^N e^{x_i\cdot (\Id + \varepsilon M)x_j} m_i\cdot m_j,
\end{align}
where $m\in\R^N$ is a probability vector. Note that already for $n=3$ a more sophisticated quadrature rule would be required, e.g. the Lebedev quadrature on the sphere \cite{marchuk1986numerical}. To deal with the simplex constraint for the vector $m$ we use exponentiated gradient descent, specifically mirror descent with 
the negative log-entropy as the distance generating function~\cite{kivinen1997exponentiated} which yields the update
\begin{align}\label{eq:scheme-m-eps}
m(\varepsilon)_i \gets \frac{m_i e^{-\tau \grad \tilde{E}_\varepsilon(m(\varepsilon))_i}}{\sum_{j=1}^N m(\varepsilon)_j e^{-\tau \grad \tilde{E}(m(\varepsilon))_j}} 
= 
\operatorname{SoftMax}(\log(m(\varepsilon)) - \tau \grad{\tilde{E}_\varepsilon}(m(\varepsilon))_i.
\end{align}

We take the perturbation matrix as $M=\text{diag}(0,1)$, that is, the perturbed matrix $D$ is given by 
$
D_\eps = \text{diag}(1,1+\eps).    
$
Recall the asymptotic expansion~\eqref{eq:expansionInt}. As noted in \cref{sec:perturabation-identity}, the contribution of the term $\eps^2\omega$ vanishes in the second-order expansion of the energy and we are left with a solution
\begin{equation}\label{eq:asympt-exp-numerics}
    \mu^*_\varepsilon=\mu_0 + \varepsilon \nu^*,
\end{equation}
where $\nu^*$ is as in \cref{thm:perturbidentity}. We note that this measure has a Lebesgue density that can be evaluated at the grid points in $\Theta$; we denote the resulting vector  by $\d\mu^*_\varepsilon\vert_{\Theta}$.
In \cref{fig:densityall} we compare this solution to the vector $m(\varepsilon)$ obtained by solving~\eqref{eq:mirror}--\eqref{eq:scheme-m-eps}. 
The vector $m(\varepsilon)$ for different values of $\varepsilon$ is shown in \cref{fig:density} and in \cref{fig:diffdensity} we plot the $\ell^2$ error 
%
$\abs{m(\epsilon) - d\mu^*_\varepsilon\vert_{\Theta}}_2$.  

%
%
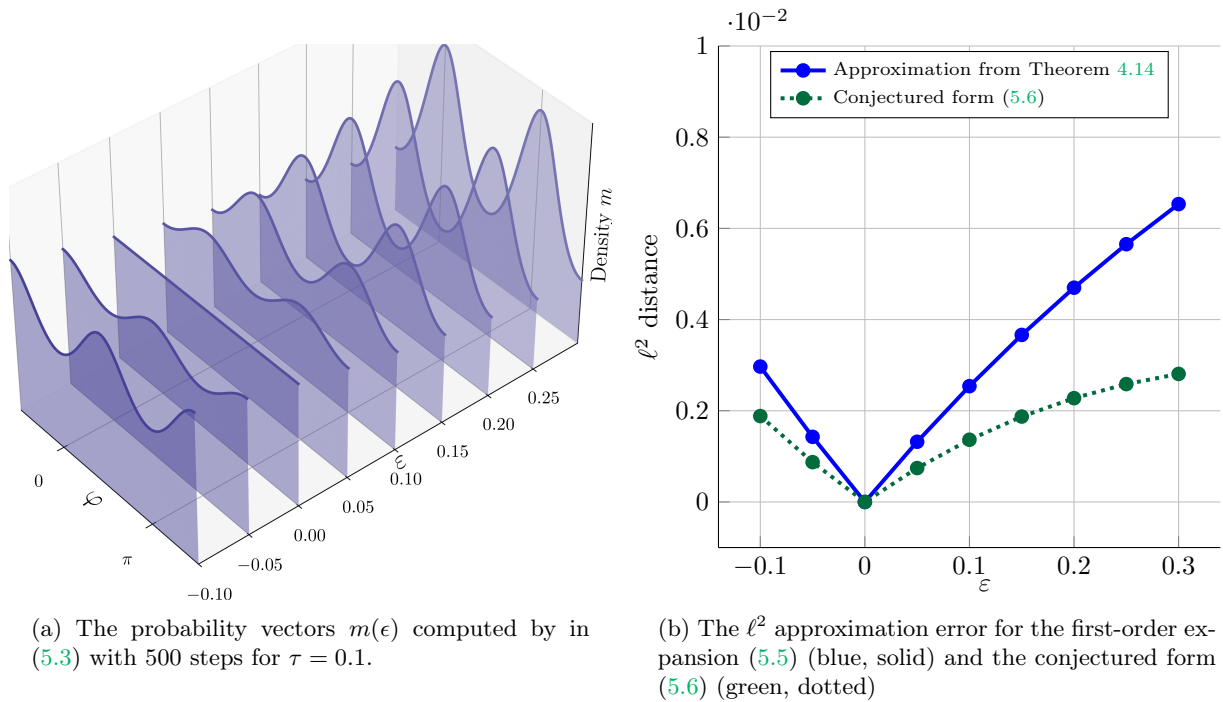
\begin{figure}
\begin{subfigure}[t]{.5\textwidth}%
\centering%
\includegraphics[width=.97\textwidth,
trim={.7cm .2cm .05cm 1.7cm}, clip]{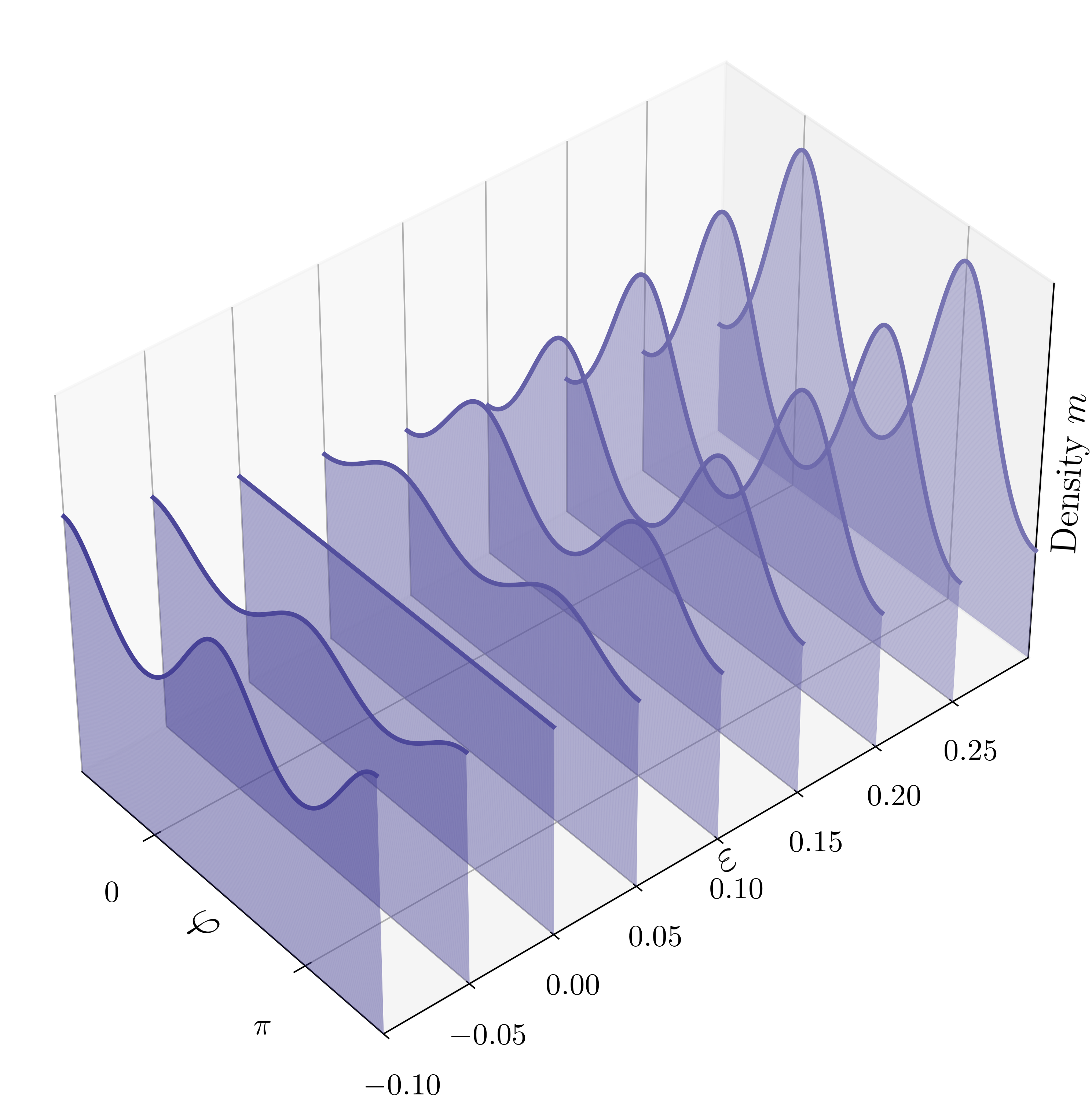}%
\captionsetup{width=.9\linewidth}%
\caption{The probability vectors $m(\epsilon)$ computed by in \cref{eq:mirror} with $500$ steps for $\tau=0.1$.}\label{fig:density}
\end{subfigure}
\begin{subfigure}[t]{.5\textwidth}
\begin{tikzpicture}
\begin{axis}[width=\textwidth,
height=\textwidth,grid=major,
legend pos=outer north east,
legend cell align={left}, xlabel=$\varepsilon$,
ylabel=$\ell^2$ distance,
axis x line*=bottom,
axis y line*=left,
x label style={anchor=west},
y label style={anchor=south},
legend style={
    at={(0.5,0.99)},
    anchor=north,legend columns=1,
    font=\scriptsize
    },
ymax=0.01
]
\addplot[line width=1.5pt, color=blue,opacity=1.,mark=*]  table[x index=0, y index=1]{results/densities/diffs.csv};
\addlegendentry{Approximation from \cref{thm:perturbidentity}}
\addplot[line width=1.5pt, color=cadmiumgreen, dotted,opacity=1.,mark=*,mark options={solid}]  table[x index=0, y index=2]{results/densities/diffs.csv};%
\addlegendentry{Conjectured form \cref{eq:guess}}
\end{axis}
\end{tikzpicture}%
\captionsetup{width=.9\linewidth}%
\caption{The $\ell^2$ approximation error  for the first-order expansion \eqref{eq:asympt-exp-numerics} (blue, solid) and the conjectured form \eqref{eq:guess} (green, dotted)}\label{fig:diffdensity}
\end{subfigure}
\caption{Numerical study of the asymptotic solution from \cref{thm:perturbidentity} in two dimensions}\label{fig:densityall}
\end{figure}

Beyond the first-order expansion~\eqref{eq:asympt-exp-numerics}, we conjecture that  $m(\varepsilon)$ behaves as follows
%
\begin{align}\label{eq:guess}
\d\mu^{\text{guess}}_\varepsilon(\theta) \sim \exp(\Upsilon(\varepsilon)\cos(2\theta)),
\end{align}
where $\Upsilon(\varepsilon)$ is a function to be determined. Taking a second-order Taylor expansion $\Upsilon(\varepsilon)$, we estimate the coefficients via  linear regression with the  given  vectors $m(\varepsilon)$ as data points and obtain $\Upsilon(\varepsilon) \approx 1/5\, \varepsilon^2 + e/2\, \varepsilon$. 
The $\ell^2$ error of this approximation is shown in \cref{fig:diffdensity} and is lower than that of the first-order expansion~\eqref{eq:asympt-exp-numerics}. We leave the analysis of this ansatz to future work.%

%
%
%
\subsection{Maximizers for negative definite and indefinite matrices}\label{sec:num_diracmaxis_neg}
We proceed to numerical examples for  \cref{sec:maxnd}, i.e. maximization of the energy corresponding to a negative definite matrix. We take a system of $N=100$ particles and consider the two matrices from \cref{fig:max} multiplied by $-1$. The results are shown in \cref{fig:maxnd}. We observe that a single final state consists of clusters at $\pm z$, where $z$ is an eigenvector corresponding to the smallest eigenvalue, in agreement with \cref{thm:maxnd}. As shown there, the behavior does not change if one of the eigenvalues is zero, as only the eigenvectors corresponding to the smallest eigenvalue are relevant. 
For this reason, we don't consider the semi-definite case separately. \color{black} The results here are not impacted by the choice of the normalization; we only show the ones obtained with \cref{eq:normpart}.
\color{black}
\begin{figure}
\begin{subfigure}[t]{.47\textwidth}%
\centering%
\includegraphics[width=.8\textwidth,trim={.9cm 0cm 0cm 0cm},clip]{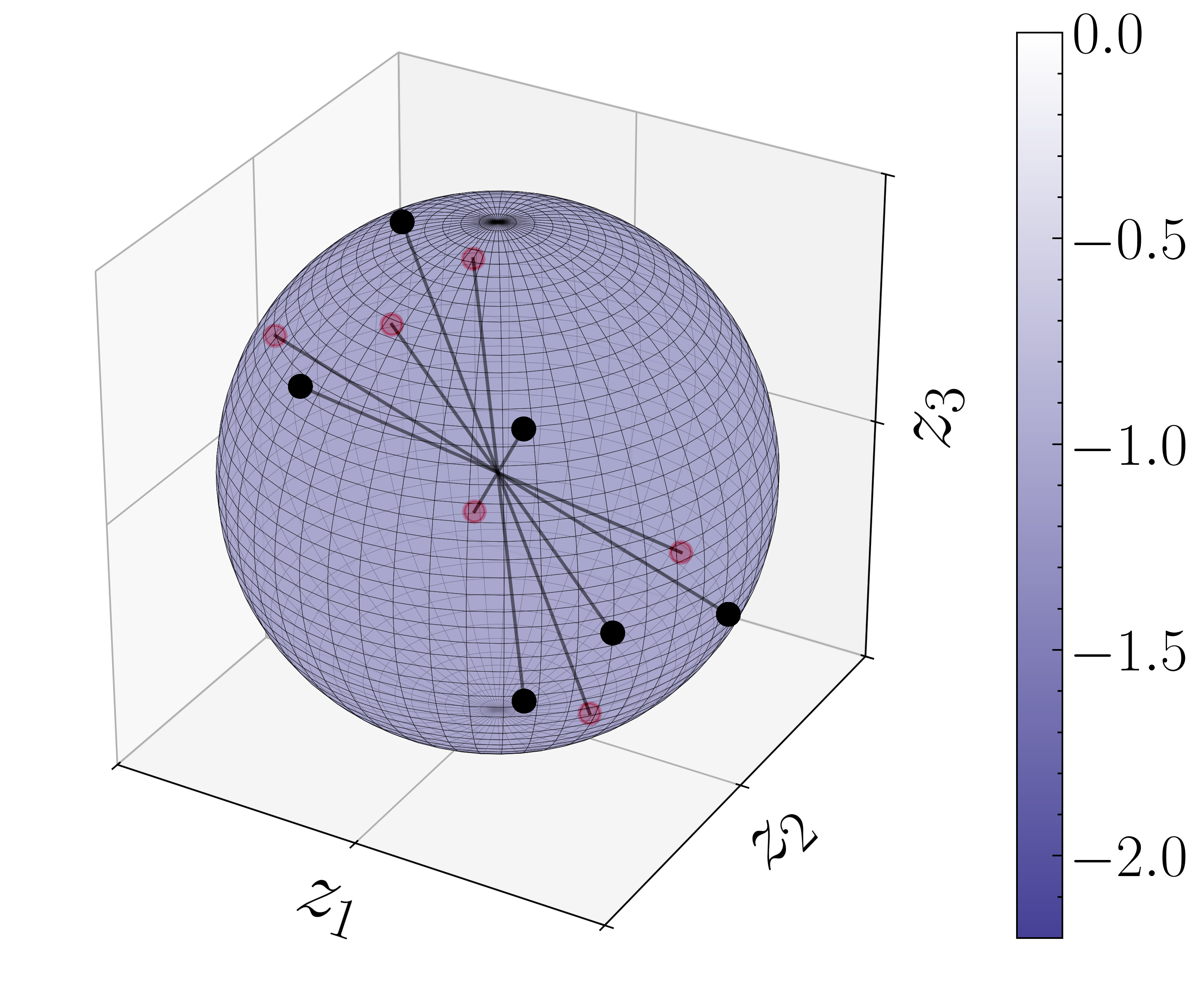}\caption{
For $D=-\Id$ a single final state has clusters at both $z$ \textbf{and} $-z$ for any $z\in\S$. For clarity, we only show results for $6$ different initializations.}
\end{subfigure}\hfill%
\begin{subfigure}[t]{.47\textwidth}%
\centering%
\includegraphics[width=.8\textwidth,trim={.9cm 0cm 0cm 0cm},clip]{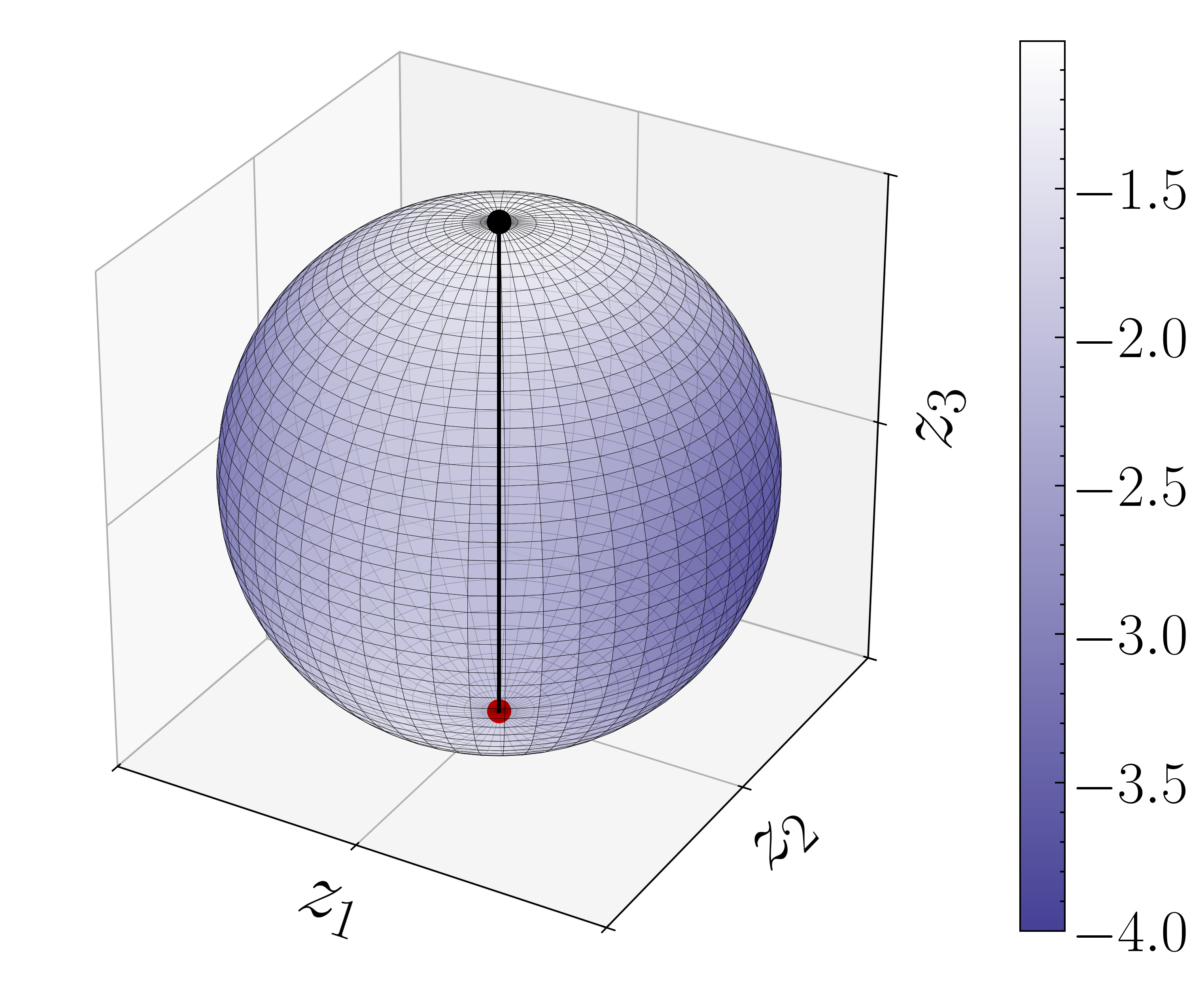}\caption{For $D=-\text{diag}(1, 3, 4)$ a single final state has clusters both at $(0,0,1)$ \textbf{and} $(0,0,-1)$. We show the results for $100$ different initializations.}
\end{subfigure}
\caption{Discrete maximizers on the sphere for negative definite matrices obtained with $N=100$ particles. We visualize the two-cluster final states by connecting the two components of each cluster corresponding to the same run with a line, assigning different colors to the two opposite clusters. The color of the sphere indicates the value of $\dpr{x}{Dx}$ at each point on the sphere.}\label{fig:maxnd}
\end{figure}
%
%
%
%
%
%
%
%
%

Finally, we turn to the case of indefinite matrices. As noted in \cref{rem:ndd}, for a matrix $D$ that is not negative definite, a Dirac delta placed at the eigenvector corresponding to the largest eigenvalue may not be a maximizer. This can be observed numerically as shown in \cref{fig:maxind} where we plot the energies of one- and two-cluster states for $D=\text{diag}(-1, \lambda_2)$ with $\lambda_2 \in [-1,1]$.
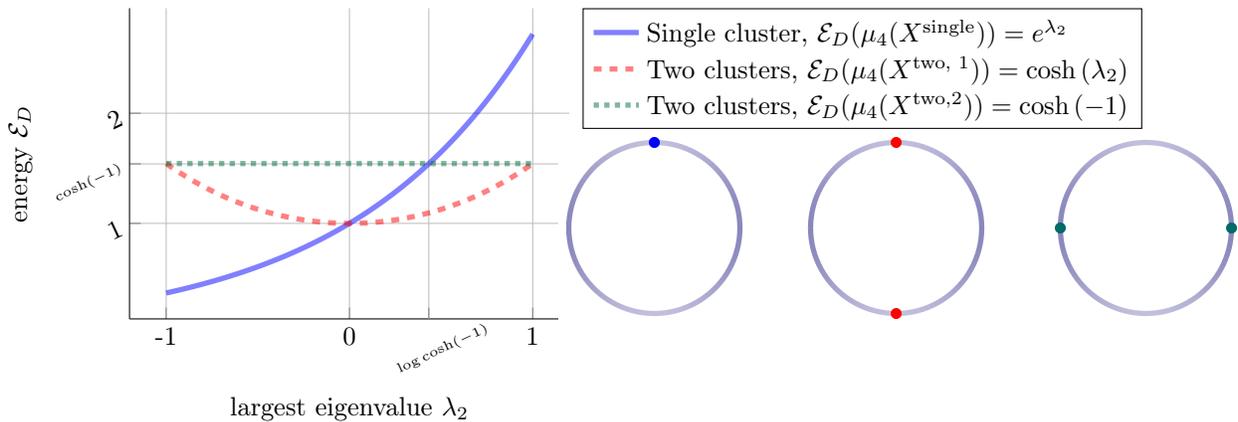
\begin{figure}
\centering%
\begin{subfigure}[b]{.45\textwidth}%
\begin{tikzpicture}
\begin{axis}[width=\textwidth,
height=.25\textheight,grid=major,
legend pos=outer north east,
legend cell align={left}, 
ytick={1, 1.54, 2},
yticklabels={1,\tiny$\cosh(-1)$,2},
y tick label style={rotate=25,anchor=east},
extra x tick style = {xticklabel style={rotate=25,anchor=north, yshift=-1.5ex}},
xtick={-1, 0, 1},
xticklabels={-1, 0, 1},
axis x line*=bottom,
axis y line*=left,
x label style={anchor=north},
y label style={anchor=south},
xlabel = {largest eigenvalue $\lambda_2$},
extra x ticks = {0.433},
extra x tick label = {\tiny$\log\cosh(-1)$},
ylabel = {energy $\erg$}
]
\addplot[line width=2pt, color=blue, opacity=0.5]  table[x index=0, y index=1, mark=none]{results/max_nd/maxid.csv};
\addlegendentry{Single cluster, $\erg(\mu_4(X^\text{single})) = e^{\lambda_2}$}
\addplot[line width=2pt, color=red,opacity=0.5, dashed]  table[x index=0, y index=3, mark=none]{results/max_nd/maxid.csv};
\addlegendentry{Two clusters, $\erg(\mu_4(X^{\text{two, 1}}))  = \cosh{(\lambda_2)}$}
\addplot[line width=2pt, color=cadmiumgreen, opacity=0.5, dotted]  table[x index=0, y index=2, mark=none]{results/max_nd/maxid.csv};
\addlegendentry{Two clusters, $\mathcal{E}_{D}(\mu_4(X^{\text{two},2}))  = \cosh{(-1)}$}
\end{axis}
\end{tikzpicture}
\end{subfigure}\hfill%
\begin{subfigure}[b]{.55\textwidth}%
\hbox to \textwidth{%
\includegraphics[width=.27\textwidth]{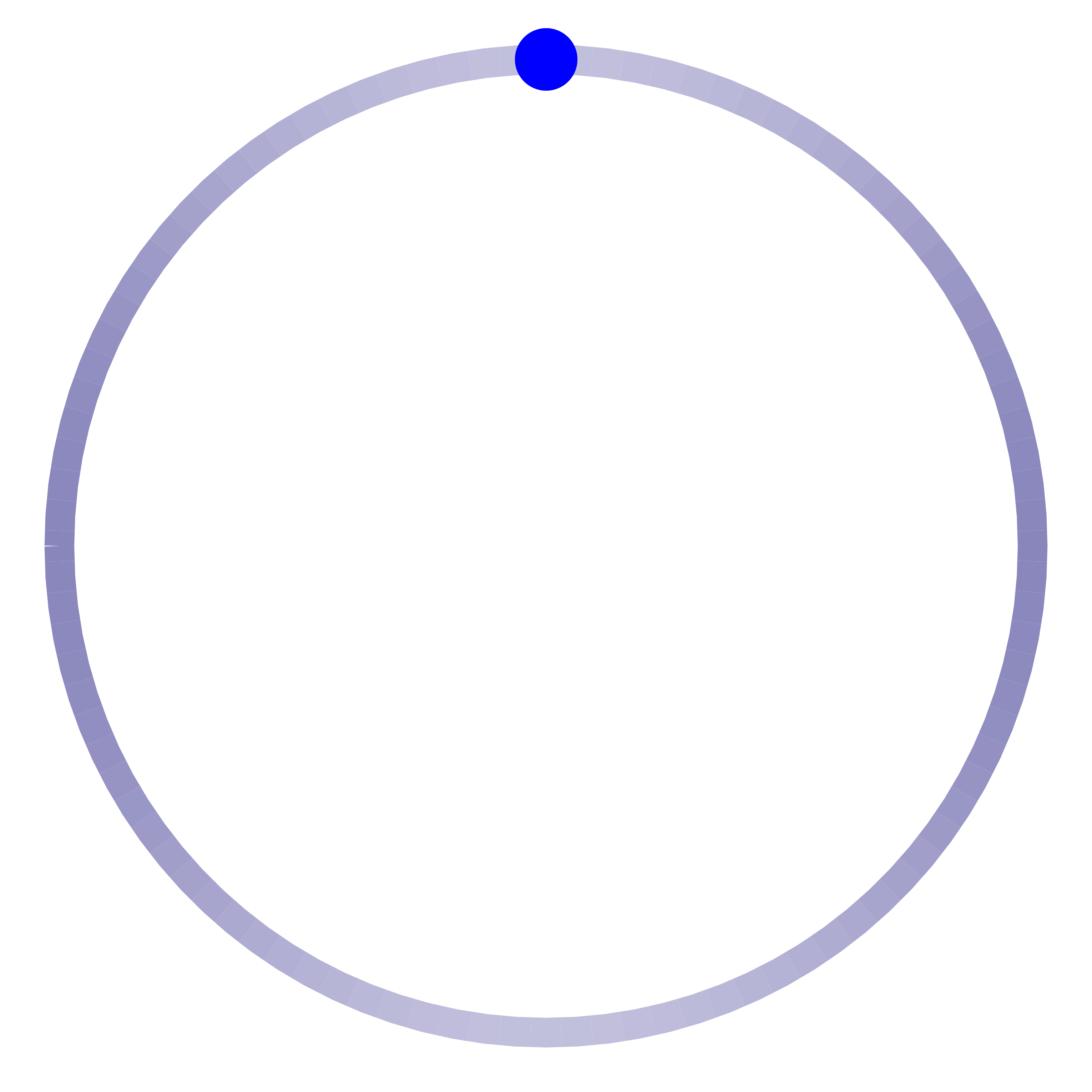}%
\hfill%
\includegraphics[width=.27\textwidth]{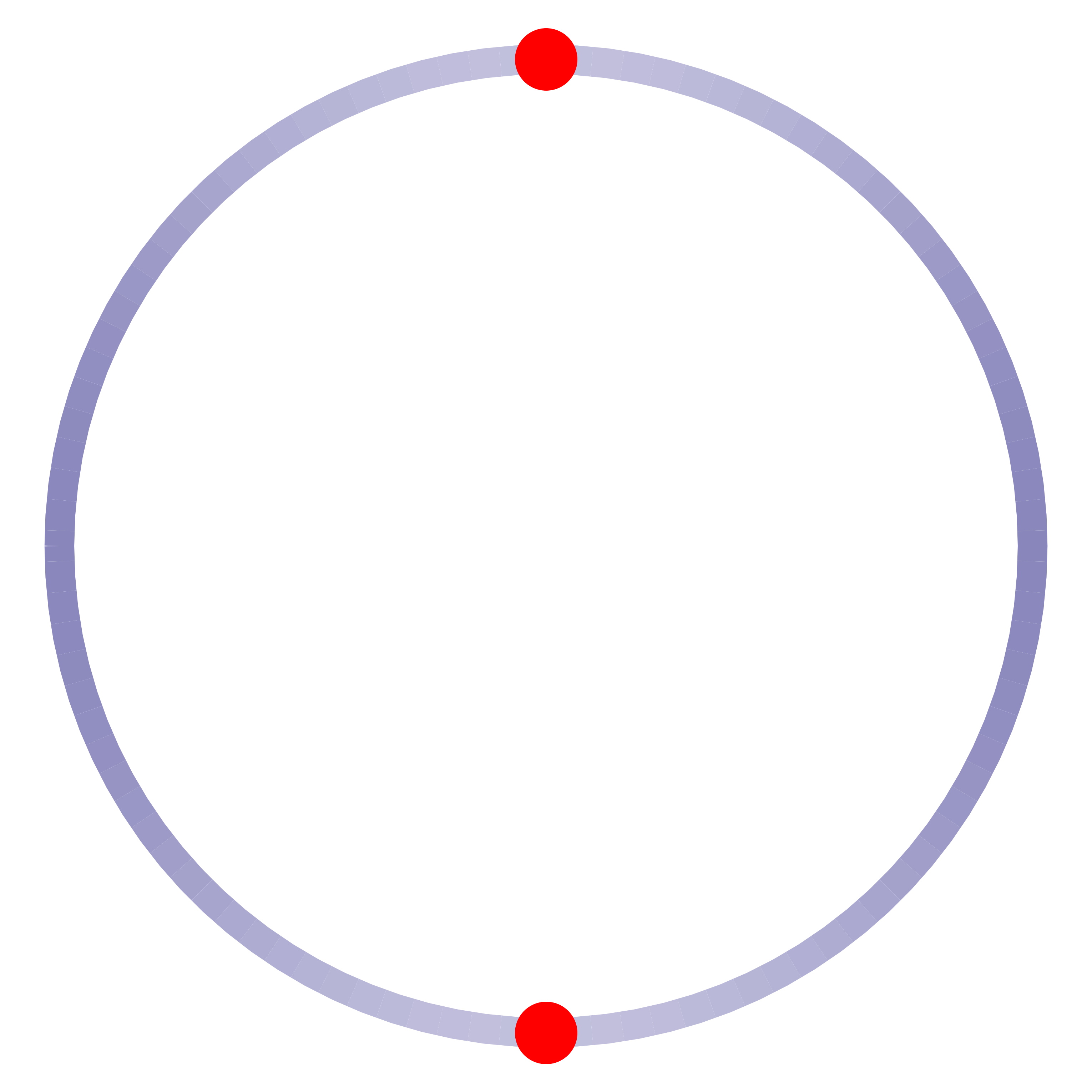}
\hfill%
\includegraphics[width=.27\textwidth]{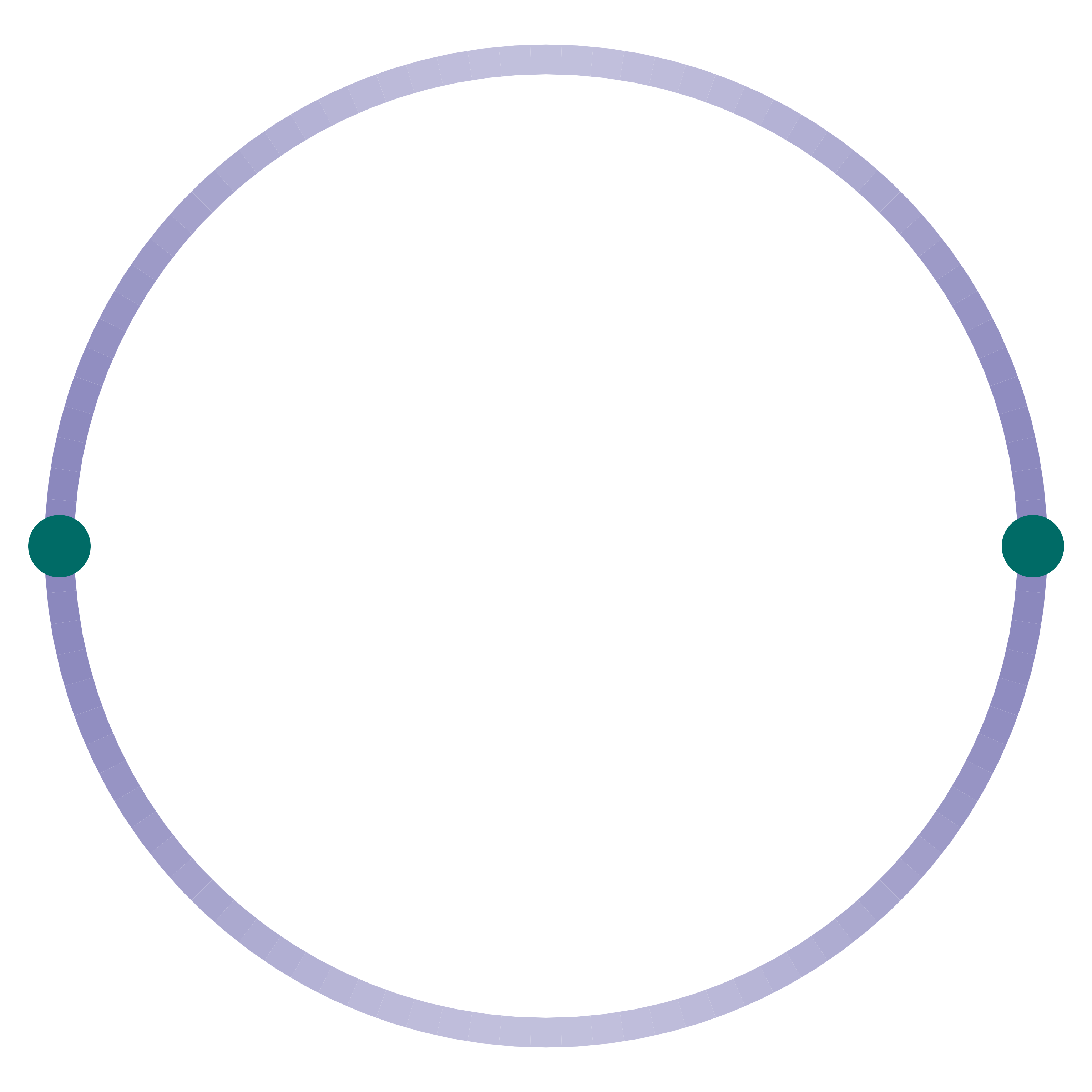}}%
\end{subfigure}%
\caption{Energies of the states $X^{\text{single}} = ((0,1))$ in blue, $X^{\text{two},1} = ((0,1), (0,-1))$ in red, and $X^{\text{two},2} = ((1,0), (-1,0))$ in green for the matrix $D=\text{diag}(-1, \lambda_2)$ with varying values of $\lambda_2$.}\label{fig:maxind}
\end{figure}

%% file: conclusion.tex
\section{Conclusion}\label{sec:conclusion}

In this work, we studied a mathematical model of self-attention layers used in the  transformer architecture. Building upon~\cite{geshkovski2023mathematical} we analyzed a continuum limit in the space of probability measures on a sphere.
In order to understand the underlying geometry, we studied a new optimal transport distance $W_{m,2}$ with a non-local mobility. We proved that the space of probability measures with this distance is a geodesic space and characterized absolutely continuous curves in this space. This allowed us to interpret  continuity equations \cref{eq:GradFlowPDE} as curves of maximal slope of the interaction energy and to analyze the large-time behavior using the energy dissipation property, showing that the dynamics converge to a stationary point of the interaction energy.

We analyzed these critical points (in particular, minimizers and maximizers) for various types of interactions determined by the matrix $D$ in~\eqref{eq:energy}. These results are summarized in~\cref{tab:summary-minimizers}. 
We find that positions of stationary points are strongly connected to normalized eigenvectors of $D$, which form a strict subset of $\S$ in the case $D \neq \lambda\,\id$. In other words, the regions where clusters appear do
not only depend on the initial configuration, but also on the interaction matrix itself. This could be related to mode collapse often observed in practice.
It is an interesting question to understand whether an alternative, rotation-invariant architecture could prevent mode collapse.

\begin{table}
	\centering%
\begin{tabular}{l | c | c}
 Property of $D$ & Minimizers & Maximizers \\ 
 \toprule
 positive definite & symmetric w.r.t. all eigenvectors& $\mu = \delta_{z_{\max}}$\\
 & (\cref{cor:psd-symmetric}, \cref{sec:num_minposdef}) & (\cref{thm:Diracmaxmin}, \cref{sec:num_diracmaxis})\\
 \midrule%
 positive semi-definite & any $\mu$ concentrated on $\mathcal{N}(D)$&$\mu = \delta_{z_{\max}}$\\
 & (\cref{thm:minmaxindefinite}, \cref{sec:num_minposdef}) & (\cref{thm:Diracmaxmin}
 )\\
\midrule%
negative (semi-)definite & $\mu = \delta_{z_{\min}}$ &  $\mu = 1/2\,(\delta_{z_{\min}} + \delta_{-z_{\min}} )$ \\  &
(\cref{thm:Diracmaxmin}) & (\cref{cor:ndsymm}, \cref{sec:num_diracmaxis_neg})\\
\midrule
indefinite & $\mu = \delta_{z_{\min}}$ & $|\lambda_{\max}|$ maximal: $\mu = \delta_{z_{\max}}$\\
& (\cref{thm:minmaxindefinite}) & (\cref{thm:Diracmaxmin}, \cref{sec:num_diracmaxis_neg})\\
\bottomrule%
\end{tabular}
\caption{Summary of results on minimizers/maximizers of the interaction energy~\eqref{eq:energy}. We denote by $z_{\min}$ and $z_{\max}$ the eigenvectors that correspond to the smallest, respectively largest, eigenvalue of $D$.}
\label{tab:summary-minimizers}
\end{table}

Several further questions remain open for future work: 
as already discussed, it would be interesting to study the optimal transport distance for mobilities $m_\mu$ that cannot be bounded from below,  which is the case, for example, in problems of opinion dynamics where the  Gaussian kernel on the Euclidean space is often used. In this case the metric $W_{m,2}$ is no longer equivalent to $W_2$.
So far we have only shown that \cref{eq:GradFlowPDE} are gradient flows in $(\mathcal{P}(M), W_{m,2})$ using the concept of curves of maximal slope. We do not know if these curves  satisfy the slightly stronger energy variational inequality (EVI), which would yield an easy stability estimate for solutions of \cref{eq:GradFlowPDE}.

From a practical point of view, an even more interesting direction is studying more general flows in $W_{m,2}$ that correspond to non-symmetrical matrices $D$ in~\eqref{eq:energy}, which is common in transformer architectures. As mentioned above, basic properties of the distance carry over to the non-symmetric case but characterizing the stationary states is non-trivial; one possibility is splitting the effective velocity fields into a dissipative and a (generalized) divergence-free part, similarly to non-symmetric Fokker--Planck equations.

Finally, to justify the use of the continuum limit for studying practical behavior of transformers one 
needs to establish convergence of discrete time-stepping in arbitrary time intervals.
Moreover, it is worth studying how the step size influences the behavior of the system and what effect weight-sharing would have.



%% file: acks.tex
\ack{
MB and TR acknowledge funding by the German Ministry of Science and Technology (BMBF) under grant agreement No. 01IS24072A (COMFORT).
MB, SK, TR and LW acknowledge support from DESY (Hamburg,
Germany), a member of the Helmholtz Association HGF. This research was
supported in part through the Maxwell computational resources operated
at Deutsches Elektronen-Synchrotron DESY, Hamburg, Germany. MB, SK acknowledge support from the German Research Foundation, project BU 2327/19-1. MB and LW acknowledge support from the German Research Foundation, project BU 2327/20-1. 
YK acknowledges support from the German Research Foundation as visiting fellow within the priority programme Foundations of Deep Learning. Part of this study was carried out while SK and TR were visiting the California institute of technology, supported by the DAAD grant for project 57698811 \enquote{Bayesian Computations for Large-scale (Nonlinear) Inverse Problems in Imaging}.
YK acknowledges the support of the EPSRC (Fellowship EP/V003615/2 and Programme Grant EP/V026259/1).
SK and YK are grateful for the hospitality of the University of Bath during the workshop \enquote{Machine Learning in Infinite Dimensions}, sponsored by the ICMS, LMS, IMI Bath, ProbAI, and Maths4DL, where part of this work was undertaken.
}

%% file: AGradFlow.tex
\section{Proofs of \cref{sec:gradientflow}}
\subsection{Continuity equation on manifolds}\label{app:ContEq}

Let $M$ be a compact, $n$-dimensional Riemannian manifold and  $TM= \sqcup_{x\in M} T_xM$ its tangent bundle. Although $TM$ is not a vector space, the tangent bundle $TM$ itself can be considered as an $2n$-dimensional Riemannian manifold.
For its proper definition and the topology on $TM$ we refer to
\cite[Chapter 3 (The Tangent Bundle)]{Lee2012}. Velocity fields on manifolds are maps $V:M \rightarrow TM$ such that $\pi\circ V=Id_{M}$, where $\pi:TM \rightarrow M$ is the projection map sending each vector in $T_xM$ to $x$. We will regularly commit the mild crime of interpreting $V(x)$ as an element in $T_xM$ instead of $TM$.
Let $I=(0,T)$ be an open interval, ${(\mu_t)}_{t\in I}$ be a Borel family  of probability measures on $M$ and $v: (x,t)\in M\times I \mapsto v_t(x) \in T_xM$ be a time dependent Borel velocity field such that
\begin{align}\label{eq:VBound}
    \int \int  |v_t(x)| \d\mu_t \d t<\infty,
\end{align}
where $|\cdot|: TM_x \rightarrow [0,+\infty)$ denotes the norm induced by the inner product of the Riemannian structure.
The continuity equation holds in the sense of distributions if
\begin{align}\label{eq:CEA}
    \int_{(0,T)} \int_M \partial_t \varphi(x,t)+ \langle \mathcal{D} \varphi(x,t)  ,v_t(x)\rangle \d\mu_t \d t=0 \quad \forall \varphi\in C_c^1(M\times(0,T)).
\end{align}
Here $\mathcal{D}\varphi$ denotes the differential of the map $x\in M\mapsto \varphi(t,x)$ for a fixed $t\in [0,T]$.

\begin{prop}[Properties]\label{pro:Pro} Solutions to the continuity equation have the following properties:
\begin{itemize}
    \item \textbf{Continuous representative:} Let $\mu_t$ be a Borel family of probability measures satisfying \cref{eq:CEA} for a Borel vector field $v_t$ satisfying \cref{eq:VBound}. Then there
 exists a narrowly continuous curve $t\in[0,T] \rightarrow \tilde{\mu}_t\in\mathcal{P}(M)$ such that $\mu_t=\tilde{\mu}_t$ for a.e. $t\in(0,T)$. Moreover, if $\varphi\in C^1_c(M\times [0,T])$ and $s\leq r\in[0,T]$ we have~\cite[Lemma 8.1.2]{ambrosio2008gradient}
 \begin{align}\label{eq:BInt}
     \int_M \varphi(x,r) \d\tilde{\mu}_r-\int_M \varphi(x,s) \d\tilde{\mu}_s=\int_s^r \int_M\partial_t \varphi+\mathcal{D}\varphi(v_t)  \d\mu_t \d t.
 \end{align}
    \item \textbf{Time rescaling:} Let $t: s\in[0,T'] \rightarrow t(s)\in[0,T]$ be a strictly increasing absolutely continuous map with absolutely continuous inverse $s\coloneqq t^{-1}$.
 Then $(\mu_t,v_t)$ is a distributional solution of the continuity equation if and only if~\cite[Lemma 8.1.3]{ambrosio2000functions}
 \begin{align*}
    \hat{\mu}\coloneqq\mu\circ t,\  \hat{v}\coloneqq t'v\circ t \; \text{ is a distributional solution of the continuity equation on } (0,T').
 \end{align*}
    \item \textbf{Gluing solutions:} Let $\{\mu_t\}_{t\in[0,T_1]}, \{\nu_t\}_{t\in[0,T_2]}$ be two narrowly continuous curves in $\mathcal{P}(M)$ with $\mu_{T_1}=\nu_0$.
    Let further $\{v\}_{t\in[0,T_1]},\{w\}_{t\in[0,T_2]}$ be the corresponding Borel velocity fields such that \cref{eq:BInt} is satisfied. Then $\{\eta_t\}_{t\in[0,T_1+T_2]}$ and $\{u_t\}_{t\in[0,T_1+T_2]}$ defined by
    \begin{align*}
        \eta_t\coloneqq
        \begin{cases}
            \mu_t &\text{if } t\in [0,T_1],\\
            \nu_{t-T_1} &\text{if } t\in (T_1,T_1+T_2],
        \end{cases}
        \quad u_t\coloneqq
        \begin{cases}
            v_t &\text{if } t\in [0,T_1],\\
            w_{t-T_1} &\text{if } t\in (T_1,T_1+T_2]
        \end{cases}
    \end{align*}
    satisfy \cref{eq:BInt} \cite[Lemma 4.4]{dolbeault2009new}.
\end{itemize}
\end{prop}
\subsection{Proof of \cref{eq:Geo}}\label{app:Geo}
We follow the proof strategy from \cite{dolbeault2009new} for the \enquote{flat} Euclidean case but since $TM$ is not a vector space  modifications are required.
We start by establishing a compactness result for solutions of continuity equations with finite energy. 
For our purposes we define the \enquote{lifted} flux $J_t \in \mathcal{P}(TM \times M)$ in duality with $C_c(TM\times M)$ (see \cite[Theorem 7.2]{folland1999real}) by
\begin{align}\label{eq:Flux}
    \int_{TM\times M} \varphi(w,y ) \d J_t(w,y)=\int_M \int_M \varphi(v_t(x),y) \d\mu_t(x) \d\mu_t(y) \quad \forall \varphi\in C_c(TM\times M).
\end{align}
Notably, $(\mu_t,J_t)$ solve the continuity equation in the sense that for all $s\leq r\in[0,T]$ 
\begin{align}\label{eq:CoEqFl}
    \int_M \varphi_r \d\mu_r-\int_M \varphi_s \d\mu_s =\int_s^r\int_M \partial_t \varphi \d\mu_t \d t+\int_s^r \int_{TM\times M} \tilde{\mathcal{D} \varphi} \d J_t \d t \quad \forall \varphi\in C^1(M\times [0,T]),
\end{align}
where $\tilde{\mathcal{D} \varphi}: (w,y) \mapsto \langle\mathcal{D}\varphi(\pi(w)),w\rangle$ is the extension of $\mathcal{D}\varphi$ onto $TM \times M$ that is constant along $y\in M$. 
Further, we define $\bm{J}\in\mathcal{P}(TM\times M \times[0,T])$ in duality with $C_c(TM\times M\times[0,T])$ by 
 \begin{align*}
 \int_{TM\times M \times (0,T)} \varphi \d\bm{J}=\int_0^T\int_{TM\times M} \varphi \d J_t \d t \quad\forall \varphi\in C_c(TM\times M \times [0,T]).
\end{align*}

\begin{lemma}\label{lm:BEComp} Let $(\mu^n,v^n)$ be a sequence in $CE(0,T)$ with
\begin{align*}
\sup_n \left\{\int_0^1 \int_{M} m_\mu(x)|v_t^n(x)|^2\d\mu_t^n(x)\d t \right\}< +\infty.
\end{align*}
Then there exists a subsequence and a couple $(\mu, J)$ satisfying the continuity equation in the sense of \cref{eq:CoEqFl} such that
\begin{align*}
\mu_t^n \rightharpoonup \mu_t \quad \forall t\in[0,T]\qquad\text{and}\quad
    \bm{J}^n \rightharpoonup \bm{J}
\end{align*}
and for the map $g: (v,p)\in TM\times M \mapsto (\pi(v),p)$ one has
\begin{align} \label{eq:DisIntJ}
    g_\# J_t= \mu_t \otimes \mu_t\quad  \text{for a.e. } t\in(0,T).
\end{align}
\end{lemma}

\begin{proof}
    \textbf{Step 1} (Convergence of $\bm{J}$)\textbf{:}\\
     The estimate 
    \begin{align*}
    \sup_n\int_0^T \int_{TM\times M} |w|^2 \d J_t^n(w,x) \d t \leq \frac{1}{C}    \sup_n\int_0^T \int_{TM} m_{\mu^n_t} |v^n_t|^2 \d\mu^n_t \d t<\infty
    \end{align*}
    combined with the fact that $M$ is compact and \cite[Remark 5.1.5]{ambrosio2008gradient} implies tightness of $\bm{J}\in \mathcal{P}(TM\times M\times [0,T])$.
By disintegrating $\bm{J}$ we obtain a Borel family $J_t$ such that $\d\bm{J}=\d J_t \d t$.
Since $M$ is compact $\mu_0^n$ is tight and we extract a further subsequence such that $\mu_0^n{\rightharpoonup} \mu_0$. \\
\textbf{Step 2} (Convergence of $\mu_t$)\textbf{:}\\
Consider a function $\varphi \in C^1(M)$ and for $t \in [0,T]$  set $\bm{\zeta}:(v,y,t)\in TM\times M\times[0,T]\mapsto \chi_{[0,t]} \langle \mathcal{D}\varphi(\pi(v)),v\rangle$.  Since the discontinuity set of $\bm{\zeta}$ is concentrated on $N=TM\times M\times \{0,t\}$ and $|F|(N)=0$, general convergence
 theorems (see, e.g. [3, Prop. 5.1.10]) imply
 \begin{align}\label{eq:LimIntJ}
     \lim_{n\rightarrow\infty} \int_0^t \int_{TM\times M} \tilde{\mathcal{D}\varphi} \d J^n_t \d t &=\lim_{n\rightarrow\infty}\int_{TM\times M\times [0,T]} \bm{\zeta} \d\bm{J}^n \nonumber \\
     &=\int_{TM\times M \times [0,T]} \bm{\zeta} \d\bm{J}=\int_0^t \int_{TM\times M} \tilde{\mathcal{D}\varphi} \d J_t \d t. 
 \end{align}
 Let us fix a $t\in(0,T]$. Since $M$ is compact, $\mu_t^n$ is tight and we can extract from any subsequence  a further subsequence such that  $\mu_t^n$ converges narrowly.
 Then by \cref{eq:CoEqFl} and \cref{eq:LimIntJ} and the fact that $C^1$ is dense in $C^0$ we know that for all subsequences have the same limit. Therefore, $\mu_t^n {\rightharpoonup}\mu_t\in \mathcal{P}(M)$ for a particular $\mu_t$. By the previous calculations we also immediately obtain that $(\mu,J)$ satisfy the continuity equation in the sense of \cref{eq:CoEqFl}.
 To show \cref{eq:DisIntJ} we observe that since $M$ compact
 \begin{align*}
g_{\#}J_t^n=\mu^n_t\otimes\mu^n_t {\rightharpoonup} \mu_t\otimes\mu_t \quad \forall t\in[0,T] .
 \end{align*}
\end{proof}


\begin{proof}[Proof of \cref{eq:Geo}] \textbf{Step 1:}\\
Let $(\mu^n,v^n) \in CE(0,1)$ be a minimizing sequence of \cref{eq:OpTrDi} for some $\mu_0,\mu_1$. Then the conditions of \cref{lm:BEComp} are met and we obtain that 
\begin{align*}
    \mu_t^n \rightharpoonup \mu_t \in \mathcal{P}(M) \quad \forall[0,T] \qquad \text{and} \qquad 
    \bm{J}^n \rightharpoonup \bm{J} \in\mathcal{P}(TM\times M\times[0,T]),
\end{align*}
where the limit satisfies the continuity equation in the sense of \cref{eq:CoEqFl}. Equation \cref{eq:DisIntJ} in particular implies that $\bm{J}$ can be disintegrated in the following way,
$$\d\bm{J}=\d u_{t,x}(v) \d\mu_t(x)\d\mu_t(y)\d t,$$
where $u_{t,x}(v) \in \mathcal{P}(TM_p= \pi^{-1}(x))$. Using \cite[Lemma 5.1.7]{ambrosio2008gradient} we now show that for $\Bar{u}_{x,t}=\textbf{Mean}(u_{t,x})=\int_{TM_x} v \d u_{t,x}(v)$ it holds that
\begin{align*}
    & W_{m,2}(\mu_0,\mu_1)^2 =\lim_{n\rightarrow \infty}\int_0^1\int_M m_{\mu^n_t}(x)|v^n_t(x)|^2 \d\mu_t^n(x) \d t \\
    =&\lim_{n\rightarrow \infty}\int_0^1\int_{TM \times M} K(\pi(v),y)|v|^2 \d J_t^n(v,y) \d t
    = \int_0^1\int_{TM \times M} K(\pi(v),y)|v|^2 \d J_t(v,y) \d t \\
    =&\int_0^1\int_M\int_M K(x,y)\d\mu_t(y) \int_{TM_p} |v|^2 \d u_{x,t}(v)   \d\mu_t(x) \d t \\
    \geq& \int_0^1\int_M m_{\mu_t}(x) \int_{TM_x} |v|^2  \d \delta_{\Bar{u}_{x,t}}(v)   \d\mu_t(x) \d t=\int_0^1\int_M m_{\mu_t}(x) |\Bar{u}_{x,t}|^2\d\mu_t(x)\d t,
\end{align*}
where in the last line we used Jensen's inequality. Since $\mathcal{D}\varphi(x): TM_x \rightarrow \R $ is linear and $(\mu,\bm{J})$ satisfy \cref{eq:CoEqFl} this implies that $\left(\mu,v=(\Bar{u}_{x,t})_{t\in[0,T]}\right)\in CE(0,1)$ and for this couple the infimum in \eqref{eq:OpTrDi} is obtained. \\
\textbf{Step 2:}\\
\cref{pro:Pro} and a linear time rescaling show that
\begin{align}\label{eq:AltDisDef}
    W_{m,2}^2(\mu_0,\mu_T)=\inf\left\{ T\int_0^T \int_M m_{\mu_t} |v_t|^2\d\mu_t\d t\ :\quad (\mu_t,v_t)\in CE(0,T; \mu_0\rightarrow\mu_T)\right\}.
\end{align}
We denote by $\Bar{W}_{m,2}(\mu,\nu)$ the infimum in \cref{eq:GeoRep} and show that indeed $\Bar{W}_{m,2}(\mu,\nu)=W_{m,2}(\mu,\nu)$. By Hölder's inequality we immediately obtain that $\Bar{W}_{m,2}(\mu,\nu)\leq W_{m,2}(\mu,\nu)$. To show the reverse we follow the arguments of \cite[Theorem 5.4]{dolbeault2009new} and define for $(\mu,v)\in CE(0,T; \mu \rightarrow \nu)$ 
\begin{align*}
    s_\epsilon(t)\coloneqq \int_0^t \left(\epsilon+\int_M m_{\mu_t} |v_t|^2 \d\mu_t\right)^{1/2}\d r\quad \text{for }t\in[0,T].
\end{align*}
Then $s_\epsilon$ is strictly increasing,  $s_\epsilon'\geq \epsilon$ and $s_\epsilon(0,T)=(0,S_\epsilon)$ with $S_\epsilon\coloneqq s_\epsilon(T)$, so that its inverse map $t_\epsilon: [0,S_\epsilon]\rightarrow [0,T]$ is well defined and Lipschitz continuous and
\begin{align*}
    t'_\epsilon \circ s_\epsilon\coloneqq \left(\epsilon+\int_{M} m_{\mu_t} |v_t|^2\d\mu_t\right)^{-1/2} \quad \text{for a.e. } t\in (0,T).
\end{align*}
By \cref{pro:Pro} we have that for  $\mu^\epsilon \coloneqq \mu \circ t_\epsilon$, $v^\epsilon\coloneqq t_\epsilon' v\circ t_\epsilon$ the couple $(\mu^\epsilon,v^\epsilon)\in CE(0,S_\epsilon;\mu,\nu)$ and
\begin{align*}
    W^2_{m,2}(\mu,\nu)&\leq S_\epsilon \int_0^{S_\epsilon} \int_M m_{\mu_t^\epsilon} |v_t^\epsilon|\d\mu_t^\epsilon \d s\\ &= S_\epsilon \int_0^T \frac{\int_M m_{\mu_t}|v_t|^2\d\mu_t}{\epsilon+\int_M m_{\mu_t}|v_t|^2\d\mu_t}\left(\epsilon +\int_M m_{\mu_t}|v_t|^2\d\mu_t \right)^{1/2}\d t,
\end{align*}
with the last term being smaller or equal to $S_\epsilon^2$. Sending $\epsilon \rightarrow 0$ we obtain
\begin{align*}
    W_{m,2}(\mu,\nu)=\int_0^T\left( \int_M m_{\mu_t} |v_t|^2 \d\mu_t\right)^{1/2} \d t \quad \text{for all } (\mu,v)\in CE(0,T; \mu \rightarrow \nu)
\end{align*}
and hence $W_{m,2}(\mu,\nu)=\Bar{W}_{m,2}(\mu,\nu)$. This in particular implies that for every minimizer $(\mu,v)\in CE(0,1; \mu \rightarrow \nu)$ of \cref{eq:OpTrDi} the equality 
\begin{align*}
    \left(\int_0^1 \int_M m_{\mu_t} |v_t|^2 \d\mu_t \d t\right)^{1/2} =\int_0^1\left( \int_M m_{\mu_t} |v_t|^2 \d\mu_t \right)^{1/2}\d t
\end{align*}
holds, which is only the case when $\int_M m_{\mu_t} |v_t|^2 \d\mu_t $ is constant for a.e. $t\in (0,T)$, implying by a further time rescaling argument
\begin{equation} \label{eq:ConstGeo}
    W_{m,2}(\mu_s,\mu_t)=|s-t|W_{m,2}(\mu_0,\mu_1) \quad \forall\ 0\leq s\leq t\leq 1. \qedhere
\end{equation}
\end{proof}
\subsection{Proof of \cref{lm:AC}}\label{app:AC}
\begin{proof}[Proof of \cref{lm:AC}]
    If $(\mu,v)\in CE(0,T)$ and $\int_0^T\left( \int_M m_{\mu_t} |v_t|^2 \d\mu_t\right)^{1/2} \d t<+\infty$  then by \cref{eq:GeoRep} we have
    \begin{align*}
    W_m(\mu_s,\nu_r)\leq\int_s^r \left( \int_M m_{\mu_t} |v_t|^2 \d\mu_t\right)^{1/2}\d t
    \quad \forall 0\leq s\leq r\leq T.\end{align*}
    On the other hand, if $\mu_t$ is an absolutely continuous curve then  by a standard reparametrization argument \cite[ Lemma 1.1.4]{ambrosio2008gradient} we may assume $\mu_t$ to be Lipschitz. For $N\in \mathbb{N}$ we set the step size as $\tau=T 2^{-N}$ and choose a family of constant-speed geodesics $(\mu^{k,N},v^{k,N})\in CE({(k-1)\tau}, k\tau; \mu_{(k-1)\tau}\rightarrow \mu_{k\tau})$, $k\in \{1,...,N\}$ such that for $t\in ((k-1)\tau, k\tau)$
    \begin{align*}
        \tau \int_{M} m_{\mu_t} |v_t|^2 \d\mu_t \stackrel{\cref{eq:AltDisDef}}{=} \frac1\tau W_m^2(\mu_{(k-1)\tau},\mu_{k\tau})\leq \frac1\tau \left(\int_{(k-1)\tau}^{k\tau} |\dot\mu|(t) \d t\right)^2\stackrel{\text{Hölder}}{\leq} \int_{(k-1)\tau}^{k\tau} |\dot\mu|(t)^2 \d t.
    \end{align*}
Gluing all geodesics together by \cref{pro:Pro} we obtain a curve $(\mu^N,v^N)\in CE(0,1)$.  \cref{lm:BEComp} gives us a subsequence, still denoted by $N$, and a couple $(\tilde{\mu},\tilde{v})\in CE(0,1)$ such that $\mu_t^N\rightharpoonup \tilde{\mu}_t$ and $\bm{J}\rightharpoonup \Tilde{\bm{J}}$. By construction $\tilde{\mu}_t$ and $\mu_t$ coincide on the dense (in $[0,T]$) set $\{0\}\cup\left\{ \frac T M 2^{-N}: M,N\in \R, M\leq N\right\}$. Since both $\tilde{\mu}_t$ and $\mu_t$ are narrowly continuous $\tilde{\mu}_t=\mu_t$ must hold. 
Again, equation \cref{eq:DisIntJ} implies that $\bm{J}$ can be disintegrated in the following way
$$\d\bm{J}=\d u_{t,x}(v) \d\mu_t(x)\d\mu_t(y)\d t,$$
where $u_{t,x}(v) \in \mathcal{P}(TM_x= \pi^{-1}(x))$. Then $(\mu, \tilde{v})\in CE(0,T)$ with 
$\tilde{v}_t\coloneqq \int_{TM_x} w \d u_{t,x}(w)$ and

\begin{align}\label{eq:ACH}
\begin{split}
    &\int_0^T  \int_M m_{\mu_t} |\tilde{v}_t|^2  \d\mu_t \d t\stackrel{\text{Jensen}}{\leq} \int_0^T\int_M\int_M K(x,y)\d\mu_t(y) \int_{TM_p} |v|^2 \d u_{x,t}(v)   \d\mu_t(x) \d t \\
    \leq& \int_0^T\int_{TM \times M} K(\pi(v),y)|v|^2 \d J_t(v,y) \d t\leq\liminf_{n\rightarrow \infty}\int_0^T\int_{TM \times M} K(\pi(v),y)|v|^2 \d J_t^n(v,y) \d t\\
    =& \liminf_{n\rightarrow \infty}\int_0^T \int_M m_{\mu^n_t}|v^n_t|^2 \d\mu_t \d t\leq \int_0^T |\dot\mu|^2(t)\d t 
\end{split}
\end{align}
Since $(\mu,\tilde{v})\in CE(0,T)$ we have that
$$ |\dot{\mu}|(t)\leq \left( \int_M m_{\mu_t} |v_t|^2 \d\mu_t\right)^{1/2}\quad \text{for a.e. }t\in(0,T).$$
Finally, for \cref{eq:ACH} to hold $  |\dot{\mu}|(t)= \left( \int_M m_{\mu_t} |v_t|^2 \d\mu_t\right)^{1/2} $ must hold for a.e. $t\in(0,T)$ 
\end{proof}

\subsection{Proof of Lemma \ref{lm: ChainRule}}\label{sc:ProofChain}

\begin{proof}[Proof of Lemma \ref{lm: ChainRule}]
From \cref{cor:Top} we know that the distances $W_2$ and $W_{m,2}$ are equivalent. Therefore, we can assume absolute continuity wit respect to $W_2$. Further, by a standard rescaling argument (e.g. \cite[Lemma 1.1.4]{ambrosio2008gradient} or \cite[Lemma 8.1.3]{ambrosio2008gradient}) it is enough to show prove \cref{eq:CheinRule} for $1$-Lipschitz curves (w.r.t. $W_2$), i.e. we only need to consider absolutely continuous curves $(\mu_t,v_t)\in CE(0,1; \mu\rightarrow \nu)$ such that
\begin{align*}
    \int_M |v_t(x)|^2 \d\mu_t(x)=1 \quad \text{for a.e. }t\in (0,T).
\end{align*}
For convenience we will set $\mu_t=\mu_0$ for $t\leq 0$ and $\mu_t=\mu_T$ for $t\geq T$ as well as $v_t=0$ for $t\not\in[0,T]$.
We define the function $\eta: (x,t)\in M\times \R \mapsto \frac{1}{2}\int_M W(x,y) \d\mu_t(y)$ for which 
$$ \partial_t \eta(t,x) =\begin{cases}
    0&\quad \text{if } t\not \in [0,T]\\
    \frac{1}{2}\int_M \langle \mathcal{D}_y W(x,y),v_t(y)\rangle   \d\mu_t(y)&\quad \text{else}
\end{cases} $$
in the distributional sense. Using the mollifier $g_\epsilon$ as described in \cite[C.4.]{evans2022partial}, one can  smooth out $\eta$ in the time direction by setting 
\begin{align*}
    \eta_\epsilon(t,x)\coloneqq \int_\R \eta(\tau, x) g_\epsilon(t-\tau) \d\tau.
\end{align*}
By \cite[C.5. Theorem 7 (iii)]{evans2022partial} we have that  $\eta_\epsilon \rightarrow \eta$ pointwise and with the use of the dominated convergence theorem with the upper bound $|\eta_\epsilon| \leq \sup_{(x,y)\in M\times M} |W(x,y)|<\infty$ we calculate 
\begin{align*}
    \mathcal{E}(\mu_T)-\mathcal{E}(\mu_0)=\int_M \eta \d\mu_T-\int_M \eta \d\mu_0=\lim_{\epsilon\rightarrow 0} \int_M \eta_\epsilon \d\mu_T-\int_M \eta_\epsilon \d\mu_0.
\end{align*}
We further have that
\begin{align*}
    +\infty&>\frac{1}{2}\int_0^T\int_M \int_M \langle \mathcal{D}_{y} W(y,x),v_t(y)\rangle  \d\mu_t(y) \d\mu_t(x) \d t\\
    &\stackrel{(**)}{=}\lim_{\epsilon\rightarrow 0} \frac{1}{2}\int_0^T\int_\R\int_M \int_M \langle \mathcal{D}_{y} W(x,y),v_t\rangle   \d\mu_t(y)  g_\epsilon(t-\tau)(t) \d\tau \d\mu_t(x)  \d t\\
    &\stackrel{(*)}{=}-\lim_{\epsilon\rightarrow 0} \int_0^T\int_M\int_\R \eta(\tau,x) \partial_\tau g_\epsilon(t-\tau) \d\tau \d\mu_t(x)  \d t\\
   & =\lim_{\epsilon\rightarrow 0} \int_0^T\int_M\int_\R \eta(\tau,x) \partial_t g_\epsilon(t-\tau) \d\tau \d\mu_t(x)  \d t=\lim_{\epsilon\rightarrow 0} \int_0^T\int_M \partial_t \eta_\epsilon(t,x) \d\mu_t(x)  \d t,
\end{align*}
where for $(*)$ we use the definition of the distributional derivative and rearrange the integral using the Fubini--Tonelli theorem.
To prove $(**)$, we need to define a piecewise constant approximation of $\mu_t$.
We fix a $N\in \mathbb{N}$ $\tau=\frac{T}{N}$ and set for $k\in \{1,N\}$
\begin{align*}
    \bar{\mu}_t\coloneqq \mu_{k\tau} \quad\text{for } t\in[k\tau, (k+1) \tau) , \qquad 
    \bar{\mu}_T\coloneqq\mu_T.
\end{align*}
Since $\mu_t$ is $1$-Lipschitz we have $W_2(\mu_t,\bar{\mu}_t)\leq \tau$ for all $t\in[0,T]$.
Then we estimate 
\begingroup
\allowdisplaybreaks
\begin{align}\label{eq:CES1}
&\bigg|\int_0^T\int_M \int_M \langle \mathcal{D}_{y} W(x,y),v_t(y)\rangle
\d\mu_t(y) \d\mu_t(x) \d t\nonumber\\
- &\int_0^T\int_M \int_M \langle \mathcal{D}_{y} W(x,y),v_t(y)\rangle  \d\mu_t(y)\d\bar{\mu}_t(x) \d t \bigg| \nonumber \\
\leq &\int_0^T \int_{M\times M} \int_M \bigg|\langle \mathcal{D}_{y} W(x_1,y),v_t(y)\rangle-\langle \mathcal{D}_{y} W(x_2,y),v_t(y)\rangle\bigg|\d\mu_t(y)  \d\pi_t(x_1,x_2) \d t \nonumber \\
\leq &\int_0^T \int_{M\times M} \int_M \bigg| \mathcal{D}_{y} W(x_1,y)-\mathcal{D}_{y} W(x_2,y)\bigg|_* |v_t(y)|\d\mu_t(y)  \d\pi_t(x_1,x_2) \d t \nonumber \\
\leq &C\, \int_0^T \int_{M\times M} \int_M\bigg| x_1-x_2\bigg| |v_t(y)|\d\mu_t(y)  \d\pi_t(x_1,x_2) \d t \\
=&C\, \int_0^T \bigg(\int_{M\times M} \bigg| x_1-x_2\bigg|   \d\pi_t(x_1,x_2)\bigg)\left(\int_M |v_t(y)|\d\mu_t(y)\right) \d t \nonumber \\
\leq &C\,\bigg(\int_0^T W^2_2(\mu_t,\bar{\mu}_t)\d t \int_0^T \int_M |v_t(y)|^2 \d\mu_t(y)\d t\bigg)^{1/2} \nonumber \\
=&C\,\bigg(T \int_0^T W^2_2(\mu_t,\bar{\mu}_t)\d t\bigg)^{1/2}
\quad \leq\quad  C\bigg(T \int_0^T \tau^2\d t\bigg)^{1/2}\quad =\quad  C \frac{T^2}{N}, \nonumber 
\end{align}
\endgroup
where $\pi_t\in \mathcal{P}(M\times M)$ is the optimal transport plan between $\mu_t$ and $\bar{\mu}_t$ and $| \cdot|_*$ denotes the dual norm of $|\cdot|$. (For more details on the static formulation of Wasserstein distances via optimal transport plans, we refer to \cite[Chapter 6]{ambrosio2008gradient} ). We can argue similarly in the mollified case

\begingroup
\allowdisplaybreaks
\begin{align}\label{eq:CES2}
&\bigg|
\int_0^T\int_M \int_\R\int_M \langle \mathcal{D}_{y} W(x,y),v_\tau(y)\rangle  \d\mu_\tau(y) g_\epsilon(t-\tau) \d\tau \d\mu_t(x) \d t - \nonumber \\
&\int_0^T\int_M \int_\R\int_M \langle \mathcal{D}_{y} W(x,y),v_\tau(y)\rangle  \d\mu_\tau(y)  g_\epsilon(t-\tau) \d\tau\d\bar{\mu}_t(x) \d t \bigg| 
\nonumber \\
\leq& 
\int_0^T \int_{M\times M} \int_\R \int_M\big|\langle \mathcal{D}_{y} W(x_1,y),v_\tau(y)\rangle-\nonumber\\
&\hspace{3cm}\langle \mathcal{D}_{y} W(x_2,y),v_\tau(y)\rangle\big|\d\mu_\tau(y) g_\epsilon(t-\tau) \d\tau \d\pi_t(x_1,x_2) \d t \nonumber \\
\leq& \int_0^T \int_{M\times M} \int_\R\int_M\big| \mathcal{D}_{y} W(x_1,y)-\mathcal{D}_{y} W(x_2,y)\big|_* |v_\tau(y)|\d\mu_\tau(y)  g_\epsilon(t-\tau) \d\tau\d\pi_t(x_1,x_2) \d t \nonumber \\
\leq& C \int_0^T \int_{M\times M} \int_\R\int_M\big| x_1-x_2\big| |v_\tau(y)|\d\mu_\tau(y) g_\epsilon(t-\tau) \d\tau \d\pi_t(x_1,x_2) \d t\\
=& C \int_0^T \bigg(\int_{M\times M} \big| x_1-x_2\big|   \d\pi_t(x_1,x_2)\bigg) \left(\int_\R\int_M |v_\tau(y)|\d\mu_\tau(y) g_\epsilon(t-\tau)\d\tau\right) \d t \nonumber\\
\leq& C \int_0^T  W_2(\mu_t,\bar{\mu}_t)\int_\R\int_M |v_\tau(y)|\d\mu_\tau(y) g_\epsilon(t-\tau) \d\tau \d t \nonumber \\
\leq& C \frac{T}{N}\int_0^T \int_\R\int_M |v_\tau(y)|\d\mu_\tau(y) g_\epsilon(t-\tau) \d\tau \d t
 = C \frac{T}{N} \int_\R\int_M |v_\tau(y)|\d\mu_\tau(y) \int_0^T g_\epsilon(t-\tau)\d t \d\tau \nonumber \\
\leq& C \frac{T}{N} \int_\R\int_M |v_\tau(y)|\d\mu_\tau(y)\d\tau
\quad =\quad C \frac{T}{N} \int_0^T\int_M |v_\tau(y)|\d\mu_\tau(y) \d\tau \nonumber \\
\leq& C \frac{T}{N} \int_0^T\bigg(\int_M |v_\tau(y)|^2 \d\mu_\tau(y)\bigg)^{1/2}\d\tau\quad \leq\quad C \frac{T^2}{N} . \nonumber 
\end{align}
\endgroup
We denote $\tilde{C}=\sup_{(x,y)\in M\times M} |\mathcal{D}_{y} W(x,y)|_*<+\infty$ and combine \cref{eq:CES1} and \cref{eq:CES2} to estimate
\begingroup
\allowdisplaybreaks
\begin{align*}
&\bigg|\int_0^T\int_M \int_M \langle \mathcal{D}_{y} W(x,y),v_t(y)\rangle  \d\mu_t(y) \d\mu_t(x) \d t\\ 
-& \int_0^T\int_M \int_\R\int_M \langle \mathcal{D}_{y} W(x,y),v_\tau(y)\rangle  \d\mu_\tau(y) g_\epsilon(t-\tau) \d\tau \d\mu_t(x) \d t \bigg|\\
\leq& 2C\frac{T^2}{N}+ \bigg|\int_0^T\int_M \hspace{2em}\hspace{-2em}\underbrace{\int_M \langle \mathcal{D}_{y} W(x,y),v_t(y)\rangle  \d\mu_t(y)}_{\coloneqq f(x,t)}\d\bar{\mu}_t(x) \d t - \\
& \phantom{2C\frac{T^2}{N}+\bigg|} \int_0^T\int_M \hspace{3em}\hspace{-3em}\underbrace{\int_\R\int_M \langle \mathcal{D}_{y} W(x,y),v_\tau(y)\rangle  \d\mu_\tau(y) g_\epsilon(t-\tau) \d\tau }_{\coloneqq f_\epsilon(x,t)}\d\bar{\mu}_t(x) \d t \bigg|\\
\leq& 2C\frac{T^2}{N}+\sum\limits_{i=1}^N \int_{(i-1)\tau+\epsilon}^{i\tau-\epsilon}\int_M|f-f_\epsilon| \d\bar{\mu}_t \d t+
\int_{(i-1)\tau}^{(i-1)\tau+\epsilon}\int_M |f-f_\epsilon|\d\bar{\mu}_t \d t +\\
&\hspace{6cm}\int_{i\tau-\epsilon}^{i\tau} \int_M|f-f_\epsilon|\d\bar{\mu}_t \d t\\
\leq& 2C\frac{T^2}{N}+\sum\limits_{i=1}^N \int_{(i-1)\tau+\epsilon}^{i\tau-\epsilon}\int_M|f-f_\epsilon| \d\bar{\mu}_t \d t+
\int_{(i-1)\tau}^{(i-1)\tau+\epsilon}\int_M 2\tilde{C}\d\bar{\mu}_t \d t +\int_{i\tau-\epsilon}^{i\tau} \int_M 2\tilde{C}\d\bar{\mu}_t \d t\\
 \leq& 2C\frac{T^2}{N}+\sum\limits_{i=1}^N \int_{(i-1)\tau+\epsilon}^{i\tau-\epsilon}\int_M|f-f_\epsilon| \d\bar{\mu}_t \d t+ 4N\epsilon \tilde{C}
 \leq \frac{\delta}{3}+\sum_{i=1}^N \frac{\delta}{3 N} +\frac{\delta}{3},
\end{align*}
\endgroup
where, firstly, $N$ is chosen such that $N\geq 6 C \frac{T^2}{\delta}$ and, secondly, $\epsilon$ such that  $\epsilon\leq \frac{\delta}{12 N \tilde{C}}$ and for each $i\in\{1,...,N\}$ it holds $\int_{(i-1)\tau+\epsilon}^{i\tau-\epsilon}\int_M|f-f_\epsilon| \d\bar{\mu}_t \d t \leq \frac{\delta}{3 N}$ (by \cref{lm:L1Con}). Therefore $(**)$ is proven.

Finally, by \cref{lm:Smooth} we obtain $n_\epsilon \in C^1(M\times[0,T])$ that we can use as a test function in \cref{eq:BInt} and send $\epsilon \rightarrow 0$ to obtain
\begin{multline*}
    \mathcal{E}(\mu_T)-\mathcal{E}(\mu_0)=\int_M \eta \d\mu_T-\int_M \eta \d\mu_0=\int_0^T\int_M \partial_t \eta \d\mu_t+\int_M \langle\mathcal{D} \eta,v_t\rangle  \d\mu_t \d t\\
    = \int_0^T\int_{M\times M} \langle \mathcal{D}_{x} W(x,y),v_t(x)\rangle \d\mu_t(x) \d\mu_t(y) \d t. \qedhere  
\end{multline*}
\end{proof}

\begin{lemma}\label{lm:L1Con}
    Let $f: M \times [0,T] \rightarrow \R$ be Borel measurable and $\mu \in \mathcal{P}(M)$
with
$$\int_a^b \int_M |f| \d\mu \d t<\infty \quad \text{for } 0\leq a < b\leq T.$$
For 
$$ \mu_a^b(A)\coloneqq \mu \otimes \mathcal{L}(a,b)$$
it holds
$$ \|f_\epsilon\|_{L^1(\mu_{a+\epsilon}^{b-\epsilon})}\leq \|f\|_{L^1(\mu_a^b)} \quad \text{and} \quad
f_\epsilon \rightarrow f \quad \text{in }{L^1(\mu_{a+\epsilon}^{b-\epsilon})}.$$ 
\end{lemma}
\begin{proof}
We adapt \cite[C.5 Theorem 7]{evans2022partial} to our case and start by showing
\begin{align*}
    &\|f_\epsilon\|_{L^1(\mu_{a+\epsilon}^{b-\epsilon})}\leq \int_M\int_{a+\epsilon}^{b-\epsilon} \int_a^b |f(x,\tau)| g_\epsilon(t-\tau)  \d\tau \d t \d\mu(x) \\
    =&\int_M\int_a^b  |f(x,\tau)| \int_{a+\epsilon}^{b-\epsilon}g_\epsilon(t-\tau)   \d t \d\tau \d\mu(x)= \int_M\int_a^b  |f(x,\tau)| \d\tau \d\mu(x)= \|f\|_{L^1(\mu_{a}^{b})}.
\end{align*}
We approximate $f$ in $L^1(\mu_a^b)$ by $\gamma\in C_c(M\times [a,b])$ (see \cite[Proposition 7.9]{folland1999real})
and calculate
\begin{multline*}
    \| f-f_\epsilon\|_{L^1(\mu_{a+\epsilon}^{b-\epsilon})}\leq \| f-\gamma\|_{L^1(\mu_{a+\epsilon}^{b-\epsilon})}+\| \gamma-\gamma_\epsilon\|_{L^1(\mu_{a+\epsilon}^{b-\epsilon})}+\| \gamma_\epsilon-f_\epsilon\|_{L^1(\mu_{a+\epsilon}^{b-\epsilon})}\\
    \leq 2 \| f-\gamma\|_{L^1(\mu_{a}^{b})}+\| \gamma-\gamma_\epsilon\|_{L^1(\mu_{a+\epsilon}^{b-\epsilon})}.
\end{multline*}
From \cite[C.5 Theorem 7]{evans2022partial} we know that $\gamma_\epsilon \rightarrow \gamma$ for all $(x,t) \in M\times[a,b] $ because $\gamma$ is continuous. 
Choosing $\gamma$ such that $\| f-\gamma\|_{L^1(\mu_{a}^{b})} <\delta$ and using the dominated convergence theorem we get $\limsup_{\epsilon\rightarrow 0} \| f-f_\epsilon\|_{L^1(\mu_{a+\epsilon}^{b-\epsilon})} \leq 2 \delta $. As $\delta$ can be chosen arbitrary small we obtain convergence.
\end{proof}

\begin{lemma}\label{lm:Smooth}
    We have $\eta_\epsilon \in C^1(M\times [0,T])$. 
\end{lemma}
\begin{proof}
Let $\gamma: V \subset \R^d  \rightarrow U(x)$ be a smooth local chart for an open set $U(x)$ containing  $x$. Then since $W(x,y)\in C^1(M\times M)$ the function  $z \mapsto \int_M \partial_{z_i} W(\gamma(z),y) \d\mu_\tau$ is continuous in $z$ and the product 
$$ (z,t) \mapsto \int_M \partial_{z_i} W(\gamma(z),y) \d\mu_\tau(y) g_\epsilon(t-\tau)$$ is continuous on $V\times \R$. Taking any sequence $(z_n,t_n) \rightarrow (z,t)$ we can use the dominated convergence theorem to obtain
\begin{align*}
    \lim_{n\rightarrow\infty}\partial_t \eta_\epsilon(\gamma(z_n),t_n)=\lim_{n\rightarrow\infty}\int_\R\int_M \partial_{z_i} W(\gamma(z_n),y)\d\mu_\tau(y) g_\epsilon(t_n-\tau) \d\tau\\
    =\int_\R\int_M \partial_{z_i} W(\gamma(z),y)\d\mu_\tau(y) g_\epsilon(t-\tau) \d\tau=\partial_t \eta_\epsilon(\gamma(z),t).
\end{align*}
An upper bound is given by the function $\sup_{V\times W} \partial_{z_i} W(\gamma(z),y) \ \chi_{[\inf_{n} t_n-\epsilon,\sup_n t_n+\epsilon]}(\tau)$. Thus, $\partial_t \eta_\epsilon(\gamma(z),t)$ is continuous in $V\times [0,T]$.
With the same argument, a similar statement can be shown for  
\begin{align*}
     \partial_t \eta_\epsilon(x,t)= \int_\R \int_M W(x,y)\d\mu_\tau(y) \partial_t g_\epsilon(t-\tau)(t) \d\tau
\end{align*}
By \cite[Theorem 2.8]{spivak2018calculus} it follows that $\eta_\epsilon(t,\gamma(z))\in C^1(V\times[0,T])$ and since the local chart was chosen arbitrarily $\eta_\epsilon \in C^1(M\times[0,T])$.
\end{proof}

%% file: Asphericalcoords.tex
\section{Spherical coordinates}\label{app:sphericalCoords}
For many computations in \cref{sec:stationary} we use spherical coordinates. Up to small notational changes, we use the definition provided in \cite{blumenson1960spherical}.
We define the coordinate transform $X_n: \phi \in [0,\pi]^{n-2} \times [0,2\pi] \rightarrow \S^{n-1}$ for $\phi \in [0,\pi]^{n-2} \times [0,2\pi]$ as
\begin{align*}
    X_n(\phi) = \cos(\phi_1)\, e_1 + \sum_{i = 2}^{n-1} \cos(\phi_i) \prod_{j = 1}^{i-1} \sin(\phi_j)\, e_i + \prod_{j = i}^{n-1} \sin(\phi_i)\, e_n .
\end{align*}
Here and in the following, $e_i\in \R^n$ denotes the $i$-th standard basis vector.

The Jacobian determinant is given by 
\begin{align*}
    JX_n(\phi) =  \prod_{i = 1}^{n-2} \sin^{n-1-i}(\phi_i).
\end{align*}
To highlight the recursive character of $X_n$ with respect to $n$, we further note that 
\begin{align*}
    X_n(\phi)_{\hat{1}} = \sin(\phi_1) \,X_{n-1}(\phi_{\hat{1}}) 
\quad \text{and} \quad
    JX_n(\phi) = \sin^{n-2}(\phi_1)JX_{n-1}(\phi_{\hat{1}})
\end{align*}
where the index $\hat{1}$ denotes that we drop the first element, i.e. for $\phi \in \R^{n-1}$, $\phi_{\hat{1}} = \sum_{i = 2}^{n-1} (\phi_i)\, e_{i-1}$. A practical consequence of this property is the recursive computation formula for the Hausdorff measure of the $n$-dimensional sphere.
\begin{lemma}\label{lem:recursiveSphere}
    Denote $|\S^{n-1}| := \mathcal{H}^n(\S^{n-1})$. For $n \geq 2$ it holds that 
    \begin{align*}
        |\S^{n-1}| = |\S^{n-2}| \int_0^\pi \sin^{n-2}\phi\,\d\phi.
    \end{align*}
\end{lemma}
\begin{proof}
    For $n = 2$, the proof follows from a simple computation and the fact that $|\S^0| = 2$ and $|\S^1| = 2\pi$.
    For $n > 2$, we have
    \begin{align*}
        |\S^{n-1}| &= \int_{[0,\pi]^{n-2} \times [0,2\pi]} JX_n(\phi)\, \d\phi = \int_0^{\pi} \sin^{n-2}\phi_1 \,JX_{n-1}(\phi_{\hat{1}})\,\d\phi\\ &= \int_0^\pi \sin^{n-2}\phi\,\d\phi\, \int_{[0,\pi]^{n-3} \times [0,2\pi]} JX_{n-1}(\psi)\, \d\psi = |\S^{n-2}| \int_0^\pi \sin^{n-2}\phi\,\d\phi,
    \end{align*}
    where we use the recursive property of the Jacobian determinant.
\end{proof}
\subsection{Definition using Givens rotations}
Spherical coordinates can equivalently be defined using Givens rotations (see e.g. \cite[Chapter 5.1.8]{golubmatrix2013}). A Givens rotation for an angle $\phi \in [0,2\pi)$ and indices $i,j \leq n$ with $i \neq j$ is determined by the rotation matrix $G(i,j,\phi) \in \R^{n\times n}$
\begin{align*}
    G(i,j,\phi)_{k,l} = \begin{cases} 
    \cos(\phi) & \text{if } k = l = i \text{ or } k = l = j,\\
    1 & \text{if } k = l \neq i \text{ and } k = l \neq j,\\
    \sin(\phi) & \text{if } k = i, l = j, \\
    -\sin(\phi) & \text{if } k = j, l = i, \\
    0 & \text{otherwise.}
    \end{cases}
\end{align*}
Applying $G(i,j,\phi)^T$ to a vector $x \in R^n$ corresponds to a counterclockwise rotation of $x$ by the angle $\phi$ in the $(i,j)$-plane.
For a given vector of angles $\phi \in [0,\pi]^{n-2}\times [0,2\pi]$ we can thus construct the matrix 
\begin{align}\label{eq:givens}\begin{aligned}
    R(\phi) = G(n-1,n,\phi_{n-1}) \circ \hdots &\circ G(2,3,\phi_2) \\&\circ G(1,2,\phi_1) \circ G(2,3,\phi_2)^T \circ\hdots \circ G(n-1,n, \phi_{n-1})^T.
\end{aligned}
\end{align}
The rotation matrix $R(\phi)$ can be written as a 2-dimensional rotation of angle $\phi_1$ in the $(e_1, X_{n-1}(\phi_{\hat{1}}))$-plane, as the following lemma shows.
\begin{lemma}
Let $R(\phi)$ be the rotation matrix as described in \eqref{eq:givens}. Then it holds that 
\begin{align*}
    R(\phi) = U\,G(1,2,\phi_1)\,U^T, 
\end{align*}
with $UU^T = \text{Id}$,  $U_{1,\cdot} = e_1$ and $U_{2,\cdot} = (
    0,\,
    X_{n-1}(\phi_{\hat{1}}))^T$. 
\end{lemma}
\begin{proof}
    For $n = 2$, the statement can be verified by inserting $U = \id$ and the definition of $R(\phi)$. For $n > 2$ we define
    \begin{align*}
        U = G(2,3,\phi_2) \circ\hdots \circ G(n-1,n, \phi_{n-1}).
    \end{align*}
    With this choice of $U$, $R(\phi)$ has the claimed form and $UU^T = \id$ due to the orthogonality of Givens matrices. It remains to show that the first two rows of $U$ fulfill $U_{1,\cdot} = e_1$ and $U_{2,\cdot} = (
    0,\,
    X_{n-1}(\phi_{\hat{1}}))^T$. For $n = 3$, $U$ reduces to
    \begin{align*}
        U = G(2,3,\phi_2) = \begin{pmatrix}
            1 & 0 & 0\\
            0 & \cos{\phi_2} & \sin{\phi_2}\\
            0 & -\sin{\phi_2} & \cos{\phi_2}
        \end{pmatrix},
    \end{align*}
    and clearly, $U_{1, \cdot}= e_1$ and $U_{2,\cdot} = (0,\cos{\phi_2},\sin{\phi_2})^T = (0,X_2(\phi_2))^T$. For $n > 3$, the proof follows from induction over $n$.
\end{proof}
\begin{cor}\label{cor:rotationsplit}
    Let $x = X_n(\phi)$, $\tilde{x} = (0,X_{n-1}(\phi_{\hat{1}}))$ then 
    \begin{align*}
        R(\phi)^Ty = y - (y\cdot e_1)\, e_1 - (y\cdot \tilde{x})\, \tilde{x}
        &+ (\cos(\phi_1)(y\cdot e_1) - \sin(\phi_1)  (y\cdot \tilde{x})) \,e_1 \\ &+
        (\sin(\phi_1)(y\cdot e_1) + \cos(\phi_1)  (y\cdot \tilde{x}))\, \tilde{x}.
    \end{align*}
    In particular, if $y\cdot e_1 = 0$ it holds that
    \begin{align*}
        R(\phi)^Ty = y & - (y\cdot \tilde{x})\, \tilde{x}
        +(y\cdot \tilde{x})\left( -\sin(\phi_1) \,e_1 + \cos(\phi_1)\, \tilde{x}\right).
    \end{align*}
\end{cor}
With the above results we obtain
\begin{align*}
    X_n(\phi) = R(\phi)^T e_1,
\end{align*}
and since Givens matrices are orthonormal it also holds that \begin{align*}
     R(\phi) X_n(\phi) =e_1.
\end{align*}
We can therefore as well consider rotated spherical coordinates 
\begin{align*}
    X_n^{\theta}(\phi) = R(\theta)^T X_n(\phi)
\end{align*}
for a reference point $x = X_n(\theta)$, with the same Jacobian determinant as before, i.e.  $JX^{\theta}_n(\phi)  = JX_n(\phi)$.

%% file: AStationary.tex
\section{Proofs for \cref{sec:stationary}}\label{app:stationary}
\subsection{Proof of \cref{lem:stationaryuniform}}\label{app:stationuni}
\begin{namedthm*}{Lemma 4.7 (cont.)}[]
Let $n > 2$. The uniform distribution $\mu = \frac{1}{|\mathcal{S}^{n-1}|}\mathcal{H}^n$ is a stationary point of ${\cal E}$ if and only if all eigenvalues $\{\lambda_i\}_{i = 1}^n$ of $D$ have the same absolute value, i.e. $|\lambda_i| = \lambda$ for some $\lambda \in \R$.
\end{namedthm*}
\begin{proof}
The proof for $n > 2$ uses the same arguments as for $n = 2$, however, the rotation corresponding to a translation of the angle in two dimensions is technically more complicated. We use the notation and techniques from \cref{app:sphericalCoords} (spherical coordinates $X_n$ and rotations $R$).

Again, we first fix $x\in \S^{n-1}$ and consider the integral
\begin{align*}
   \int_{\S^{n-1}} e^{x\cdot Dy} P^\perp_x Dy \,\d\mathcal{H}^{n}(y) = (*).
\end{align*}
Similarly to the two-dimensional case, we choose $\phi \in [0,\pi]^{n-2} \times [0,2\pi]$ such that 
\begin{align*}
    X_n(\phi) = \frac{Dx}{\|Dx\|},
\end{align*}
and therefore also
\begin{align*}
    R(\phi)Dx = \|Dx\| R(\phi)R(\phi)^Te_1 = \|Dx\|e_1,
\end{align*}
where $e_1 = (1,0,\dots,0)^T \in \R^n$ denotes the first standard basis vector.
We rewrite the integral using  rotated spherical coordinates and substitute it into the above identity to obtain
\begin{align*}
     (*)= &\int_{[0,\pi]^{n-2} \times [0,2\pi]} e^{x\cdot DR(\phi)^TX_n(\theta)}P^\perp_xDR(\phi)^TX_n(\theta)\, JX_n(\theta)\,\d\theta \\ = &\int_{[0,\pi]^{n-2} \times [0,2\pi]} e^{\|Dx\|\,\cos(\theta_1)} (DR(\phi)^TX_n(\theta) - \|Dx\|\,\cos(\theta_1)\,x)\, JX_n(\theta)\,\d\theta,
\end{align*}
where $JX_n$ denotes the Jacobian determinant of $X_n$.
To reduce the above integral over the vector $\theta$ to an integral over only the first component $\theta_1$, we write
\begin{align*}
    R(\phi)^TX_n(\theta) &= \cos(\theta_1)\,R(\phi)^Te_1 + R(\phi)^T\begin{pmatrix}
        0\\ X_{n}(\theta)_{\hat{1}}
    \end{pmatrix}\\ &= \cos(\theta_1) \frac{Dx}{\|Dx\|} + \sin(\theta_1)\, R(\phi)^T\begin{pmatrix}
        0\\ X_{n-1}(\theta_{\hat{1}})
    \end{pmatrix},
\end{align*}
where the subscript $\hat{1}$ denotes that we neglect the first component. 
Inserting this into $(*)$ we get
\begin{align*}
    (*) &= \left(D^2x/\|Dx\| - \|Dx\|x\right) \int_0^{\pi} e^{\|Dx\|\,\cos(\theta_1)} \cos(\theta_1) \sin^{n-2}(\theta_1) \,\d\theta_1\\
    &+ \int_0^\pi e^{\|Dx\|\,\cos(\theta_1)} \sin^{n-1}(\theta_1) \,DR(\theta)^T\underbrace{\int_{S^{n-2}} \begin{pmatrix}
        0\\ z
    \end{pmatrix} \d\mathcal{H}^{n-1}(z)}_{=0} \,\d\theta_1\\
    & = C(n, \|Dx\|) \left(D^2x/\|Dx\| - \|Dx\|x\right)
\end{align*}
and due to the symmetry of sine and cosine we have that $C(n,\|Dx\|) > 0$ for any $n\geq2$, $\|Dx\| > 0$. We can thus deduce that $(*) = 0$ if and only if $x$ is an eigenvector of $D^2$, exactly as in the case $n = 2$. This holds true for $\mu$-almost all $x \in \mathcal{S}^{n-1}$ if and only if all eigenvalues of $D$ have the same absolute value, which then automatically yields $\d\erg{}(\mu, V) = 0$. 

Again, it remains to show that this is also  necessary. Without loss of generality, we assume $|\lambda_1|> |\lambda_2|$ and $\lambda_1$ and $\lambda_2$ to be the eigenvalues of largest, respectively second largest, absolute value corresponding to the eigenvectors $z_1$, respectively $z_2$. 

From here, the strategy is the exact same as in the two-dimensional case, which we restate here for completeness.  
The factor $\left(D^2x/\|Dx\| - \|Dx\|x\right) \cdot z_2$ is strictly negative on the set
\begin{align*}
    A = \{x \in \S^{n-1} \, | \, (x\cdot z_1) \in (|\lambda_2/\lambda_1|, 1),\, (x\cdot z_2) > 0\}.
\end{align*}
Since $\mu(A) > 0$ we can find a Lipschitz continuous $V$ such that $V \cdot z_1 = 0$ for $\mu$-a.e. on $\mathcal{S}^{n-1}$ and
\begin{align*}
    V(x)\cdot z_2 \begin{cases}
        > 0 &\text{for a.e. } x \in A, \\
        = 0 &\text{for a.e. } x \in \mathcal{S}^{n-1}\backslash A.
    \end{cases}
\end{align*}
For all such $V$ it holds that $\d\erg{}(\mu, V) > 0$, which concludes the proof.
\end{proof}

\subsection{Proof of \cref{lem:technicalx}}\label{app:stationtechx}
\begin{namedthm*}{Lemma 4.8 (cont.)}
Let $n \geq 2$, and $\mu_0 = \frac{1}{|\mathcal{S}^{n-1}|} \mathcal{H}^n$. Then it holds that
    \begin{align}
        \int_{\mathcal S^{n-1}} e^{x\cdot y} x \,\d\mu_0(x) = C_1\,y
    \end{align}
    for any $y \in \mathcal{S}^{n-1}$, where the constant $C_1$ is positive and depends only on the dimension $n$.
\end{namedthm*}
\begin{proof}
The proof for $n > 2$ goes along the lines of the proof for $n = 2$. However, the rotation corresponding to a translation of the angle in two dimensions technically more complicated in higher dimensions. For an introduction to rotated spherical coordinates used in this proof we refer the reader to \cref{app:sphericalCoords}.

We first fix $y \in \S^{n-1}$ and choose $\theta \in \R^{n-1}$ such that $y = X_n(\theta)$. We proceed to write the integral using rotated spherical coordinates $x = X^{\theta}_n(\phi)$ and obtain
   \begin{align*}
       &\int_{\mathcal S^{n-1}} e^{x\cdot y} x_j \,\d\mu_0(x) =\frac{1}{|\mathcal{ S}^{n-1}|} \int_{[0,\pi]^{n-2} \times [0,2\pi]} e^{X_n^\theta(\phi)\cdot X_n(\theta)} \left(X_n^\theta(\phi)\right)_i\, JX_n(\phi)\,\d\phi = (*).
   \end{align*}
Substituting the expressions for $X_n$ and $X_n^\theta$ yields 
   \begin{align*}
       X^\theta_n(\phi) \cdot X_n(\theta) &= R(\theta)^TX_n(\phi) \cdot X_n(\theta) =  X_n(\phi) \cdot R(\theta)X_n(\theta) = X_n(\phi) \cdot e_1 = \cos(\phi_1).
   \end{align*}
   Additionally, we note that we can write any $x = x_1e_1 + (0,x_2,\hdots,x_n)^T$ and see that
   \begin{align*}
    X_n^\theta(\phi) = R(\theta)^T X_n(\phi)
    = R(\theta)^T \cos(\phi_1) e_1 + R(\theta)^T \begin{pmatrix}0\\ X_{n}(\phi)_{\hat{1}}\end{pmatrix}&\\ = \cos(\phi_1) \, y + \sin(\phi_1) R(\theta)^T\begin{pmatrix}0\\ X_{n-1}(\phi_{\hat{1}})\end{pmatrix}&,
   \end{align*}
   where $e_1 = (1,0,\dots,0)^T \in \R^n$ denotes the first standard basis vector.
   Substituting the above equality into the integral we derive
   \begin{align*}
        (*) =& \frac{1}{|\mathcal{ S}^{n-1}|} \int_{[0,\pi]^{n-2}\times [0,2\pi]} 
        e^{\cos(\phi_1)} \left[\cos(\phi_1) y + \sin(\phi_1) R(\theta)^T\begin{pmatrix}0\\ X_{n-1}(\phi_{\hat{1}})\end{pmatrix}\right]_j \,JX_n(\phi)\,\d\phi \\
        =&y_i\, \frac{|\mathcal{ S}^{n-2}|}{|\mathcal{ S}^{n-1}|} \int_{0}^\pi 
        e^{\cos\phi} \cos\phi \sin^{n-2}\phi\,\d\phi\\ &+ \frac{1}{|\mathcal{ S}^{n-1}|} \int_0^\pi e^{\cos\phi} \sin^{n-1}\phi \left[R(\theta)^T \underbrace{\int_{\mathcal{S}^{n-2}} \begin{pmatrix}
            0 \\ z
        \end{pmatrix}  \,\d \mathcal{H}^{n-1}(z)}_{ = 0}\right]_j \, \d\phi.
   \end{align*}
   The proof now follows from choosing the constant
   \begin{align*}
       C_1 = \frac{|\mathcal{ S}^{n-2}|}{|\mathcal{ S}^{n-1}|} \int_{0}^\pi 
        e^{\cos\phi} \cos\phi \sin^{n-2}\phi\,\d\phi = \frac{|\mathcal{ S}^{n-2}|}{|\mathcal{ S}^{n-1}|} \int_0^{\pi/2} \sin^{n-2}\phi\cos\phi\sinh({\cos\phi})\,\d\phi,
   \end{align*}
   which is positive for all $n \geq 2$ since the function $t \mapsto t \sinh{t}$ is positive for $t > 0$ and
   both sine and cosine are positive for $\phi \in (0,\pi/2)$.
\end{proof}
\subsection{Proof of \cref{lem:technicalxx}}\label{app:stationtechxx}
\begin{namedthm*}{Lemma 4.9 (cont.)}
     Let $n \geq 2$, and $\mu_0 = \frac{1}{|\mathcal{S}^{n-1}|} \mathcal{H}^n$. Then for all $y \in \mathcal{S}^{n-1}$ and  $1 \leq i\leq n$ it holds that
    \begin{align}
        \int_{\mathcal S^{n-1}} e^{x\cdot y} x_i^2 \,\d\mu_0(x) = C_2\,y_i^2+ C_3,
    \end{align}
    where the constants $C_2$ and $C_3$ are positive and depend only on the dimension $n$.
\end{namedthm*}
\begin{proof}
    Using the same arguments as in the previous proof, we obtain 
    \begin{align}\label{eq:rotsquareint}\int_{0}^\pi e^{x\cdot y} x^2_j \,\d\mu_0(x) = \nonumber  y^2_j \,&\frac{|\mathcal{ S}^{n-2}|}{|\mathcal{ S}^{n-1}|}\int_{\S^{n-1}} e^{\cos\phi}\cos^2\phi\sin^{n-2}\phi\,\d\phi 
\\   + &\frac{1}{|\mathcal{ S}^{n-1}|} \int_0^\pi e^{\cos\phi} \sin^n\phi \int_{\mathcal{S}^{n-2}} \left[R(\theta)^T \begin{pmatrix}
            0 \\ z
        \end{pmatrix}\right]^2_j   \,\d \mathcal{H}^{n-1}(z)\,\d\phi,
    \end{align}
    where the mixed term containing $x_iy_i$ vanishes due to symmetry.
    Since the second term still depends on $y$ due to the rotation, we write $\tilde{y} = (0, X_{n-1}(\theta_{\hat{1}}))$ and decompose $\tilde{z} = (0, z)$ into its rotation-invariant and rotation-variant part. More precisely, we use \cref{cor:rotationsplit} to get
    \begin{align*}
        R(\theta)^T \tilde{z} =  \tilde{z} - (\tilde{y}\cdot\tilde{z})\tilde{y} + (\tilde{y}\cdot\tilde{z})\left[-\sin(\theta_1)\,e_1 +\cos(\theta_1)\,\tilde{y}\right]
    \end{align*}
    and thus 
    \begin{align*}
         \left[R(\theta)^T \tilde{z}\right]^2 = (\tilde{z}-(\tilde{y}\cdot\tilde{z})\tilde{y})^2 + (\tilde{y}\cdot\tilde{z})^2(\sin^2(\theta_1)\,e_1 + \cos^2(\theta_1)\,\tilde{y}^2) + 2 \cos(\theta_1)(\tilde{z}-(\tilde{y}\cdot\tilde{z})\,\tilde{y})(\tilde{y}\cdot \tilde{z})\,\tilde{y}.
    \end{align*}
    Making use of the trigonometric identity $\cos^2(\theta_1) + \sin^2(\theta_1) = 1$ we get
    \begin{align}\label{eq:rotsquare}
         \left[R(\theta)^T \tilde{z}\right]^2 &= \tilde{z}^2 + (\tilde{y}\cdot\tilde{z})^2\,e_1 + 2(\tilde{y}\cdot\tilde{z})^2\,\tilde{y}^2 - 2 (\tilde{y}\cdot\tilde{z})\,\tilde{z}\tilde{y}  + 2 \cos(\theta_1)(\tilde{z}-(\tilde{y}\cdot\tilde{z})\,\tilde{y})(\tilde{y}\cdot \tilde{z})\,\tilde{y}\nonumber \\ &
         - (\tilde{y}\cdot\tilde{z})^2\,(\cos^2(\theta_1)\,e_1 + \sin^2(\theta_1)\,\tilde{y}^2) \nonumber \\
         &= \tilde{z}^2 + (\tilde{y}\cdot\tilde{z})^2\,\left(e_1 - y^2\right)
          + 2 (\cos(\theta_1)-1)(\tilde{z}-(\tilde{y}\cdot\tilde{z})\,\tilde{y})(\tilde{y}\cdot \tilde{z})\,\tilde{y},
    \end{align}
    where in the last step we use the fact that $y^2 = \cos^2\theta_1 e_1 + \sin^2\theta_1\tilde{y}^2$.
    To prove that the integral over \eqref{eq:rotsquare} can be written as claimed, we observe that for all $j = 2,\hdots,n$
    \begin{align*}
       \int_{\S^{n-2}}\tilde{z}_j^2 \,\d\mathcal{H}^{n-1}(z) 
       =: \tilde{C},
    \end{align*}
    where $\tilde{C}$ is positive and depends only on $n$,
    and therefore also $
         \int_{\S^{n-2}}(\dpr{\tilde{z}}{\tilde{y}})^2\,\d\mathcal{H}^{n-1}(z)  = \tilde{C}\,\|\tilde{y}\|^2 = \tilde{C}$. With this we derive that
         \begin{align}\label{eq:firstpart}
             \left[\int_{\S^{n-2}} \tilde{z}^2 + (\tilde{y}\cdot\tilde{z})^2\,\left(e_1 - y^2\right) \,\d \mathcal{H}^{n-1}(z)\right]_j = \tilde{C}\left(1 - y_j^2\right)
         \end{align}
         for all $j = 1,\hdots,n$ and it remains to show that for any $1\leq j \leq n$
    \begin{align}\label{eq:intzero}
       \int_{\mathcal{S}^{n-2}} [(\tilde{z}-(\tilde{y}\cdot\tilde{z})\,\tilde{y})(\tilde{y}\cdot \tilde{z})\,\tilde{y}]_j\,\d\mathcal{H}^{n-1}(z) = 0
    \end{align}
    The case $j=1$ is trivial as $\tilde{y}_1=\tilde{z}_1 = 0$. For $2\leq j \leq n$, we write out the integrand and obtain
    \begin{align*}
        &[(\tilde{z}-(\tilde{y}\cdot\tilde{z})\,\tilde{y})(\tilde{y}\cdot \tilde{z})\,\tilde{y}]_j = \left(\sum_{k = 1}^n \tilde{z}_k\tilde{y}_k\right)\,\tilde{z}_j\tilde{y}_j - \left(\sum_{k,l = 1}^n \tilde{z}_k\tilde{z}_l\tilde{y}_k\tilde{y}_l\right)\tilde{y}_j^2\\
        =&  \left(\sum_{k = 1, k \neq j}^n \tilde{z}_j\tilde{z}_k\tilde{y}_j\tilde{y}_k\right)\, - \left(\sum_{k,l = 1, k\neq l}^n \tilde{z}_k\tilde{z}_l\tilde{y}_k\tilde{y}_l\right)\tilde{y}_j^2 + \left( \tilde{z}^2_j- (\dpr{\tilde{z}}{\tilde{y}})^2\right)\tilde{y}_j^2,
    \end{align*}
    where we can use the same argument as for \eqref{eq:firstpart} to show that the last summand integrates to zero. Since also $\int_{\mathcal{S}^{n-2}} \tilde{z}_j\tilde{z}_k\d\mathcal{H}^{n-1}(z) = 0$ for any $j \neq k$ we derive \eqref{eq:intzero}. Together with \eqref{eq:firstpart} and \eqref{eq:rotsquare} this yields
    \begin{align*}
         \int_{\S^{n-2}} \left[R(\theta)^T \tilde{z}\right]_j^2\,\d\mathcal{H}^{n-1}(z) &= \tilde{C}\left(1-y_j^2\right).
    \end{align*}
    The statement now follows from substituting the above into \eqref{eq:rotsquareint}, with constants given by
    \begin{align*}
        &C_2 = \frac{|\mathcal{ S}^{n-2}|}{|\mathcal{ S}^{n-1}|}\int_0^\pi e^{\cos\phi}\cos^2\phi\sin^{n-2}\phi\,\d\phi- C_3, \qquad 
        C_3 = \frac{\tilde{C}}{|\mathcal{ S}^{n-1}|} \int_0^\pi e^{\cos\phi}\sin^n\phi\,\d\phi. 
    \end{align*}


Since $\Tilde{C} > 0$ for all $n \geq 2$, it directly follows that $C_3 > 0$. To show that $C_2 > 0$ for all $n \geq 2$ we first show that $\Tilde{C} = |\S^{n-2}|/(n-1)$. For $n = 2$ this follows directly from $\Tilde{C} = |\S^0| = 2$. For $n > 2$ we have
\begin{align*}
    \tilde{C} = |\S^{n-3}| \int_0^{\pi} \cos^2\phi \sin^{n-3}\phi \,\d\phi = |\S^{n-3}| \int_0^{\pi} \sin^{n-3}\phi - \sin^{n-1}\phi\, \d \phi.
\end{align*}
Using integration by parts we further derive that
\begin{align*}
     \int_0^{\pi}\sin^{n-1}\phi\, \d \phi = (n-2)/(n-1) \sin^{n-3}\phi\, \d \phi.
\end{align*}
As shown in \cref{lem:recursiveSphere}, the recursive form of the Jacobian determinant of spherical coordinates yields that
\begin{align*}
    |\S^{n-2}| = |\S^{n-3}|\,\int_0^{\pi} \sin^{n-3}\phi \,\d\phi.
\end{align*}
Combining these equalities, we see that
\begin{align*}
    \Tilde{C} = (1- ((n-2)/(n-1))) |\S^{n-2}| = |\S^{n-2}|/(n-1),
\end{align*}
and therefore with integration by parts we get
\begin{align*}
    C_2 &= \frac{|\mathcal{ S}^{n-2}|}{|\mathcal{ S}^{n-1}|}\int_0^\pi e^{\cos\phi}\left[\cos^2\phi\sin^{n-2}\phi-\frac{1}{1-n}\sin^n\phi\right]\,\d\phi \\ &= \frac{|\mathcal{ S}^{n-2}|}{|\mathcal{ S}^{n-1}|}\int_0^\pi e^{\cos\phi}\sin^{n-2}\phi\cos
    \phi\left(\cos\phi-1\right)\,\d\phi.
\end{align*}
Due to the symmetry of sine and cosine we get 
\begin{align*}
    C_2 &= \frac{|\mathcal{ S}^{n-2}|}{|\mathcal{ S}^{n-1}|} \int_0^{\pi/2} e^{\cos\phi}\sin^{n-2}\phi\cos
    \phi\left(\cos\phi-1\right) + e^{-\cos\phi}\sin^{n-2}\phi\cos
    \phi\left(\cos\phi+1\right)\,\d\phi \\
    &= 2\frac{|\mathcal{ S}^{n-2}|}{|\mathcal{ S}^{n-1}|} \int_0^{\pi/2} \sin^{n-2}\phi \left( \cos(\phi)\cosh(\cos(\phi)) - \sinh(\cos(\phi))\right) > 0,
\end{align*}
where the positivity follows from the fact that the function $t \mapsto t \cosh(t) - \sinh(t)$ is positive for $t > 0$ and both sine and cosine are positive for $\phi \in (0,\pi/2)$. 
    \end{proof}